\documentclass[12pt, twoside]{article}   	
\usepackage[utf8x]{inputenc} 
\usepackage{geometry}               		
\usepackage{graphicx}				
			
\usepackage{xcolor}								
\usepackage{amssymb}

\usepackage{amsmath}
\usepackage{breqn}
\usepackage{mathtools}
\usepackage{color}
\usepackage{mathrsfs}
\usepackage{amsthm}
\usepackage{amsmath, amssymb, amsthm, enumerate,graphicx,textcomp,fullpage, cite}
\usepackage{tikz}
\usepackage{esvect}
\definecolor{red}{rgb}{1,0,0}
\usepackage{fancyhdr} 
\fancypagestyle{mypagestyle}{%
  \fancyhf{}
  \fancyhead[OC]{Francisc Bozgan}
  \fancyhead[EC]{Local wellposedness of the mKP-I equations in periodic setting}
  \fancyfoot[C]{\thepage}%

}
\newtheorem{theorem}{Theorem}[section]
\newtheorem*{corollary}{Corollary}
\newtheorem{lemma}{Lemma}
\newtheorem*{remark}{Remark}
\newtheorem{definition}{Definition}[section]
\newtheorem{proposition}{Proposition}

\usepackage{hyperref}

\title{ \Large \textbf{Local wellposedness of the modified KP-I equations in periodic setting with small initial data}}
\author{Francisc Bozgan}

\date{ }							

\begin{document}

\maketitle  
  \begin{abstract} 
We prove local well-posedness of partially periodic and periodic modified KP-I equations, namely for $\partial_t u+(-1)^{\frac{l+1}{2}}\partial^l_x u-\partial_x^{-1}\partial_y^2 u+u^2\partial_x u=0$ in the anisotropic Sobolev space $H^{s,s}(\mathbb{R}\times \mathbb{T})$ if $l=3$ and $s>2$, in $H^{s,s}(\mathbb{T}\times \mathbb{T})$ if $l=3$ and $s>\frac{19}{8}$, and in   $H^{s,s}(\mathbb{R}\times \mathbb{T})$ if $l=5$ and $s>\frac{5}{2}$. All three results require the initial data to be small. 
\end{abstract}

\tableofcontents
\newpage

 \section{Introduction} 

Let $\mathbb{T}=\mathbb{R}/(2\pi \mathbb{Z}).$ This paper is dedicated to the study of the partially periodic and periodic equations of the hierarchy of modified Kadomtsev-Petviashvili I equations (mKP-I) 
\begin{equation}\label{eqmkp} 
\begin{cases} \partial_t u+(-1)^{\frac{l+1}{2}}\partial^l_x u-\partial_x^{-1}\partial_y^2 u+u^2\partial_x u=0,\\
u(0,x,y)=u_0(x,y)
\end{cases}
\end{equation} 
in the anisotropic Sobolev spaces $H^{s_1,s_2}(M\times \mathbb{T})$, where, if $l=3$, $M$ is either $\mathbb{R}$ or $\mathbb{T}$ and if $l=5$, $M=\mathbb{R}.$

The third order ($l=3$) mKP-I equation (\ref{eqmkp}) is the modified version of the third order KP-I equation 
\begin{equation}\label{eqkp3} 
\partial_t u+\partial^3_x u-\partial_x^{-1}\partial_y^2 u+u\partial_x u=0 
\end{equation} 
which appears when modeling certain long dispersive waves with weak transverse effects, as we see in \cite{ablowitz}, \cite{kadomtsev}. The modified KP equation appear in \cite{turitsyn} which describe the evolution of sound waves in antiferromagnetics. These KP-I equations, as well as KP-II equations in which the sign of the term $\partial_x^{-1}\partial_y^2 u$ in (\ref{eqkp3}) is $+$ instead of $-$ appear in several physical contexts. Here the operator $\partial_x^{-1}$ is defined via Fourier transform, $\widehat{\partial_x^{-1}f}(\xi,\eta)=\frac{1}{i\xi}\widehat{f}(\xi,\eta).$

The third order KP equations are well studied. The KP-II equations are much more well understood for the third order, mainly to the $X_b^s$ method of Bourgain \cite{bourgain1}. The third order KP-II initial value problem is globally wellposed in $L^2$ on both $\mathbb{R}\times \mathbb{R}$ and $\mathbb{T}\times \mathbb{T}$, see \cite{bourgain1}. On $\mathbb{R}^2$, Takaoka and Tzvetkov \cite{takaoka} and Isaza and Mej{\'i}a \cite{isaza} pushed the low regularity local well-posedness theory down to the anistropic Sobolev space $s_1 > -\frac{1}{3},s_2 \geq 0$. Hardac \cite{hardac1} and Hardac, Herr and Koch in \cite{hardac1} and \cite{hardac2} reached the threshold $s_1\geq -\frac{1}{2},s_2\geq 0$ which is the scaling critical regularity for the KP-II equation. As for the initial value problem on $\mathbb{R}\times\mathbb{T},$ in order to study the stability of the KdV soliton under the flow of the KP-II equation, Molinet, Saut and Tzvetkov \cite{molinet1} proved global well-posedness on $L^2(\mathbb{R}\times\mathbb{T}).$

In the case of the KP-I initial value problem, as the Picard iterative methods in the standard Sobolev spaces, since the flow map fails to be $C^2$ at the origin in these spaces, as Molinet, Saut and Tzevtkov showed in \cite{molinet1}. Due to this fact, the wellposedness theory is more limited. For the third order KP-I equation, Kenig in \cite{kenig1} showed global well-posedness in the second energy space $Z^2_{(3)}=\{\phi \in L^2(\mathbb{R}^2): \|(1+\xi^2+\frac{\eta^2}{\xi^2})\widehat{\phi}(\xi,\eta)\|_{L^2_{\xi\eta}}<\infty\}$ and later, in \cite{ionescu2}, Ionescu, Kenig and Tataru showed global well-posedness in the first energy space $Z^1_{(3)}=\{\phi \in L^2(\mathbb{R}^2): \|(1+\xi+\frac{\eta}{\xi})\widehat{\phi}(\xi,\eta)\|_{L^2_{\xi\eta}}<\infty\}.$ Guo, Peng and Wang in \cite{guopeng} showed local well-posedness in $H^{1,0}(\mathbb{R}\times\mathbb{R}).$  For the third order modified KP-I and KP-II equation, Saut \cite{saut} showed that the generalized KP-I/KP-II equation 
\begin{equation}\label{eqgkp} 
\begin{cases} \partial_t u+\partial^3_x u+\epsilon\partial_x^{-1}\partial_y^2 u+u^p\partial_x u=0,\\
u(0,x,y)=u_0(x,y)
\end{cases}
\end{equation} ($\epsilon=\pm1$) is locally well-posed in $C([-T,T];H^s(\mathbb{R}^2))\cap C^1([-T,T];H^{s-3}(\mathbb{R}^2))$ for $s\geq 3$ with the momentum $V(u)(t)=\int_{\mathbb{R}^2} u^2(t)dxdy$ and energy $$E(u)(t)=\int_{\mathbb{R}^2}\frac{(\partial_x u)^2}{2}-\epsilon\frac{(\partial_x^{-1}\partial_y u)^2}{2}-\frac{u^{p+2}}{(p+1)(p+2)}dxdy$$ being conserved quantities. Several blow-up result were found as well for the mKP-I equations. In \cite{saut}, if $p\geq 4$, the corresponding solution $u$ in (\ref{eqgkp}) blows up in finite time, i.e. there exists $\infty>T>0$ such that $\lim_{t\rightarrow T^{-}}\|\partial_y u(\cdot,y)\|_{L^2}=+\infty.$  Liu \cite{liu} improved the blow-up result for  $\frac{4}{3}\leq p <4$, also by showing that $\lim_{t\rightarrow T^{-}}\|\partial_y u(\cdot,y)\|_{L^2}=+\infty.$ Both proofs are based on some virial-type identities.

On $\mathbb{R}\times\mathbb{T}$, Ionescu and Kenig \cite{ionescu1} showed global well-posedness in the second energy space, i.e. $Z^2_{(3)}=\{\phi \in L^2: \|(1+\xi^2+\frac{n^2}{\xi^2})\widehat{\phi}(\xi,n)\|_{L^2_{\xi,n}}<\infty\}$ and Robert \cite{robert} proved global well-posedness in the first energy space  $Z^1_{(3)}=\{\phi \in L^2: \|(1+\xi+\frac{n}{\xi})\widehat{\phi}(\xi, n)\|_{L^2_{\xi, n}}<\infty\}.$

For our case of the third order partially periodic modified KP-I we prove the following theorem 

\begin{theorem}\label{maintheorem3RT} 
Assume $\phi \in H^{s,s}(\mathbb{R}\times\mathbb{T})$ with $s>2$. Then the initial value problem 
\begin{equation}\label{eqmkp3RT}
 \begin{cases}
  \partial_t u+\partial^3_x u-\partial_x^{-1}\partial_y^2 u + u^2\partial_x u=0,\\
u(0,x,y)=\phi(x,y) 
 \end{cases} 
 \end{equation}
  admits a unique solution in $C([-T,T]:H^{s,s}(\mathbb{R}\times\mathbb{T})$ with $T=T(\|\phi\|_{H^{s,s}})$ with $u,\partial_x u, \partial_y u \in {L^2_TL^{\infty}_{xy}}$ if $\|\phi\|_{H^{s,s}}$ is sufficiently small. Moreover, the mapping $\phi \rightarrow u$ is continuous from $H^{s,s}(\mathbb{R}\times\mathbb{T})$ to $C([-T,T];H^{s,s}(\mathbb{R}\times\mathbb{T}))$. 

\end{theorem} 

For the $\mathbb{T}\times\mathbb{T}$, Ionescu and Kenig showed in \cite{ionescu1} showed global well-posedness in the second energy space $Z^2_{(3)}$ as defined above. Zhang showed in \cite{zhang} that the third-order periodic KP-I equation is locally well-posed in a Besov type space, namely $B^1_{2,1}(\mathbb{T}^2)=\{\phi: \mathbb{T}^2\rightarrow \mathbb{R}: \widehat{\phi}(0,n)=0 \mbox{ for all } n \in \mathbb{Z}\setminus\{0\} \mbox{ and } \|\phi\|_{B^1_{2,1}}=\sum_{k=0}^{\infty}2^k\|1_{[2^{k-1},2^{k+1}]}(m)\widehat{\phi}\left(1+\frac{|n|}{|m|(1+|m|)}\right)\|_{l^2_{m,n}}\}$. 

For our case of the third order periodic modified KP-I we prove the following theorem 

\begin{theorem}\label{maintheorem3TT} 
Assume $\phi \in H^{s,s}(\mathbb{T}\times\mathbb{T})$ with $s>\frac{19}{8}$. Then the initial value problem 
\begin{equation}\label{eqmkp3TT}
 \begin{cases}
  \partial_t u+\partial^3_x u-\partial_x^{-1}\partial_y^2 u + u^2\partial_x u=0,\\
u(0,x,y)=\phi(x,y) 
 \end{cases} 
\end{equation} 
admits a unique solution in $C([-T,T]:H^{s,s}(\mathbb{T}\times\mathbb{T})$ with $T=T(\|\phi\|_{H^{s,s}})$ with $u,\partial_x u, \partial_y u \in {L^2_TL^{\infty}_{xy}}$ if $\|\phi\|_{H^{s,s}}$ is sufficiently small. Moreover, the mapping $\phi \rightarrow u$ is continuous from $H^{s,s}(\mathbb{T}\times\mathbb{T})$ to $C([-T,T];H^{s,s}(\mathbb{T}\times\mathbb{T}))$. 

\end{theorem} 

The fifth order ($l=5$) mKP-I equation (\ref{eqmkp}) is the modified version of the fifth order KP equation 
\begin{equation}\label{eqkp5} 
\partial_t u-\partial^5_x u-\partial_x^{-1}\partial_y^2 u+u\partial_x u=0 
\end{equation} 
which appears when modeling certain long dispersive waves with weak transverse effects, as we see in \cite{abramyan},  \cite{karpman}. The fifth order equation is part of a higher hierarchy of the third order KP equation (\ref{eqkp3}). 
By the work of Saut and Tzvetkov in \cite{saut1}, we know that the fifth order KP-II initial value problem is globally wellposed in $L^2$ on both $\mathbb{R}\times \mathbb{R}$ and $\mathbb{T}\times \mathbb{T}$. In \cite{lili}, local wellposedness in $H^{s,0}(\mathbb{R}\times\mathbb{T})$ for $s>-\frac{3}{4}$ and global wellposedness in $L^2.$ 

The fifth order KP-I initial value problem is known to be globally well-posed in the energy spaces $Z^1_{(5)}=\{\phi \in L^2: \|(1+\xi^2+\frac{\eta}{\xi})\widehat{\phi}(\xi,\eta)\|_{L^2_{\xi\eta}}<\infty\}$ on both $\mathbb{R}\times\mathbb{R}$ and $\mathbb{T}\times\mathbb{R}$ from the work of Saut and Tzetkov in \cite{saut1} and \cite{saut2}, using Picard iterative methods (se also \cite{chen1}).  Using the Fourier restriction norm method and sufficiently exploiting the geometric structure of the resonant set of \ref{eqkp5} to deal with the high-high frequency interaction, Li and Xiao established in \cite{li1}  the global well-posedness in $L^2(\mathbb{R}^2).$ Guo et al. \cite{guo1} established the local-wellposedness of the Cauchy problem in $H^{s,0}(\mathbb{R}\times\mathbb{R})$ for $s\geq -\frac{3}{4}$, Yan et al \cite{yan} showed global well-posedness  in $H^{s,0}(\mathbb{R}\times\mathbb{R})$ for $s>-\frac{6}{23}$ and finally Li et al. \cite{li2} proved global-wellposedness in $H^{s,0}(\mathbb{R}\times\mathbb{R})$ for $s> -\frac{4}{7}$ and local well-posedness for $s>-\frac{9}{8}.$ We conclude with the result from \cite{ionescu1} which proves global well-posedness on $\mathbb{R}\times\mathbb{T}$ and from \cite{robert1} which proves global well-posedness on $\mathbb{R}\times\mathbb{T}$, both results in $Z^1_{(5)}(\mathbb{R}\times\mathbb{T})$, resepctively  $Z^1_{(5)}(\mathbb{T}\times\mathbb{T})$, the natural energy spaces in these cases. For the fifth order modified KP-I, Esfahani \cite{esfahani} showed that the generalized KP-I equation 
\begin{equation}\label{eqgkp5} 
\begin{cases} \partial_t u-\partial^5_x u-\partial_x^{-1}\partial_y^2 u+u^p\partial_x u=0,\\
u(0,x,y)=u_0(x,y)
\end{cases}
\end{equation} is locally well-posed in $C([-T,T];H^s(\mathbb{R}^2))\cap C^1([-T,T];H^{s-5}(\mathbb{R}))$ for $s\geq 5$. In the same paper, if $p\geq 4$, the corresponding solution $u$ in (\ref{eqgkp5}) blows up in finite time, i.e. there exists $\infty>T>0$ such that $\lim_{t\rightarrow T^{-}}\|\partial_y u(\cdot,y)\|_{L^2}=+\infty.$
For our case of the fifth order partially periodic modified KP-I we prove the following theorem 

\begin{theorem}\label{maintheorem5RT} 
Assume $\phi \in H^{s,s}(\mathbb{R}\times\mathbb{T})$ with $s>\frac{5}{2}$. Then the initial value problem 
\begin{equation}\label{eqmkp5RT}
 \begin{cases}
  \partial_t u-\partial^5_x u-\partial_x^{-1}\partial_y^2 u + u^2\partial_x u=0,\\
u(0,x,y)=\phi(x,y) 
 \end{cases} 
 \end{equation} 
  admits a unique solution in $C([-T,T]:H^{s,s}(\mathbb{R}\times\mathbb{T})$ with $T=T(\|\phi\|_{H^{s,s}})$ with $u,\partial_x u,\partial_y u \in {L^2_TL^{\infty}_{xy}}$ if $\|\phi\|_{H^{s,s}}$ is sufficiently small. Moreover, the mapping $\phi \rightarrow u$ is continuous from $H^{s,s}(\mathbb{R}\times\mathbb{T})$ to $C([-T,T];H^{s,s}(\mathbb{R}\times\mathbb{T}))$. 

\end{theorem} 
For all our proofs, we will use the same method as used in Ionescu and Kenig \cite{ionescu1}. We are bounding $\|u\|_{L^2_TL^{\infty}_{xy}}+\|\partial_x u\|_{L^2_TL^{\infty}_{xy}}+\|\partial_y u\|_{L^2_TL^{\infty}_{xy}}$ as we have a cubic nonlinearity instead of a quadratic nonlinearity. In order to achieve this we are going to use the same time-frequency localized Strichartz estimates, instead of the Strichartz estimates for the flow of KP-I in $\mathbb{R}\times\mathbb{R}.$ 
  
The rest of the paper is organized as follows: in Section \ref{Notation}, we introduce the notation we are using through out the paper. In Section \ref{Dispersive Estimates}, we are stating the dispersive estimates that appear in \cite{ionescu1} for $\mathbb{R}\times\mathbb{T}$, which will help us to show the linear estimates in Section \ref{Linear Estimates}. Finally, in Section \ref{Local Well-Posedness} we begin to prove Theorems \ref{maintheorem3RT}, \ref{maintheorem3TT}, \ref{maintheorem5RT}, where we start with the existence result and unicity and finish by using a Bona-Smith argument as in \cite{bona} to prove continuity in the space and continuity of the flow map.   
  
\section{Notation and Preliminaries} \label{Notation}

   We start by defining, for $g\in L^2(\mathbb{R}\times\mathbb{T})$, $\widehat{g}(\xi,n)$ denote its Fourier transform in both $x$ and $y$. We define the Sobolev spaces which we will consider from now on: for $s_1,s_2\geq0$ 
  $$H^{s_1,s_2}(\mathbb{R}\times\mathbb{T})=\{g\in L^2(\mathbb{R}\times\mathbb{T}):\|g\|_{H^{s_1,s_2}}=\|\widehat{g}(\xi,n)[(1+\xi^2)^{\frac{s_1}{2}}+(1+n^2)^{\frac{s_2}{2}}]\|_{L^2(\mathbb{R}\times\mathbb{Z})}<\infty\}$$ 
  and for $s\geq0$ $$H^{s}(\mathbb{R}\times\mathbb{T})=\{g\in L^2(\mathbb{R}\times\mathbb{T}):\|g\|_{H^{s}}=\|\widehat{g}(\xi,n)[(1+\xi^2+n^2)^{\frac{s}{2}}]\|_{L^2(\mathbb{R}\times\mathbb{Z})}<\infty\}$$
  and so $$H^{\infty}(\mathbb{R}\times\mathbb{T})=\cap_{k=0}^{\infty}H^{k}(\mathbb{R}\times\mathbb{T}).$$
  
  For $s \in \mathbb{R}$ we define the operators $J^s_x,J^s_y$ by 
  $$\widehat{J^s_xg}(\xi,n)=(1+\xi^2)^{\frac{s}{2}}\widehat{g}(\xi,n);$$
  $$\widehat{J^s_y g}(\xi,n)=(1+n^2)^{\frac{s}{2}}\widehat{g}(\xi,n)$$
  on $\mathcal{S'}(\mathbb{R}\times \mathbb{T}).$
  
  For $g\in L^2(\mathbb{T}\times\mathbb{T})$, $\widehat{g}(m,n)$ denote its Fourier transform in both $x$ and $y$. In this case, we define similarly the Sobolev spaces which we will consider from now on: for $s_1,s_2\geq0$ 
  $$H^{s_1,s_2}(\mathbb{T}\times\mathbb{T})=\{g\in L^2(\mathbb{T}\times\mathbb{T}):\|g\|_{H^{s_1,s_2}}=\|\widehat{g}(m,n)[(1+m^2)^{\frac{s_1}{2}}+(1+n^2)^{\frac{s_2}{2}}]\|_{L^2(\mathbb{Z}\times\mathbb{Z})}<\infty\}$$ 
  and for $s\geq0$ $$H^{s}(\mathbb{T}\times\mathbb{T})=\{g\in L^2(\mathbb{T}\times\mathbb{T}):\|g\|_{H^{s}}=\|\widehat{g}(m,n)[(1+m^2+n^2)^{\frac{s}{2}}]\|_{L^2(\mathbb{Z}\times\mathbb{Z})}<\infty\}$$
  and so $$H^{\infty}(\mathbb{T}\times\mathbb{T})=\cap_{k=0}^{\infty}H^{k}(\mathbb{T}\times\mathbb{T}).$$
  
  By slight abuse of notation, for $s \in \mathbb{R}$ we define the operators $J^s_x,J^s_y$ by 
  $$\widehat{J^s_xg}(m,n)=(1+m^2)^{\frac{s}{2}}\widehat{g}(m,n);$$
  $$\widehat{J^s_y g}(m,n)=(1+n^2)^{\frac{s}{2}}\widehat{g}(m,n)$$
  on $\mathcal{S'}(\mathbb{T}\times \mathbb{T}).$
  
  For any set $A$ let $\textbf{1}_A$ denote its the characteristic function. Given a Banach space $X$, a measurable function $u : \mathbb{R}\rightarrow X$, and an exponent $p\in [1,\infty]$, we define
  \begin{equation*}
  \begin{split} 
  &\|u\|_{L^pX}=\Big[ \int_{\mathbb{R}}(\|u(t)\|^p_X)dt\Big]^{\frac{1}{p}} \mbox{ if } p\in [1,\infty) \mbox{ and } \\
  &\|u\|_{L^{\infty}X}=\mbox{esssup}_{t \in \mathbb{R}}\|u(t)\|_{X}
  \end{split}
  \end{equation*} 
  Also, if $I\subseteq \mathbb{R}$ is a measurable set, and $u : I \rightarrow X$ is a measurable function, we define $$\|u\|_{L^p_IX}=\|\textbf{1}_I(t)u\|_{L^pX}.$$ 
  
  For $T \geq 0$, we define $\|u\|_{L^p_TX}=\|u\|_{L^p_{[-T,T]}X}$  
  
  We also introduce the Kato-Ponce commutator estimates (as in Lemma XI from \cite{kato} and Appendix 9.A  from \cite{ionescu1}): 
  
  \begin{lemma} \label{katoponce}
\begin{itemize}
\item[(a)] Let $m\geq 0$ and $f,g \in H^{m}(\mathbb{R})$. If $s\geq1$ then 
  $$\|J^s_{\mathbb{R}}(fg)-fJ^s_{\mathbb{R}}g\|_{L^2}\leq C_s [\|J^s_{\mathbb{R}}f\|_{L^2}\|g\|_{L^{\infty}}+(\|f\|_{L^{\infty}}+\|\partial f\|_{L^{\infty}})\|J^{s-1}_{\mathbb{R}}g\|_{L^2}]$$
  and if $s\in(0,1)$ then   $$\|J^s_{\mathbb{R}}(fg)-fJ^s_{\mathbb{R}}g\|_{L^2}\leq C_s \|J^s_{\mathbb{R}}f\|_{L^2}\|g\|_{L^{\infty}}.$$
  \item[(b)] Let $m > 0$ and $f,g \in H^m(\mathbb{T}).$ If $s\geq 1$, then 
  $$\|J^s_{\mathbb{T}}(fg)-fJ^s_{\mathbb{T}}g\|_{L^2}\leq C_s [\|J^s_{\mathbb{T}}f\|_{L^2}\|g\|_{L^{\infty}}+(\|f\|_{L^{\infty}}+\|\partial f\|_{L^{\infty}})\|J^{s-1}_{\mathbb{T}}g\|_{L^2}]$$ and if $s\in (0,1)$ then   $$\|J^s_{\mathbb{T}}(fg)-fJ^s_{\mathbb{T}}g\|_{L^2}\leq C_s \|J^s_{\mathbb{T}}f\|_{L^2}\|g\|_{L^{\infty}}$$
    
  \item[(c)] Let $m\geq 0$ and $f,g \in H^{m}(M)$, where $M$ is either $\mathbb{R}$ or $\mathbb{R}$. If $s>0$ then  
  $$\|J^s_M(fg)\|_{L^2}\leq C_s \|J^s_Mf\|_{L^2}\|g\|_{L^{\infty}}+ \|J^s_Mg\|_{L^2}\|f\|_{L^{\infty}}.$$
\end{itemize}
  \end{lemma} 
  We have the following corollary which we will use later. 
\begin{corollary} \label{corollary} Let $M$ be either $\mathbb{R}$ or $\mathbb{T}.$
\begin{itemize} 
\item[(a)] If $s \geq 1$, then we have $\|J_M^s(u^3)\|\lesssim \|u\|^2_{L^{\infty}}\|J_M^su\|_{L^2}+\|u\|_{L^{\infty}}\|\partial u\|_{L^{\infty}}\|J_M^{s-1}u\|_{L^2}.$
\item[(b)] If $s\in (0,1)$, then we have $\|J_M^s(u^3)\|\lesssim \|u\|^2_{L^{\infty}}\|J_M^su\|_{L^2}.$
\end{itemize} 
\end{corollary} 

\section{Dispersive Estimates} \label{Dispersive Estimates}
For integers $k=0,1,\ldots$ we define the operators $Q_x^k,Q_y^k,\widetilde{Q}_x^k,\widetilde{Q}_y^k$ on $H^{\infty}(\mathbb{R}\times\mathbb{T})$ by 
 $$\widehat{Q_x^kg}(\xi,n)=\textbf{1}_{[2^{k-1},2^k)}(|\xi|) \mbox{ if } k\geq 1$$ 
with  $$\widehat{Q_x^0g}(\xi,n)=\textbf{1}_{[0,1)}(|\xi|)$$ 
and 
$$\widehat{Q_y^kg}(\xi,n)=\textbf{1}_{[2^{k-1},2^k)}(|n|) \mbox{ if } k\geq 1$$ with 
$$\widehat{Q_y^0g}(\xi,n)=\textbf{1}_{[0,1)}(|n|).$$ 
Also, $\widetilde{Q}_x^k=\sum_{k'=0}^k Q_x^{k'}, \widetilde{Q}_y^k=\sum_{k'=0}^k Q_y^{k'},k\geq1$. 

By slight abuse of notation, we define the operators $Q_x^k,Q_y^k,\widetilde{Q}_x^k,\widetilde{Q}_y^k$ on $H^{\infty}(\mathbb{T}\times\mathbb{T})$ by 
 $$\widehat{Q_x^kg}(m,n)=\textbf{1}_{[2^{k-1},2^k)}(|m|) \mbox{ if } k\geq 1$$ 
with  $$\widehat{Q_x^0g}(m,n)=\textbf{1}_{[0,1)}(|m|)$$ 
and 
$$\widehat{Q_y^kg}(m,n)=\textbf{1}_{[2^{k-1},2^k)}(|n|) \mbox{ if } k\geq 1$$ with 
$$\widehat{Q_y^0g}(m,n)=\textbf{1}_{[2^{k-1},2^k)}(|n|).$$ 
Also, $\widetilde{Q}_x^k=\sum_{k'=0}^k Q_x^{k'}, \widetilde{Q}_y^k=\sum_{k'=0}^k Q_y^{k'},k\geq 1$ . 

We are stating the dispersion estimates for the partially periodic and fully periodic cases that appear in Kenig and Ionescu \cite{ionescu1}. 

  \begin{theorem} 
    For $t \in \mathbb{R}$ let $W_{(3)}(t)$ denote the operator on $H^{\infty}(\mathbb{R}\times\mathbb{T})$ defined by the Fourier multiplier $(\xi,n)\mapsto \mbox{e}^{i(\xi^3+\frac{n^2}{\xi})t}.$ Assume $\phi \in H^{\infty}(\mathbb{R}\times \mathbb{T}).$ Then for any $\epsilon >0,$ we have 
  \begin{equation}\label{eqdispersive3RT}
  \|W_{(3)}(t)\widetilde{Q}^{2j}_yQ_x^j\phi \|_{L^2_{2^{-j}}L^{\infty}_{xy}}\leq C_{\epsilon} 2^{\epsilon j}\|\tilde{Q}_y^{2j}Q_x^j\phi\|_{L^2_{xy}}
  \end{equation}
  and 
  \begin{equation}\label{eqdispersive3bisRT}
  \|W_{(3)}(t)Q_y^{2j+k}Q_x^j\phi\|_{L^2_{2^{-j-k}}L^{\infty}_{xy}}\leq C_{\epsilon}2^{\epsilon j} \|Q_y^{2j+k}Q_x^j\phi \|_{L^2_{xy}}
  \end{equation} 
  for any integers $j\geq 0$ and $k\geq 1.$  
  \end{theorem}
  
  \begin{theorem} 
  For $t \in \mathbb{R}$ let $\widetilde{W}_{(3)}(t)$ denote the operator on $H^{\infty}(\mathbb{T}\times\mathbb{T})$ defined by the Fourier multiplier $(m,n)\mapsto \mbox{e}^{i(m^3+\frac{n^2}{m})t}.$Assume $\phi \in H^{\infty}(\mathbb{T}\times \mathbb{T}).$ Then for any $\epsilon >0,$ we have 
  \begin{equation}\label{eqdispersive3TT}
  \|\widetilde{W}_{(3)}(t)\widetilde{Q}^{2j}_yQ_x^j\phi \|_{L^2_{2^{-j}}L^{\infty}_{xy}}\leq C_{\epsilon} 2^{(\frac{3}{8}+\frac{\epsilon}{2}) j}\|\tilde{Q}_y^{2j}Q_x^j\phi\|_{L^2_{xy}}
  \end{equation}
  and 
  \begin{equation}\label{eqdispersive3bisTT}
  \|\widetilde{W}_{(3)}(t)Q_y^{2j+k}Q_x^j\phi\|_{L^2_{2^{-j-k}}L^{\infty}_{xy}}\leq C_{\epsilon}2^{(\frac{3}{8}+\frac{\epsilon}{2}) j} \|Q_y^{2j+k}Q_x^j\phi \|_{L^2_{xy}}
  \end{equation} 
  for any integers $j\geq 0$ and $k\geq 1.$  
  \end{theorem}
 
  \begin{theorem} 
  For $t \in \mathbb{R}$ let $W_{(5)}(t)$ denote the operator on $H^{\infty}(\mathbb{R}\times\mathbb{T})$ defined by the Fourier multiplier $(\xi,n)\mapsto \mbox{e}^{i(\xi^5+\frac{n^2}{\xi})t}.$ Assume $\phi \in H^{\infty}(\mathbb{R}\times\mathbb{T}).$ Then, for any $\epsilon>0,$ 
 \begin{equation}\label{eqdispersive5}
 \|W_{(5)}(t)\tilde{Q}^{3j}_yQ^j_x\phi\|_{L^2_{2^{-j}}L^{\infty}_{xy}}\leq C_{\epsilon}2^{(-\frac{1}{2}+\epsilon)j}\|\tilde{Q}^{2j}_yQ_x^j\phi\|_{L^2_{xy}}
 \end{equation}
 and 
 \begin{equation}\label{eqdispersive5bis}
 \|W_{(5)}(t)Q_y^{3j+k}Q^j_x\phi\|_{L^2_{2^{-2j-k}}L^{\infty}_{xy}}\leq C_{\epsilon}2^{(-\frac{1}{2}+\epsilon)j}\|Q_y^{3j+k}Q_x^j\phi\|_{L^2_{xy}}
 \end{equation}
 for any integers $j\geq 0$ and $k\geq 1.$ 
 \end{theorem}
 
  \section{Linear Estimate} \label{Linear Estimates}
  We continue by adapting the argument in \cite{ionescu1} to get the linear estimates. 
  \begin{proposition} \label{linearestimate3RT}
  Assume $N \geq 4$, $u \in C^1([-T,T]:H^{-N-1}(\mathbb{R}\times \mathbb{T}))$, $f \in C([-T,T]:H^{-N}(\mathbb{R}\times \mathbb{T})$ with $T \in [0,\frac{1}{2}]$ and$$ [\partial_t+\partial_x^3-\partial_x^{-1}\partial_y^2]u=\partial_x f \mbox{ on } \mathbb{R}\times\mathbb{T}\times [-T,T].$$ 
  Then for any $\epsilon>0$, we have 
  $$\|u\|_{L^2_TL^{\infty}_{xy}}\leq C_{\epsilon}[\|J^{1+\epsilon}_x u\|_{L^{\infty}_TL^2_{xy}}+\|J_x^{-1}J_y^{1+\epsilon}u\|_{L^{\infty}_TL^2_{xy}}+\|J_x^{1+\epsilon}J_y^{\epsilon}f\|_{L^{1}_TL^2_{xy}}].$$
  \end{proposition}
  \begin{proof} 
   Without loss of generality, we may assume that $u \in C^1([-T,T]:H^{\infty}(\mathbb{R}\times\mathbb{T}))$ and $f \in C([-T,T]:H^{\infty}(\mathbb{R}\times\mathbb{T})).$ It suffices to prove that if, for $\epsilon>0,$ 
\begin{equation}\label{eqdisp3RT}
\|\widetilde{Q}^{2j}_yQ^j_xu\|_{L^2_TL^{\infty}_{xy}}\leq C_{\epsilon}2^{-\frac{\epsilon j}{2}}\Big[\|J_x^{1+\epsilon}u\|_{L^{\infty}_TL^{2}_{xy}}+\|J_x^{1+\epsilon}f\|_{L^{1}_TL^{2}_{xy}}\Big]
\end{equation}
 and 
 \begin{equation}\label{eqdisp3bisRT} 
 \|Q^{2j+k}_yQ^j_xu\|_{L^{2}_TL^{\infty}_{xy}}\leq C_{\epsilon} 2^{-\frac{\epsilon(j+k)}{2}}\Big[\|J_x^{-1}J_y^{1+\epsilon}u\|_{L^{\infty}_TL^{2}_{xy}}+\|J_x^{1}J_y^{\epsilon}f\|_{L^{1}_TL^{2}_{xy}}\Big]
 \end{equation}
  for any integers $j\leq 0$ and $k\leq 1.$ 
 For ($\ref{eqdisp3RT}$), we partition the interval $[-T,T]$ into $2^{j}$ equal subintervals of length $2T2^{-j},$ denoted by $[a_{j,l},a_{j,l+1}),l=1,\ldots,2^{j}.$  By Duhamel's formula, for $t\in [a_{j,l},a_{j,l+1}],$ 
 $$u(t)=W_{(3)}(t-a_{j,l})[u(a_{j,l})]+\int_{a_{j,l}}^tW_{(3)}(t-s)[\partial_xf(s)]ds.$$ 
 It follows from the dispersive estimate ($\ref{eqdispersive3RT}$) that 
 \begin{equation} \label{eqint2'RT} 
 \begin{split} 
 \|\textbf{1}_{[a_{j,l},a_{j,l+1})}(t)&\widetilde{Q}_y^{2j}Q_x^ju\|_{L^2_TL^{\infty}_{xy}}\\&\leq C_{\epsilon} \|\textbf{1}_{[a_{j,l},a_{j,l+1})}(t)W_{(3)}(t-a_{j,l})\widetilde{Q}_y^{2j}Q_x^j u(a_{j,l})\|_{L^2_TL^{\infty}_{xy}}\\&+C_{\epsilon}\|\textbf{1}_{[a_{j,l},a_{j,l+1})}(t)\int_{a_{j,l}}^tW_{(3)}(s)\widetilde{Q}_y^{2j}Q_x^j \partial_xf(s)ds\|_{L^2_TL^{\infty}_{xy}}\\& \lesssim C_{\epsilon} 2^{\frac{\epsilon j}{2}}\|\widetilde{Q}_y^{2j}Q_x^j u(a_{j,l})\|_{L^{2}_{xy}}\\&+C_{\epsilon}2^{\frac{\epsilon j}{2}}2^j\|\textbf{1}_{[a_{j,l},a_{j,l+1})}(t)\widetilde{Q}_y^{2j}Q_x^j f\|_{L^1_TL^{2}_{xy}}.
 \end{split}
 \end{equation} 
 For the first term of the right-hand side of ($\ref{eqint2'RT}$), we have 
 \begin{equation} \label{eq17'RT}
 \begin{split} 
 \sum_{l=1}^{2^{j}}2^{\frac{\epsilon j}{2}}&\|\textbf{1}_{[a_{j,l},a_{j,l+1})}(t)\widetilde{Q}_y^{2j}Q_x^ju(a_{j,l})\|_{L^{2}_{xy}}\\&\lesssim \sum_{l=1}^{2^{j}}2^{\frac{\epsilon j}{2}}2^{-(1+\epsilon)j}\|\textbf{1}_{[a_{j,l},a_{j,l+1})}(t)\widetilde{Q}_y^{2j}Q_x^jJ_x^{1+\epsilon}u(a_{j,l})\|_{L^{2}_{xy}}\\& \lesssim 2^{j}2^{\frac{\epsilon j}{2}}2^{-(1+\epsilon)j}\|\widetilde{Q}_y^{2j}Q_x^jJ_x^{1+\epsilon}u\|_{L^{\infty}_TL^{2}_{xy}}\\& \lesssim 2^{-\frac{\epsilon j}{2}}\|J_x^{1+\epsilon}u\|_{L^{\infty}_TL^{2}_{xy}}.
 \end{split} 
 \end{equation} 
 For the second term of the right-hand side of ($\ref{eqint2'RT}$) we have 
 \begin{equation} \label{eq18'RT}
 \begin{split}
  \sum_{l=1}^{2^{j}}2^{\frac{\epsilon j}{2}}&\|\textbf{1}_{[a_{j,l},a_{j,l+1})}(t)\widetilde{Q}_y^{2j}Q_x^j f\|_{L^1_TL^{2}_{xy}} \\&\lesssim \sum_{l=1}^{2^{j}}2^{\frac{\epsilon j}{2}}2^j2^{-(1+\epsilon)j}\|\textbf{1}_{[a_{j,l},a_{j,l+1})}(t)\widetilde{Q}_y^{2j}Q_x^j J_x^{1+\epsilon}f\|_{L^1_TL^{2}_{xy}}\\& \lesssim 2^{-\frac{\epsilon j}{2}}\sum_{l=1}^{2^{j}}\|\textbf{1}_{[a_{j,l},a_{j,l+1})}(t)\widetilde{Q}_y^{2j}Q_x^j J_x^{1+\epsilon}f\|_{L^1_TL^{2}_{xy}}\\&\lesssim 2^{-\frac{\epsilon j}{2}}\|\widetilde{Q}_y^{2j}Q_x^j J_x^{1+\epsilon}f\|_{L^1_TL^{2}_{xy}}\\& \lesssim2^{-\frac{\epsilon j}{2}}\|J_x^{1+\epsilon}f\|_{L^1_TL^{2}_{xy}}.
  \end{split}
  \end{equation}
  Therefore, ($\ref{eq17'RT}$) and ($\ref{eq18'RT}$) give ($\ref{eqdisp3RT}$).

 For ($\ref{eqdisp3bisRT}$), we partition the interval $[-T,T]$ into $2^{j+k}$ equal subintervals of length $2T2^{-j-k},$ denoted by $[b_{j,l},b_{j,l+1}),l=1,\ldots,2^{j+k}.$
 By Duhamel's formula, for $t\in [b_{j,l},b_{j,l+1}],$ 
 $$u(t)=W_{(3)}(t-b_{j,l})[u(b_{j,l})]+\int_{b_{j,l}}^tW_{(3)}(t-s)[\partial_xf(s)]ds.$$ 
 It follows from the dispersive estimate ($\ref{eqdispersive3bisRT}$) that 
 \begin{equation} \label{eqint3'RT} 
 \begin{split} 
 \|\textbf{1}_{[b_{j,l},b_{j,l+1})}(t)&Q_y^{2j+k}Q_x^ju\|_{L^2_TL^{\infty}_{xy}}\\&\leq C_{\epsilon} \|\textbf{1}_{[b_{j,l},b_{j,l+1})}(t)W_{(3)}(t-b_{j,l})Q_y^{2j+k}Q_x^j u(b_{j,l})\|_{L^2_TL^{\infty}_{xy}}\\&+\|\textbf{1}_{[b_{j,l},b_{j,l+1})}(t)\int_{b_{j,l}}^tW_{(3)}(s)Q_y^{2j+k}Q_x^j \partial_xf(s)ds\|_{L^2_TL^{\infty}_{xy}}\\& \lesssim C_{\epsilon} 2^{\frac{\epsilon j}{2}}\|\textbf{1}_{[b_{j,l},b_{j,l+1})}(t)Q_y^{2j+k}Q_x^j u(b_{j,l})\|_{L^{2}_{xy}}\\&+C_{\epsilon}2^{\frac{\epsilon j}{2}}2^j\|\textbf{1}_{[b_{j,l},b_{j,l+1})}(t)Q_y^{2j+k}Q_x^j f\|_{L^1_TL^{2}_{xy}}.
 \end{split}
 \end{equation} 
 For the first term of the right-hand side of ($\ref{eqint3'RT}$), we have 
 \begin{equation} \label{eq19'RT}
 \begin{split} 
 \sum_{l=1}^{2^{j+k}}&2^{\frac{\epsilon j}{2}}\|\textbf{1}_{[b_{j,l},b_{j,l+1})}(t)Q_y^{2j+k}Q_x^ju(b_{j,l})\|_{L^{2}_{xy}}\\&\lesssim \sum_{l=1}^{2^{j+k}}2^{\frac{\epsilon j}{2}}2^{j}2^{-(1+\frac{\epsilon}{2})(2j+k)}\|\textbf{1}_{[b_{j,l},b_{j,l+1})}(t)Q_y^{2j+k}Q_x^jJ_x^{-1}J_y^{1+\frac{\epsilon}{2}}u(b_{j,l})\|_{L^{2}_{xy}}\\& \lesssim 2^{j+k}2^{\frac{\epsilon j}{2}}2^{ j}2^{-(1+\frac{\epsilon}{2})(2j+k)}\|Q_y^{2j+k}Q_x^jJ_x^{-1}J_y^{1+\frac{\epsilon}{2}}u\|_{L^{\infty}_TL^{2}_{xy}}\\& \lesssim 2^{-\frac{\epsilon (j+k)}{2}}\|J_x^{-1}J_y^{1+\frac{\epsilon}{2}}u\|_{L^{\infty}_TL^{2}_{xy}}.
 \end{split} 
 \end{equation} 
  For the second term of the right-hand side of ($\ref{eqint3'RT}$) we have 
 \begin{equation} \label{eq20'RT}
 \begin{split}
  \sum_{l=1}^{2^{j+k}}2^{\frac{\epsilon j}{2}}2^j&\|\textbf{1}_{[b_{j,l},b_{j,l+1})}(t)Q_y^{2j+k}Q_x^j f\|_{L^1_TL^{2}_{xy}} \\&\lesssim \sum_{l=1}^{2^{j+k}}2^{\frac{\epsilon j}{2}}2^j2^{-j}2^{-\frac{\epsilon}{2}(2j+k)}\|\textbf{1}_{[b_{j,l},b_{j,l+1})}(t)Q_y^{2j+k}Q_x^j J_x^{1}J_y^{\frac{\epsilon}{2}}f\|_{L^1_TL^{2}_{xy}}\\& \lesssim 2^{-\frac{\epsilon (j+k)}{2}}\sum_{l=1}^{2^{j+k}}\|\textbf{1}_{[b_{j,l},b_{j,l+1})}(t)Q_y^{2j+k}Q_x^j J_x^{1}J_y^{\frac{\epsilon}{2}}f\|_{L^1_TL^{2}_{xy}}\\&\lesssim 2^{-\frac{\epsilon (j+k)}{2}}\|Q_y^{2j+k}Q_x^j J_x^{1}J_y^{\frac{\epsilon}{2}}f\|_{L^1_TL^{2}_{xy}}\\& \lesssim2^{-\frac{\epsilon (j+k)}{2}}\|J_x^{1}J_y^{\frac{\epsilon}{2}}f\|_{L^1_TL^{2}_{xy}}.
  \end{split}
  \end{equation}
  Therefore, ($\ref{eq19'RT}$) and ($\ref{eq20'RT}$) give ($\ref{eqdisp3bisRT}$).
  \end{proof}

  \begin{proposition} \label{linearestimate3TT}
  Assume $N \geq 4$, $u \in C^1([-T,T]:H^{-N-1}(\mathbb{T}\times \mathbb{T}))$, $f \in C([-T,T]:H^{-N}(\mathbb{T}\times \mathbb{T})$ with $T \in [0,\frac{1}{2}]$ and$$ [\partial_t+\partial_x^3-\partial_x^{-1}\partial_y^2]u=\partial_x f \mbox{ on } \mathbb{T}\times\mathbb{T}\times [-T,T].$$ 
  Then for any $\epsilon>0$, we have 
  $$\|u\|_{L^2_TL^{\infty}_{xy}}\leq C_{\epsilon}[\|J^{\frac{11}{8}+\epsilon}_x u\|_{L^{\infty}_TL^2_{xy}}+\|J_x^{-\frac{5}{8}}J_y^{1+\epsilon}u\|_{L^{\infty}_TL^2_{xy}}+\|J_x^{\frac{11}{8}+\epsilon}J_y^{\epsilon}f\|_{L^{1}_TL^2_{xy}}].$$
  \end{proposition} 
  \begin{proof} 
   Without loss of generality, we may assume that $u \in C^1([-T,T]:H^{\infty}(\mathbb{T}\times\mathbb{T}))$ and $f \in C([-T,T]:H^{\infty}(\mathbb{T}\times\mathbb{T})).$ It suffices to prove that if, for $\epsilon>0,$ 
\begin{equation}\label{eqdisp3TT}
\|\widetilde{Q}^{2j}_yQ^j_xu\|_{L^2_TL^{\infty}_{xy}}\leq C_{\epsilon}2^{-\frac{\epsilon j}{2}}\Big[\|J_x^{\frac{11}{8}+\epsilon}u\|_{L^{\infty}_TL^{2}_{xy}}+\|J_x^{\frac{11}{8}+\epsilon}f\|_{L^{1}_TL^{2}_{xy}}\Big]
\end{equation}
 and 
 \begin{equation}\label{eqdisp3bisTT} 
 \|Q^{2j+k}_yQ^j_xu\|_{L^{2}_TL^{\infty}_{xy}}\leq C_{\epsilon} 2^{-\frac{\epsilon(j+k)}{2}}\Big[\|J_x^{-\frac{5}{8}}J_y^{1+\epsilon}u\|_{L^{\infty}_TL^{2}_{xy}}+\|J_x^{\frac{11}{8}}J_y^{\epsilon}f\|_{L^{1}_TL^{2}_{xy}}\Big]
 \end{equation}
  for any integers $j\leq 0$ and $k\leq 1.$ 
 For ($\ref{eqdisp3TT}$), we partition the interval $[-T,T]$ into $2^{j}$ equal subintervals of length $2T2^{-j},$ denoted by $[a_{j,l},a_{j,l+1}),l=1,\ldots,2^{j}.$  By Duhamel's formula, for $t\in [a_{j,l},a_{j,l+1}],$ 
 $$u(t)=\widetilde{W}_{(3)}(t-a_{j,l})[u(a_{j,l})]+\int_{a_{j,l}}^t\widetilde{W}_{(3)}(t-s)[\partial_xf(s)]ds.$$ 
 It follows from the dispersive estimate ($\ref{eqdispersive3TT}$) that 
 \begin{equation} \label{eqint2'TT} 
 \begin{split} 
 \|\textbf{1}_{[a_{j,l},a_{j,l+1})}(t)&\widetilde{Q}_y^{2j}Q_x^ju\|_{L^2_TL^{\infty}_{xy}}\\&\leq C_{\epsilon} \|\textbf{1}_{[a_{j,l},a_{j,l+1})}(t)\widetilde{W}_{(3)}(t-a_{j,l})\widetilde{Q}_y^{2j}Q_x^j u(a_{j,l})\|_{L^2_TL^{\infty}_{xy}}\\&+C_{\epsilon}\|\textbf{1}_{[a_{j,l},a_{j,l+1})}(t)\int_{a_{j,l}}^t\widetilde{W}_{(3)}(s)\widetilde{Q}_y^{2j}Q_x^j \partial_xf(s)ds\|_{L^2_TL^{\infty}_{xy}}\\& \lesssim C_{\epsilon} 2^{(\frac{3}{8}+\frac{\epsilon}{2}) j}\|\widetilde{Q}_y^{2j}Q_x^j u(a_{j,l})\|_{L^{2}_{xy}}\\&+C_{\epsilon}2^{(\frac{3}{8}+\frac{\epsilon}{2}) j}2^j\|\textbf{1}_{[a_{j,l},a_{j,l+1})}(t)\widetilde{Q}_y^{2j}Q_x^j f\|_{L^1_TL^{2}_{xy}}.
 \end{split}
 \end{equation} 
 For the first term of the right-hand side of ($\ref{eqint2'TT}$), we have 
 \begin{equation} \label{eq17'TT}
 \begin{split} 
 \sum_{l=1}^{2^{j}}2^{(\frac{3}{8}+\frac{\epsilon}{2}) j}&\|\textbf{1}_{[a_{j,l},a_{j,l+1})}(t)\widetilde{Q}_y^{2j}Q_x^ju(a_{j,l})\|_{L^{2}_{xy}}\\&\lesssim \sum_{l=1}^{2^{j}}2^{(\frac{3}{8}+\frac{\epsilon}{2}) j}2^{-(\frac{11}{8}+\epsilon)j}\|\textbf{1}_{[a_{j,l},a_{j,l+1})}(t)\widetilde{Q}_y^{2j}Q_x^jJ_x^{\frac{11}{8}+\epsilon}u(a_{j,l})\|_{L^{2}_{xy}}\\& \lesssim 2^{j}2^{(\frac{3}{8}+\frac{\epsilon}{2}) j}2^{-(\frac{11}{8}+\epsilon)j}\|\widetilde{Q}_y^{2j}Q_x^jJ_x^{\frac{11}{8}+\epsilon}u\|_{L^{\infty}_TL^{2}_{xy}}\\& \lesssim 2^{-\frac{\epsilon j}{2}}\|J_x^{\frac{11}{8}+\epsilon}u\|_{L^{\infty}_TL^{2}_{xy}}.
 \end{split} 
 \end{equation} 
 For the second term of the right-hand side of ($\ref{eqint2'TT}$) we have 
 \begin{equation} \label{eq18'TT}
 \begin{split}
  \sum_{l=1}^{2^{j}}2^{(\frac{3}{8}+\frac{\epsilon}{2}) j}&\|\textbf{1}_{[a_{j,l},a_{j,l+1})}(t)\widetilde{Q}_y^{2j}Q_x^j f\|_{L^1_TL^{2}_{xy}} \\&\lesssim \sum_{l=1}^{2^{j}}2^{(\frac{3}{8}+\frac{\epsilon}{2}) j}2^j2^{-(\frac{11}{8}+\epsilon)j}\|\textbf{1}_{[a_{j,l},a_{j,l+1})}(t)\widetilde{Q}_y^{2j}Q_x^j J_x^{\frac{11}{8}+\epsilon}f\|_{L^1_TL^{2}_{xy}}\\& \lesssim 2^{-\frac{\epsilon j}{2}}\sum_{l=1}^{2^{j}}\|\textbf{1}_{[a_{j,l},a_{j,l+1})}(t)\widetilde{Q}_y^{2j}Q_x^j J_x^{\frac{11}{8}+\epsilon}f\|_{L^1_TL^{2}_{xy}}\\&\lesssim 2^{-\frac{\epsilon j}{2}}\|\widetilde{Q}_y^{2j}Q_x^j J_x^{\frac{11}{8}+\epsilon}f\|_{L^1_TL^{2}_{xy}}\\& \lesssim2^{-\frac{\epsilon j}{2}}\|J_x^{\frac{11}{8}+\epsilon}f\|_{L^1_TL^{2}_{xy}}.
  \end{split}
  \end{equation}
  Therefore, ($\ref{eq17'TT}$) and ($\ref{eq18'TT}$) give (\ref{eqdisp3TT}).

 For ($\ref{eqdisp3bisTT}$), we partition the interval $[-T,T]$ into $2^{j+k}$ equal subintervals of length $2T2^{-j-k},$ denoted by $[b_{j,l},b_{j,l+1}),l=1,\ldots,2^{j+k}.$
 By Duhamel's formula, for $t\in [b_{j,l},b_{j,l+1}],$ 
 $$u(t)=\widetilde{W}_{(3)}(t-b_{j,l})[u(b_{j,l})]+\int_{b_{j,l}}^t\widetilde{W}_{(3)}(t-s)[\partial_xf(s)]ds.$$ 
 It follows from the dispersive estimate ($\ref{eqdispersive3bisTT}$) that 
 \begin{equation} \label{eqint3'TT} 
 \begin{split} 
 \|\textbf{1}_{[b_{j,l},b_{j,l+1})}(t)&Q_y^{2j+k}Q_x^ju\|_{L^2_TL^{\infty}_{xy}}\\&\leq C_{\epsilon} \|\textbf{1}_{[b_{j,l},b_{j,l+1})}(t)\widetilde{W}_{(3)}(t-b_{j,l})Q_y^{2j+k}Q_x^j u(b_{j,l})\|_{L^2_TL^{\infty}_{xy}}\\&+\|\textbf{1}_{[b_{j,l},b_{j,l+1})}(t)\int_{b_{j,l}}^t\widetilde{W}_{(3)}(s)Q_y^{2j+k}Q_x^j \partial_xf(s)ds\|_{L^2_TL^{\infty}_{xy}}\\& \lesssim C_{\epsilon} 2^{(\frac{3}{8}+\frac{\epsilon}{2}) j}\|\textbf{1}_{[b_{j,l},b_{j,l+1})}(t)Q_y^{2j+k}Q_x^j u(b_{j,l})\|_{L^{2}_{xy}}\\&+C_{\epsilon}2^{(\frac{3}{8}+\frac{\epsilon}{2}) j}2^j\|\textbf{1}_{[b_{j,l},b_{j,l+1})}(t)Q_y^{2j+k}Q_x^j f\|_{L^1_TL^{2}_{xy}}.
 \end{split}
 \end{equation} 
 For the first term of the right-hand side of ($\ref{eqint3'TT}$), we have 
 \begin{equation} \label{eq19'TT}
 \begin{split} 
 \sum_{l=1}^{2^{j+k}}&2^{(\frac{3}{8}+\frac{\epsilon}{2}) j}\|\textbf{1}_{[b_{j,l},b_{j,l+1})}(t)Q_y^{2j+k}Q_x^ju(b_{j,l})\|_{L^{2}_{xy}}\\&\lesssim \sum_{l=1}^{2^{j+k}}2^{(\frac{3}{8}+\frac{\epsilon}{2}) j}2^{\frac{5}{8}j}2^{-(1+\frac{\epsilon}{2})(2j+k)}\|\textbf{1}_{[b_{j,l},b_{j,l+1})}(t)Q_y^{2j+k}Q_x^jJ_x^{-\frac{5}{8}}J_y^{1+\frac{\epsilon}{2}}u(b_{j,l})\|_{L^{2}_{xy}}\\& \lesssim 2^{j+k}2^{(\frac{3}{8}+\frac{\epsilon}{2}) j}2^{ \frac{5}{8}j}2^{-(1+\frac{\epsilon}{2})(2j+k)}\|Q_y^{2j+k}Q_x^jJ_x^{-\frac{5}{8}}J_y^{1+\frac{\epsilon}{2}}u\|_{L^{\infty}_TL^{2}_{xy}}\\& \lesssim 2^{-\frac{\epsilon (j+k)}{2}}\|J_x^{-\frac{5}{8}}J_y^{1+\frac{\epsilon}{2}}u\|_{L^{\infty}_TL^{2}_{xy}}.
 \end{split} 
 \end{equation} 
  For the second term of the right-hand side of ($\ref{eqint3'TT}$) we have 
 \begin{equation} \label{eq20'TT}
 \begin{split}
  \sum_{l=1}^{2^{j+k}}2^{(\frac{3}{8}+\frac{\epsilon}{2}) j}2^j&\|\textbf{1}_{[b_{j,l},b_{j,l+1})}(t)Q_y^{2j+k}Q_x^j f\|_{L^1_TL^{2}_{xy}} \\&\lesssim \sum_{l=1}^{2^{j+k}}2^{(\frac{3}{8}+\frac{\epsilon}{2}) j}2^j2^{-\frac{11}{8}j}2^{-\frac{\epsilon}{2}(2j+k)}\|\textbf{1}_{[b_{j,l},b_{j,l+1})}(t)Q_y^{2j+k}Q_x^j J_x^{\frac{11}{8}}J_y^{\frac{\epsilon}{2}}f\|_{L^1_TL^{2}_{xy}}\\& \lesssim 2^{-\frac{\epsilon (j+k)}{2}}\sum_{l=1}^{2^{j+k}}\|\textbf{1}_{[b_{j,l},b_{j,l+1})}(t)Q_y^{2j+k}Q_x^j J_x^{\frac{11}{8}}J_y^{\frac{\epsilon}{2}}f\|_{L^1_TL^{2}_{xy}}\\&\lesssim 2^{-\frac{\epsilon (j+k)}{2}}\|Q_y^{2j+k}Q_x^j J_x^{\frac{11}{8}}J_y^{\frac{\epsilon}{2}}f\|_{L^1_TL^{2}_{xy}}\\& \lesssim2^{-\frac{\epsilon (j+k)}{2}}\|J_x^{\frac{11}{8}}J_y^{\frac{\epsilon}{2}}f\|_{L^1_TL^{2}_{xy}}.
  \end{split}
  \end{equation}
  Therefore, ($\ref{eq19'TT}$) and ($\ref{eq20'TT}$) give ($\ref{eqdisp3bisTT}$).
  \end{proof}

 \begin{proposition} \label{linearestimate5RT}
 Assume $N\geq 4,$ $u \in C^1([-T,T]:H^{-N-1}(\mathbb{R}\times\mathbb{T})),$ $f \in C([-T,T]:H^{-N}(\mathbb{R}\times\mathbb{T})),$ $T\in[0,\frac{1}{2}]$ and 
 $$[\partial_t-\partial_x^5-\partial_x^{-1}\partial_y^2]u=\partial_x f \mbox{ on } \mathbb{R}\times\mathbb{T}\times [-T,T].$$ Then, for any $1\leq p \leq 2$, we have
\begin{equation*} 
\begin{split} 
 \|u\|_{L^p_TL^{\infty}_{xy}}\lesssim C_p T^{\frac{2-p}{2p}}[\|J_x^{\frac{5p-4}{2p}+\epsilon}u\|_{L^{\infty}_TL^{2}_{xy}}+\|J_x^{-\frac{2p-1}{p}+\epsilon}J_y^{\frac{3p-2}{2p}+\epsilon}u\|_{L^{\infty}_TL^{2}_{xy}}+\|J_x^{p'}J_y^{\frac{p-2}{2p}+\epsilon}u\|_{L^{1}_TL^{2}_{xy}}]
 \end{split} 
 \end{equation*} 
 where $p'=\mbox{max}(\frac{3p-4}{2p}+\epsilon,\frac{1}{p}-\epsilon).$
 \end{proposition} 
 \begin{proof} 
 Without loss of generality, we may assume that $u \in C^1([-T,T]:H^{\infty}(\mathbb{R}\times\mathbb{T}))$ and $f \in C([-T,T]:H^{\infty}(\mathbb{R}\times\mathbb{T})).$ It suffices to prove that if, for $\epsilon>0,$ 
\begin{equation}\label{eqdisp5}
\|\widetilde{Q}^{3j}_yQ^j_xu\|_{L^p_TL^{\infty}_{xy}}\leq C_{\epsilon}2^{-\frac{\epsilon j}{2}}T^{\frac{2-p}{2p}}[\|J_x^{\frac{5p-4}{2p}+\epsilon}u\|_{L^{\infty}_TL^{2}_{xy}}+\|J_x^{\frac{3p-4}{2p}}f\|_{L^{1}_TL^{2}_{xy}}]
\end{equation}
 and 
 \begin{equation}\label{eqdisp5bis} 
 \|Q^{3j+k}_yQ^j_xu\|_{L^{p}_TL^{\infty}_{xy}}\leq C_{\epsilon} 2^{-\frac{\epsilon(j+2k)}{2}}T^{\frac{2-p}{2p}}[\|J_x^{-\frac{2p-1}{2p}+\epsilon}J_y^{\frac{3p-2}{2p}+\epsilon}u\|_{L^{\infty}_TL^{2}_{xy}}+\|J_x^{\frac{3p-4}{2p}}f\|_{L^{1}_TL^{2}_{xy}}
 \end{equation}
  for any integers $j\leq 0$ and $k\leq 1.$ 
 For $(\ref{eqdisp5})$, we partition the interval $[-T,T]$ into $2^{2j}$ equal subintervals of length $2T2^{-2j},$ denoted by $[a_{j,l},a_{j,l+1}),l=1,\ldots,2^{2j}.$ The term in the left-hand side of $\ref{eqdisp5}$, using H{\"o}lder's inequality, is dominated by 
 \begin{equation} \label{eqint1}
 \begin{split} 
 \sum_{l=1}^{2^{2j}}\|\textbf{1}_{[a_{j,l},a_{j,l+1})}(t)&\widetilde{Q}^{3j}_yQ^j_xu\|_{L^p_TL^{\infty}_{xy}}\\&\leq C \|\textbf{1}_{[a_{j,l},a_{j,l+1})}(t)\|_{L^{\frac{2-p}{2p}}_TL^{\infty}_{xy}}\|\sum_{l=1}^{2^{2j}}\|\textbf{1}_{[a_{j,l},a_{j,l+1})}(t)\widetilde{Q}_y^{3j}Q_x^ju\|_{L^2_TL^{\infty}_{xy}}\\&\leq C2^{-\frac{2-p}{2p}2j}T^{\frac{2-p}{2p}}\sum_{l=1}^{2^{2j}}\|\textbf{1}_{[a_{j,l},a_{j,l+1})}(t)\widetilde{Q}_y^{3j}Q_x^ju\|_{L^2_TL^{\infty}_{xy}}.
 \end{split} 
 \end{equation} 
 By Duhamel's formula, for $t\in [a_{j,l},a_{j,l+1}],$ 
 $$u(t)=W_{(5)}(t-a_{j,l})[u(a_{j,l})]+\int_{a_{j,l}}^tW_{(5)}(t-s)[\partial_xf(s)]ds.$$ 
 It follows from the dispersive estimate $(\ref{eqdispersive5})$ that 
 \begin{equation} \label{eqint2} 
 \begin{split} 
 \|\textbf{1}_{[a_{j,l},a_{j,l+1})}(t)&\widetilde{Q}_y^{3j}Q_x^ju\|_{L^2_TL^{\infty}_{xy}}\\&\leq C_{\epsilon} \|\textbf{1}_{[a_{j,l},a_{j,l+1})}(t)W_{(5)}(t-a_{j,l})\widetilde{Q}_y^{3j}Q_x^j u(a_{j,l})\|_{L^2_TL^{\infty}_{xy}}\\&+\|\textbf{1}_{[a_{j,l},a_{j,l+1})}(t)\int_{a_{j,l}}^tW_{(5)}(s)\widetilde{Q}_y^{3j}Q_x^j \partial_xf(s)ds\|_{L^2_TL^{\infty}_{xy}}\\& \lesssim C_{\epsilon} 2^{(-\frac{1}{2}+\frac{\epsilon}{2})j}\|\widetilde{Q}_y^{3j}Q_x^j u(a_{j,l})\|_{L^{2}_{xy}}\\&+C_{\epsilon}2^{(-\frac{1}{2}+\frac{\epsilon}{2})j}2^j\|\textbf{1}_{[a_{j,l},a_{j,l+1})}(t)\widetilde{Q}_y^{3j}Q_x^j f\|_{L^1_TL^{2}_{xy}}.
 \end{split}
 \end{equation} 
 For the first term of the right-hand side of $(\ref{eqint2})$, we have 
 \begin{equation} \label{eq17}
 \begin{split} 
 \sum_{l=1}^{2^{2j}}2^{-\frac{2-p}{2p}2j}&2^{(-\frac{1}{2}+\frac{\epsilon}{2})j}\|\textbf{1}_{[a_{j,l},a_{j,l+1})}(t)\widetilde{Q}_y^{3j}Q_x^ju(a_{j,l})\|_{L^{2}_{xy}}\\&\lesssim \sum_{l=1}^{2^{2j}}2^{-\frac{2-p}{2p}2j}2^{(-\frac{1}{2}+\frac{\epsilon}{2})j}2^{-(\frac{5p-4}{2p}+\epsilon)j}\|\textbf{1}_{[a_{j,l},a_{j,l+1})}(t)\widetilde{Q}_y^{3j}Q_x^jJ_x^{\frac{5p-4}{2p}}u(a_{j,l})\|_{L^{2}_{xy}}\\& \lesssim 2^{2j}2^{-\frac{2-p}{2p}2j}2^{(-\frac{1}{2}+\frac{\epsilon}{2})j}2^{-(\frac{5p-4}{2p}+\epsilon)j}\|\widetilde{Q}_y^{3j}Q_x^jJ_x^{\frac{5p-4}{2p}}u\|_{L^{\infty}_TL^{2}_{xy}}\\& \lesssim 2^{-\frac{\epsilon j}{2}}\|J_x^{\frac{5p-4}{2p}}u\|_{L^{\infty}_TL^{2}_{xy}}.
 \end{split} 
 \end{equation} 
 For the second term of the right-hand side of $(\ref{eqint2})$ we have 
 \begin{equation} \label{eq18}
 \begin{split}
  \sum_{l=1}^{2^{2j}}2^{-\frac{2-p}{2p}2j}&2^{(-\frac{1}{2}+\frac{\epsilon}{2})j}\|\textbf{1}_{[a_{j,l},a_{j,l+1})}(t)\widetilde{Q}_y^{3j}Q_x^j f\|_{L^1_TL^{2}_{xy}} \\&\lesssim \sum_{l=1}^{2^{2j}}2^{-\frac{2-p}{2p}2j}2^{(-\frac{1}{2}+\frac{\epsilon}{2})j}2^j2^{-(\frac{3p-4}{2p}+\epsilon)j}\|\textbf{1}_{[a_{j,l},a_{j,l+1})}(t)\widetilde{Q}_y^{3j}Q_x^j J_x^{\frac{3p-4}{2p}+\epsilon}f\|_{L^1_TL^{2}_{xy}}\\& \lesssim 2^{-\frac{\epsilon j}{2}}\sum_{l=1}^{2^{2j}}\|\textbf{1}_{[a_{j,l},a_{j,l+1})}(t)\widetilde{Q}_y^{3j}Q_x^j J_x^{\frac{3p-4}{2p}+\epsilon}f\|_{L^1_TL^{2}_{xy}}\\&\lesssim 2^{-\frac{\epsilon j}{2}}\|\widetilde{Q}_y^{3j}Q_x^j J_x^{\frac{3p-4}{2p}+\epsilon}f\|_{L^1_TL^{2}_{xy}}\\& \lesssim2^{-\frac{\epsilon j}{2}}\|J_x^{\frac{3p-4}{2p}+\epsilon}f\|_{L^1_TL^{2}_{xy}}.
  \end{split}
  \end{equation}
  Therefore, $(\ref{eq17})$ and $(\ref{eq18})$ give $(\ref{eqdisp5}).$

 For $(\ref{eqdisp5bis})$, we partition the interval $[-T,T]$ into $2^{2j+k}$ equal subintervals of length $2T2^{-2j-k},$ denoted by $[b_{j,l},b_{j,l+1}),l=1,\ldots,2^{2j+k}.$ The term in the left-hand side of $(\ref{eqdisp5bis})$, using H{\"o}lder's inequality, is dominated by 
  \begin{equation*} 
   \begin{split} 
 \sum_{l=1}^{2^{2j+k}}\|\textbf{1}_{[b_{j,l},b_{j,l+1})}(t)&Q^{3j+k}_yQ^j_xu\|_{L^p_TL^{\infty}_{xy}}\\&\leq C \|\textbf{1}_{[b_{j,l},b_{j,l+1})}(t)\|_{L^{\frac{2-p}{2p}}_TL^{\infty}_{xy}}\|\sum_{l=1}^{2^{2j+k}}\|\textbf{1}_{[b_{j,l},b_{j,l+1})}(t)Q_y^{3j+k}Q_x^ju\|_{L^2_TL^{\infty}_{xy}}\\&\leq C2^{-\frac{2-p}{2p}(2j+k)}T^{\frac{2-p}{2p}}\sum_{l=1}^{2^{2j+k}}\|\textbf{1}_{[a_{j,l},a_{j,l+1})}(t)Q_y^{3j+k}Q_x^ju\|_{L^2_TL^{\infty}_{xy}}.
 \end{split} 
 \end{equation*} 
 By Duhamel's formula, for $t\in [b_{j,l},b_{j,l+1}],$ 
 $$u(t)=W_{(5)}(t-b_{j,l})[u(b_{j,l})]+\int_{b_{j,l}}^tW_{(5)}(t-s)[\partial_xf(s)]ds.$$ 
 It follows from the dispersive estimate  $(\ref{eqdispersive5bis})$ that 
 \begin{equation} \label{eqint3} 
 \begin{split} 
 \|\textbf{1}_{[b_{j,l},b_{j,l+1})}(t)&Q_y^{3j+k}Q_x^ju\|_{L^2_TL^{\infty}_{xy}}\\&\leq C_{\epsilon} \|\textbf{1}_{[b_{j,l},b_{j,l+1})}(t)W_{(5)}(t-b_{j,l})Q_y^{3j+k}Q_x^j u(b_{j,l})\|_{L^2_TL^{\infty}_{xy}}\\&+\|\textbf{1}_{[b_{j,l},b_{j,l+1})}(t)\int_{b_{j,l}}^tW_{(5)}(s)Q_y^{3j+k}Q_x^j \partial_xf(s)ds\|_{L^2_TL^{\infty}_{xy}}\\& \lesssim C_{\epsilon} 2^{(-\frac{1}{2}+\frac{\epsilon}{2})j}\|Q_y^{3j+k}Q_x^j u(b_{j,l})\|_{L^{2}_{xy}}\\&+C_{\epsilon}2^{(-\frac{1}{2}+\frac{\epsilon}{2})j}2^j\|\textbf{1}_{[b_{j,l},b_{j,l+1})}(t)Q_y^{3j+k}Q_x^j f\|_{L^1_TL^{2}_{xy}}.
 \end{split}
 \end{equation} 
 Denote $C_{j,k}^{\epsilon}=2^{-\frac{2-p}{2p}(2j+k)}2^{(-\frac{1}{2}+\frac{\epsilon}{2})j}.$ For the first term of the right-hand side of $(\ref{eqint3})$, we have 
 \begin{equation} \label{eq19}
 \begin{split} 
 &\sum_{l=1}^{2^{2j+k}}\|\textbf{1}_{[b_{j,l},b_{j,l+1})}(t)Q_y^{3j+k}Q_x^ju(b_{j,l})\|_{L^{2}_{xy}}
 \\&\lesssim \sum_{l=1}^{2^{2j+k}}C_{j,k}^{\epsilon}2^{(\frac{2p-1}{p}+\epsilon)j}2^{-(\frac{3p-2}{2p}+\epsilon)(3j+k)}\|\textbf{1}_{[b_{j,l},b_{j,l+1})}(t)Q_y^{3j+k}Q_x^jJ_x^{-\frac{2p-1}{p}-\epsilon}J_y^{\frac{3p-2}{2p}+\epsilon}u(b_{j,l})\|_{L^{2}_{xy}}
 \\& \lesssim 2^{2j+k}C_{j,k}^{\epsilon}2^{(\frac{2p-1}{p}+\epsilon)j}2^{-(\frac{3p-2}{2p}+\epsilon)(3j+k)}\|Q_y^{3j+k}Q_x^jJ_x^{-\frac{2p-1}{p}-\epsilon}J_y^{\frac{3p-2}{2p}+\epsilon}u\|_{L^{\infty}_TL^{2}_{xy}}
 \\& \lesssim 2^{-\frac{\epsilon (3j+2k)}{2}}\|J_x^{-\frac{2p-1}{p}-\epsilon}J_y^{\frac{3p-2}{2p}+\epsilon}u\|_{L^{\infty}_TL^{2}_{xy}}.
 \end{split} 
 \end{equation} 
  For the second term of the right-hand side of $(\ref{eqint3})$ we have 
 \begin{equation} \label{eq20}
 \begin{split}
  &\sum_{l=1}^{2^{2j+k}}2^{-\frac{2-p}{2p}(2j+k)}2^{(-\frac{1}{2}+\frac{\epsilon}{2})j}2^j\|\textbf{1}_{[b_{j,l},b_{j,l+1})}(t)Q_y^{3j+k}Q_x^j f\|_{L^1_TL^{2}_{xy}} \\&\lesssim \sum_{l=1}^{2^{2j+k}}2^{-\frac{2-p}{2p}(2j+k)}2^{(-\frac{1}{2}+\frac{\epsilon}{2})j}2^j2^{-(\frac{1}{p}-\epsilon)j}2^{(\frac{2-p}{2p}-\epsilon)(3j+k)}\|\textbf{1}_{[b_{j,l},b_{j,l+1})}(t)Q_y^{3j+k}Q_x^j J_x^{\frac{1}{p}+\epsilon}J_y^{-\frac{2-p}{2p}+\epsilon}f\|_{L^1_TL^{2}_{xy}}\\& \lesssim 2^{-\frac{\epsilon (3j+2k)}{2}}\sum_{l=1}^{2^{2j+k}}\|\textbf{1}_{[b_{j,l},b_{j,l+1})}(t)Q_y^{3j+k}Q_x^j J_x^{\frac{1}{p}-\epsilon}J_y^{-\frac{2-p}{2p}+\epsilon}f\|_{L^1_TL^{2}_{xy}}\\&\lesssim 2^{-\frac{\epsilon (3j+2k)}{2}}\|Q_y^{3j+k}Q_x^j J_x^{\frac{1}{p}-\epsilon}J_y^{-\frac{2-p}{2p}+\epsilon}f\|_{L^1_TL^{2}_{xy}}\\& \lesssim2^{-\frac{\epsilon (3j+2k)}{2}}\|J_x^{\frac{1}{p}-\epsilon}J_y^{-\frac{2-p}{2p}+\epsilon}f\|_{L^1_TL^{2}_{xy}}.
  \end{split}
  \end{equation}
  Therefore, $(\ref{eq19})$ and $(\ref{eq20})$ give $(\ref{eqdisp5bis}).$

 \end{proof} 
 \begin{remark} For the fifth order case, we will use the linear estimate 
 \begin{equation*} 
\begin{split} 
 \|u\|_{L^2_TL^{\infty}_{xy}}\lesssim \|J_x^{\frac{3}{2}+\delta}u\|_{L^{\infty}_TL^{2}_{xy}}+\|J_x^{-\frac{3}{2}+\delta}J_y^{1+\delta}u\|_{L^{\infty}_TL^{2}_{xy}}+\|J_x^{\frac{1}{2}+\delta}J_y^{\delta}f\|_{L^{1}_TL^{2}_{xy}}.
 \end{split} 
 \end{equation*} 
  \end{remark}
  
   \section{A Priori Estimates} 
 We are going to bound $f_u(T)=\|u\|_{L^2_TL^{\infty}_{xy}}+\|\partial_x u\|_{L^2_TL^{\infty}_{xy}}+\|\partial_y u\|_{L^2_TL^{\infty}_{xy}}.$ 
 \begin{lemma} 
 Suppose $u \in C([-T,T]:H^{s,s}(M\times\mathbb{T}))$ satisfies the initial value problems (\ref{eqmkp3RT}), (\ref{eqmkp3TT}), (\ref{eqmkp5RT}), with initial data $\phi \in H^{s,s}(M\times\mathbb{T})$ (here, $M$ is either $\mathbb{R}$ or $\mathbb{T}$). Then we have $$\|u\|_{L^{\infty}_TH^{s,s}}\lesssim \|\phi \|_{H^{s,s}}\mbox{exp}(2f_u(T)^2).$$
 \end{lemma} 
 
 \begin{proof} 
  First, $$\beta_f(t)=\|f(t)\|^2_{L^{\infty}_{xy}}+\|\partial_x f(t)\|^2_{L^{\infty}_{xy}}+\|\partial_y f(t)\|^2_{L^{\infty}_{xy}}$$ for a function $f$.
  If we apply to any of (\ref{eqmkp3RT}), (\ref{eqmkp3TT}), (\ref{eqmkp5RT}), the operator $J^s_x$ and then we multiply by $J^s_x u$,  we get by integration by parts
\begin{equation*} 
\begin{split} 
\frac{d}{dt}\|J_x^su\|^2_{L^2_{xy}}&=\int J_x^s uJ^s_x(u^2\partial_x u)=\int J^s_x u[J^s_x (u^2\partial_x u)-u^2J^s_x\partial_x u]+\int u^2J^s_x uJ^s_x \partial_x u \\& \lesssim \|J^s_x u\|^2_{L^2_{xy}}(\|u\|^2_{L^{\infty}_{xy}}+\|u\|_{L^{\infty}_{xy}}\|\partial_x u\|_{L^{\infty}_{xy}})\lesssim \|J^s_x u\|^2_{L^2_{xy}}\beta_u(t)
\end{split}
\end{equation*}
therefore, by Gr{\"o}nwall's inequality, we get that 
\begin{equation}\label{Jx}
\|J_x^s u\|_{L^{\infty}_TL^2_{xy}}\lesssim \|J_x^s\phi\|_{L^2_{xy}} \mbox{exp}(f_u(T)^2).
\end{equation}
  Again,  if we apply to any of (\ref{eqmkp3RT}), (\ref{eqmkp3TT}), (\ref{eqmkp5RT}), the operator $J^s_y$ and then we multiply by $J^s_y u$,  we obtain integrating by parts, 
  \begin{equation*}
  \frac{d}{dt}\|J^s_y u\|^2_{L^2_{xy}}=\int J_y^s u J^s_y(u^2\partial_x u)=\int J^s_y u [J^s_y(u^2\partial_x u)-u^2J^s_y \partial_x u]+\int u^2 J^s_y u J^s_y \partial_x u
  \end{equation*} 
  and we denote $(I)=\int J^s_y u [J^s_y(u^2\partial_x u)-u^2J^s_y \partial_x u]$ and $(II)=\int u^2 J^s_y u J^s_y \partial_x u$. For the first term, by the Kato-Ponce commutator estimates we have 
  \begin{equation*}
  \begin{split} 
  (I) &\lesssim \|J_y^s u\|_{L^{2}_{xy}}\|J_y^s(u^2\partial_x u)-u^2J_y^s\partial_x u\|_{L^2_{xy}}\\&\lesssim \|J^s_y u\|_{L^2_{xy}}[\|\partial_x u\|_{L^2_{xy}}\|J^s_y u\|_{L^2_{xy}}\|u\|_{L^{\infty}_{xy}}+(\|u\|^2_{L^{\infty}_{xy}}+\|u\|_{L^{\infty}_{xy}}\|\partial_y u\|_{L^{\infty}_{xy}})\|J^{s-1}_y\partial_x u\|_{L^2_{xy}}] \\&\lesssim  \|J^s_y u\|^2_{L^2_{xy}}(\|\partial_x u\|_{L^{\infty}_{xy}}\|u\|_{L^{\infty}_{xy}}+\|u\|^2_{L^{\infty}_{xy}}+\|u\|_{L^{\infty}_{xy}}\|\partial_y u\|_{L^{\infty}_{xy}})\\&+\|J^s_y u\|_{L^2_{xy}}\|J^s_x u\|_{L^{2}_{xy}}(\|u\|^2_{L^{\infty}_{xy}}+\|u\|_{L^{\infty}_{xy}}\|\partial_y u\|_{L^{\infty}_{xy}}) \\& \lesssim \|J^s_y\|^2_{L^{2}_{xy}} \beta_u(t)+\|J^s_y u\|_{L^{\infty}_{xy}}\|J^s_x u\|_{L^{\infty}_{xy}}\beta_u(t)\\& \lesssim \|J^s_y u\|^2_{L^{2}_{xy}}\beta_u(t)+\|J^s_y u\|_{L^{2}_{xy}}\|\phi\|_{H^{s,s}} \mbox{exp}(f_u(T)^2)\beta_u(t)
  \end{split}
  \end{equation*}
  By integration by parts, we get that $$(II) \lesssim \|J^s_y u\|^2_{L^{2}_{xy}}\beta_u(t).$$
  Therefore, we get 
$$ \frac{d}{dt}\|J^s_y u\|^2_{L^{2}_{xy}}\lesssim \|J^s_y u\|^2\beta_u(t)+\|J^s_y u\|_{L^{2}_{xy}}\|\phi\|_{H^{s,s}}\mbox{exp}(f_u(T)^2)\beta_u(t)$$ so 
$$\frac{d}{dt}\|J^s_y u\|_{L^{2}_{xy}}\lesssim (\|J^s_y u\|_{L^{2}_{xy}}+\|\phi\|_{H^{s,s}}\mbox{exp}(f_u(T)^2)\beta_u(t)$$ 
hence, by Gr{\"o}nwall's inequality, we get 
\begin{equation}\label{Jy}
\begin{split} 
\|J^s_y u\|_{L^{2}_{xy}} &\lesssim (\|J^s_y\phi\|_{L^{2}_{xy}}+\|\phi\|_{H^{s,s}}\mbox{exp}(f_u(T)^2)\mbox{exp}(f_u(T)^2)\\&\lesssim \|\phi\|_{H^{s,s}}(1+\mbox{exp}(f_u(T)^2))\mbox{exp}(f_u(T)^2)
\end{split} 
\end{equation} 
which yields that $\|u\|_{L^{\infty}_TH^{s,s}}\lesssim \|\phi \|_{H^{s,s}_{xy}}\mbox{exp}(2f_u(T)^2).$ 
 \end{proof} 
 
  \begin{proposition} \label{bound3RT}
  Let $s>2$ and $u_0 \in H^{s,s}(\mathbb{R}\times\mathbb{T}).$ Suppose $u \in C([-T,T]:H^{s,s}(\mathbb{R}\times\mathbb{T}))$ satisfies the IVP (\ref{eqmkp3RT}). 
  Then $u,\partial_x u, \partial_y u \in L^2([-T,T];L^{\infty}(\mathbb{R}\times\mathbb{T})).$ Moreover, $$f_u(T)=\|u\|_{L^2_TL^{\infty}_{xy}}+\|\partial_x u\|_{L^2_TL^{\infty}_{xy}}+\|\partial_y u\|_{L^2_TL^{\infty}_{xy}}\leq C_T$$  for a suitable small $T$, if $\|u_0\|_{H^{s,s}}$ is small enough.
  \end{proposition}
  \begin{proof} 
  From now on $s>2+2\delta.$
  By the linear estimate in Proposition \ref{linearestimate3RT}, we have 
  $$\|u\|_{L^2_TL^{\infty}_{xy}}\lesssim \|J_x^{1+\delta}u\|_{L^{\infty}_TL^{2}_{xy}}+\|J_x^{-1}J_y^{1+\delta}u\|_{L^{\infty}_TL^{2}_{xy}}+\|J_x^{1+\delta}J_y^{\delta}(u^3)\|_{L^{1}_TL^{2}_{xy}}$$ 
  and 
  $$\|\partial_x u\|_{L^{2}_TL^{\infty}_{xy}}\lesssim \|J_x^{2+\delta}u\|_{L^{\infty}_TL^{2}_{xy}}+\|J_y^{1+\delta}u\|_{L^{\infty}_TL^{2}_{xy}}+\|J_x^{2+\delta}J_y^{\delta}(u^3)\|_{L^{1}_TL^{2}_{xy}}.$$ 
  
  We observe that $$\|J_x^{2+\delta}u\|_{L^{\infty}_TL^{2}_{xy}} \leq \|J_x^s u\|_{L^{\infty}_TL^{2}_{xy}}\leq \|u \|_{L^{\infty}_TH^{s,s}_{xy}},$$  
  $$\|J_y^{1+\delta} u\|_{L^{\infty}_TL^{2}_{xy}} \lesssim \|J^s_y u\|_{L^{\infty}_TL^{2}_{xy}}\lesssim \|u\|_{L^{\infty}_TH^{s,s}_{xy}},$$ and 
  $$\|J_x^{2+\delta}J_y^{\delta} (u^3)\|_{L^{\infty}_TL^{2}_{xy}}\lesssim \|J_x^{2+2\delta} (u^3)\|_{L^{\infty}_TL^{2}_{xy}}+\|J_y^{2+2\delta} (u^3)\|_{L^{\infty}_TL^{2}_{xy}}\lesssim \|J_x^{s} (u^3)\|_{L^{\infty}_TL^{2}_{xy}}+\|J_y^{s} (u^3)\|_{L^{\infty}_TL^{2}_{xy}}$$
  so by the corollary of the Kato-Ponce commutator estimates 
  $$\|J^{s}_x (u^3)\|_{L^{2}_{xy}}\lesssim \|J^s_x u\|_{L^{2}_{xy}} \|u\|^2_{L^{\infty}_{xy}}+ \|J^s_x u\|_{L^{2}_{xy}} \|u\|_{L^{\infty}_{xy}} \|\partial_x u\|_{L^{\infty}_{xy}}\lesssim \beta_u(t)\|J^s_x u\|_{L^{2}_{xy}}$$ and 
    $$\|J^{s}_y (u^3)\|_{L^{2}_{xy}}\lesssim \|J^s_y u\|_{L^{2}_{xy}} \|u\|^2_{L^{\infty}_{xy}}+ \|J^s_y u\|_{L^{2}_{xy}} \|u\|_{L^{\infty}_{xy}} \|\partial_y u\|_{L^{\infty}_{xy}}\lesssim \beta_u(t)\|J^s_y u\|_{L^{2}_{xy}}$$
  so therefore 
 $$ \|J^{2+\delta}_xJ_y^{\delta} (u^3)\|_{L^{1}_TL^{2}_{xy}}\lesssim f_u(T)^2\|u\|_{L^{\infty}_TH^{s,s}_{xy}}.$$
  Also,  
  $$\|\partial_y u\|_{L^{2}_TL^{\infty}_{xy}} \lesssim \|J_x^{1+\delta}\partial_y u\|_{L^{\infty}_TL^{2}_{xy}}+\|J^{-1}_xJ_y^{2+\delta}u\|_{L^{\infty}_TL^{2}_{xy}}+\|J_x^{1+\delta}J_y^{\delta}\partial_y(u^3)\|_{L^{1}_TL^{2}_{xy}}.$$
  By the arithmetic-geometric inequality, we also have that 
  $$\|J_x^{1+\delta}\partial_y u\|_{L^{\infty}_TL^{2}_{xy}}\lesssim \|J_x^s u\|_{L^{\infty}_TL^{2}_{xy}}+\|J_y^{\frac{s}{s-1-\delta}} u\|_{L^{\infty}_TL^{2}_{xy}}\lesssim \|J_x^s u\|_{L^{\infty}_TL^{2}_{xy}}+\|J^s_y u\|_{L^{\infty}_TL^{2}_{xy}} \lesssim \|u\|_{L^{\infty}_TH^{s,s}_{xy}}$$ 
  the second inequality being true as $s-1-\delta>1.$ 
  We also have 
  $$\|J_x^{-1}J^{2+\delta}_y u\|_{L^{\infty}_TL^{2}_{xy}} \lesssim \|J_y^s u\|_{L^{\infty}_TL^{2}_{xy}} \lesssim \|u\|_{L^{\infty}_TH^{s,s}_{xy}}.$$ 
  Since  $$\|J_x^{1+\delta}J_y^{\delta}\partial_y (u^3)\|_{L^{2}_{xy}}\lesssim \|J_x^{1+\delta}J_y^{1+\delta}(u^3)\|_{L^{2}_{xy}}\lesssim \|J_x^{2+2\delta} (u^3)\|_{L^{\infty}_TL^{2}_{xy}}+\|J_y^{2+2\delta} (u^3)\|_{L^{\infty}_TL^{2}_{xy}}$$
 the corollary of the Kato-Ponce commutator estimates gives $$\|J_x^{1+\delta}J_y^{\delta}\partial_y (u^3)\|_{L^{2}_{xy}}\lesssim \beta_u(t)(\|J_x^s u\|_{L^2_{xy}}+\|J^s_x u\|_{L^2_{xy}}).$$ So we get from the previous estimates $\|J_x^{1+\delta}\partial_y(u^3)\|_{L^{1}_TL^{2}_{xy}} \lesssim f_u(T)^2\|u\|_{L^{\infty}_TH^{s,s}_{xy}}.$ 
  
  Hence, 
  \begin{equation*} 
  \begin{split}
  f_u(T)&=\|u\|_{L^{2}_TL^{\infty}_{xy}}+\|\partial_x u\|_{L^{2}_TL^{\infty}_{xy}}+\|\partial_y u\|_{L^{2}_TL^{\infty}_{xy}}\\&\lesssim  \|u\|_{L^{\infty}_TH^{s,s}_{xy}}(1+f_u(T)^2).
  \end{split}
  \end{equation*} 
Together by the previous lemma, $$f_u(T) \lesssim \|\phi\|_{H^{s,s}}(1+f_u(T)^2)\mbox{exp}(2f_u(T)^2)$$ and therefore, if $\|\phi\|_{H^{s,s}_{xy}}$ is small enough, by a continuity argument, we get  $f_u(T) \lesssim C$ for $T$ sufficiently small. 
\end{proof} 

 \begin{proposition} \label{bound3TT}
  Let $s>\frac{19}{8}$ and $u_0 \in H^{s,s}(\mathbb{T}\times\mathbb{T}).$ Suppose $u \in C([-T,T]:H^{s,s}(\mathbb{T}\times\mathbb{T}))$ satisfies the IVP (\ref{eqmkp3TT}). 
  Then $u,\partial_x u, \partial_y u \in L^2([-T,T];L^{\infty}(\mathbb{T}\times\mathbb{T})).$ Moreover, $$f_u(T)=\|u\|_{L^2_TL^{\infty}_{xy}}+\|\partial_x u\|_{L^2_TL^{\infty}_{xy}}+\|\partial_y u\|_{L^2_TL^{\infty}_{xy}}\leq C_T$$  for a suitable small $T$, if $\|u_0\|_{H^{s,s}}$ is small enough.
  \end{proposition}
  \begin{proof} 
  From now on $s>\frac{19}{8}+2\delta.$
  
  By the linear estimate in Proposition \ref{linearestimate3TT}, we have 
  $$\|u\|_{L^2_TL^{\infty}_{xy}}\lesssim \|J_x^{\frac{11}{8}+\delta}u\|_{L^{\infty}_TL^{2}_{xy}}+\|J_x^{-\frac{5}{8}}J_y^{1+\delta}u\|_{L^{\infty}_TL^{2}_{xy}}+\|J_x^{\frac{11}{8}+\delta}J_y^{\delta}(u^3)\|_{L^{1}_TL^{2}_{xy}}$$ 
  and 
  $$\|\partial_x u\|_{L^{2}_TL^{\infty}_{xy}}\lesssim \|J_x^{\frac{19}{8}+\delta}u\|_{L^{\infty}_TL^{2}_{xy}}+\|J_x^{\frac{3}{8}}J_y^{1+\delta}u\|_{L^{\infty}_TL^{2}_{xy}}+\|J_x^{\frac{19}{8}+\delta}J_y^{\delta}(u^3)\|_{L^{1}_TL^{2}_{xy}}.$$ 
  
  By the estimates, $$\|J_x^{\frac{19}{8}+\delta}u\|_{L^{\infty}_TL^{2}_{xy}} \leq \|J_x^s u\|_{L^{\infty}_TL^{2}_{xy}}\leq \|u \|_{L^{\infty}_TH^{s,s}_{xy}},$$  
  $$\|J_x^{\frac{3}{8}}J_y^{1+\delta} u\|_{L^{\infty}_TL^{2}_{xy}} \lesssim \|J_x^{\frac{11}{8}+\delta} u\|_{L^{\infty}_TL^{2}_{xy}}+\|J_y^{\frac{11}{8}+\delta} u\|_{L^{\infty}_TL^{2}_{xy}}\lesssim \|u\|_{L^{\infty}_TH^{s,s}_{xy}},$$ and 
  \begin{equation*}
  \begin{split}
  \|J_x^{\frac{19}{8}+\delta}J_y^{\delta} (u^3)\|_{L^{\infty}_TL^{2}_{xy}}&\lesssim \|J_x^{\frac{19}{8}+2\delta} (u^3)\|_{L^{\infty}_TL^{2}_{xy}}+\|J_y^{\frac{19}{8}+2\delta} (u^3)\|_{L^{\infty}_TL^{2}_{xy}}
  \\&\lesssim\|J_x^{s} (u^3)\|_{L^{\infty}_TL^{2}_{xy}}+\|J_y^{s} (u^3)\|_{L^{\infty}_TL^{2}_{xy}}
  \end{split}
  \end{equation*}
  so by the corollary of the Kato-Ponce commutator estimates 
  $$\|J^{\frac{19}{8}+2\delta}_x (u^3)\|_{L^{2}_{xy}}\lesssim \|J^s_x u\|_{L^{2}_{xy}} \|u\|^2_{L^{\infty}_{xy}}+ \|J^s_x u\|_{L^{2}_{xy}} \|u\|_{L^{\infty}_{xy}} \|\partial_x u\|_{L^{\infty}_{xy}}\lesssim \beta_u(t)\|J^s_x u\|_{L^{2}_{xy}}$$ and 
    $$\|J^{\frac{19}{8}+2\delta}_x (u^3)\|_{L^{2}_{xy}}\lesssim \|J^s_y u\|_{L^{2}_{xy}} \|u\|^2_{L^{\infty}_{xy}}+ \|J^s_y u\|_{L^{2}_{xy}} \|u\|_{L^{\infty}_{xy}} \|\partial_y u\|_{L^{\infty}_{xy}}\lesssim \beta_u(t)\|J^s_y u\|_{L^{2}_{xy}}$$
  so therefore 
 $$ \|J^{\frac{19}{8}+\delta}_xJ_y^{\delta} (u^3)\|_{L^{1}_TL^{2}_{xy}}\lesssim f_u(T)^2\|u\|_{L^{\infty}_TH^{s,s}_{xy}}.$$
  Lastly, 
  $$\|\partial_y u\|_{L^{2}_TL^{\infty}_{xy}} \lesssim \|J_x^{\frac{11}{8}+\delta}\partial_y u\|_{L^{\infty}_TL^{2}_{xy}}+\|J^{-\frac{5}{8}}_xJ_y^{2+\delta}u\|_{L^{\infty}_TL^{2}_{xy}}+\|J_x^{\frac{11}{8}+\delta}J_y^{\delta}\partial_y(u^3)\|_{L^{1}_TL^{2}_{xy}}.$$
  By the arithmetic-geometric inequality, we also have that 
  $$\|J_x^{\frac{11}{8}+\delta}\partial_y u\|_{L^{\infty}_TL^{2}_{xy}}\lesssim \|J_x^s u\|_{L^{\infty}_TL^{2}_{xy}}+\|J_y^{\frac{s}{s-\frac{11}{8}-\delta}} u\|_{L^{\infty}_TL^{2}_{xy}}\lesssim \|J_x^s u\|_{L^{\infty}_TL^{2}_{xy}}+\|J^s_y u\|_{L^{\infty}_TL^{2}_{xy}} \lesssim \|u\|_{L^{\infty}_TH^{s,s}_{xy}}$$ 
  the second inequality being true as $s-\frac{11}{8}-\delta>1.$ 
  We also have 
  $$\|J_x^{-\frac{5}{8}}J^{2+\delta}_y u\|_{L^{\infty}_TL^{2}_{xy}} \lesssim \|J_y^s u\|_{L^{\infty}_TL^{2}_{xy}} \lesssim \|u\|_{L^{\infty}_TH^{s,s}_{xy}}.$$ 
  Since  $$\|J_x^{\frac{11}{8}+\delta}J_y^{\delta}\partial_y (u^3)\|_{L^{2}_{xy}}\lesssim \|J_x^{\frac{11}{8}+\delta}J_y^{1+\delta}(u^3)\|_{L^{2}_{xy}}\lesssim \|J_x^{\frac{19}{8}+2\delta} (u^3)\|_{L^{\infty}_TL^{2}_{xy}}+\|J_y^{\frac{19}{8}+2\delta} (u^3)\|_{L^{\infty}_TL^{2}_{xy}}$$
 the corollary of the Kato-Ponce commutator estimates gives $$\|J_x^{\frac{11}{8}+\delta}J_y^{\delta}\partial_y (u^3)\|_{L^{2}_{xy}}\lesssim \beta_u(t)(\|J_x^s u\|_{L^2_{xy}}+\|J^s_x u\|_{L^2_{xy}}).$$ So we get from the previous estimates $\|J_x^{\frac{11}{8}+\delta}\partial_y(u^3)\|_{L^{1}_TL^{2}_{xy}} \lesssim f_u(T)^2\|u\|_{L^{\infty}_TH^{s,s}_{xy}}.$ 
  
  Hence, 
  \begin{equation*}
  \begin{split}
  f_u(T)&=\|u\|_{L^{2}_TL^{\infty}_{xy}}+\|\partial_x u\|_{L^{2}_TL^{\infty}_{xy}}+\|\partial_y u\|_{L^{2}_TL^{\infty}_{xy}}\\&\lesssim  \|u\|_{L^{\infty}_TH^{s,s}_{xy}}(1+f_u(T)^2).
  \end{split}
  \end{equation*} 
Together with the previous lemma $$f_u(T) \lesssim \|\phi\|_{H^{s,s}}(1+f_u(T)^2)\mbox{exp}(2f_u(T)^2)$$ and therefore, if $\|\phi\|_{H^{s,s}_{xy}}$ is small enough, by a continuity argument, we get  $f_u(T) \lesssim C$ for $T$ sufficiently small. 
\end{proof} 

 \begin{proposition} \label{bound5RT}
  Let $s>\frac{5}{2}$ and $u_0 \in H^{s,s}(\mathbb{R}\times\mathbb{T}).$ Suppose $u \in C([-T,T]:H^{s,s}(\mathbb{R}\times\mathbb{T}))$ satisfies the IVP (\ref{eqmkp5RT}). 
  Then $u,\partial_x u, \partial_y u \in L^2([-T,T];L^{\infty}(\mathbb{R}\times\mathbb{T})).$ Moreover, $$f_u(T)=\|u\|_{L^2_TL^{\infty}_{xy}}+\|\partial_x u\|_{L^2_TL^{\infty}_{xy}}+\|\partial_y u\|_{L^2_TL^{\infty}_{xy}}\leq C_T$$  for a suitable small $T$, if $\|u_0\|_{H^{s,s}}$ is small enough.
  \end{proposition}
  \begin{proof} 
We take $0<\delta<s-\frac{5}{2}.$ By the linear estimate in Proposition \ref{linearestimate5RT}, we have 
  \begin{equation*} 
\begin{split} 
 \|u\|_{L^2_TL^{\infty}_{xy}}\lesssim \|J_x^{\frac{3}{2}+\delta}u\|_{L^{\infty}_TL^{2}_{xy}}+\|J_x^{-\frac{3}{2}+\delta}J_y^{1+\delta}u\|_{L^{\infty}_TL^{2}_{xy}}+\|J_x^{\frac{1}{2}+\delta}J_y^{\delta}(u^3)\|_{L^{1}_TL^{2}_{xy}}
 \end{split} 
 \end{equation*} 
 with $\|J_x^{\frac{3}{2}+\delta}u\|_{L^{\infty}_TL^{2}_{xy}}\lesssim \|u\|_{L^{\infty}_TH^{s,s}_{xy}}$ and $\|J_x^{-\frac{3}{2}+\delta}J_y^{1+\delta}\|_{L^{\infty}_TL^{2}_{xy}}\lesssim  \|u\|_{L^{\infty}_TH^{s,s}_{xy}}.$ For the third term of the linear estimate we have $\|J_x^{\frac{1}{2}+\delta}J_y^{\delta}(u^3)\|_{L^{2}_{xy}}\leq \|J_x^{\frac{1}{2}+2\delta}(u^3)\|_{L^{2}_{xy}}+\|J_y^{\frac{1}{2}+2\delta}(u^3)\|_{L^{2}_{xy}}$
 so  by the Kato-Ponce commutator estimates, 
 $$\|J_x^{\frac{1}{2}+2\delta}(u^3)\|_{L^{2}_{xy}}\lesssim \|J_x^{\frac{1}{2}+2\delta}u\|_{L^{2}_{xy}}\|u\|^2_{L^{\infty}_{xy}}\mbox{ and  }\|J_y^{\frac{1}{2}+2\delta}(u^3)\|_{L^{2}_{xy}}\lesssim \|J_y^{\frac{1}{2}+\delta}u\|_{L^{2}_{xy}}\|u\|^2_{L^{\infty}_{xy}}$$ 
 and so $$\|J_x^{\frac{1}{2}+\delta}(u^3)\|_{L^1_TL^{2}_{xy}}\lesssim \|J_x^{\frac{1}{2}+\delta}u\|_{L^{\infty}_TL^{2}_{xy}}\|u\|^2_{L^{2}_TL^{\infty}_{xy}}\lesssim f_u(T)^2 \|u\|_{L^{\infty}_TH^{s,s}_{xy}}$$
 and finally we get $$\|u\|_{L^{2}_TL^{\infty}_{xy}}\lesssim (1+f_u(T)^2)\|u\|_{L^{\infty}_TH^{s,s}_{xy}}\lesssim \|\phi\|_{H^{s,s}_{xy}}\mbox{exp}(2f_u(T)^2)(1+f_u(T)^2).$$
 
 Now, we look at the second term of $f_u(T).$ By the linear estimate applied to $\partial_x u$, we have 
   \begin{equation*} 
\begin{split} 
 \|\partial_x u\|_{L^2_TL^{\infty}_{xy}}\lesssim \|J_x^{\frac{5}{2}+\delta}u\|_{L^{\infty}_TL^{2}_{xy}}+\|J_x^{-\frac{1}{2}+\delta}J_y^{1+\delta}u\|_{L^{\infty}_TL^{2}_{xy}}+\|J_x^{\frac{3}{2}+\delta}J_y^{\delta}(u^3)\|_{L^{1}_TL^{2}_{xy}}
 \end{split} 
 \end{equation*} 
 with $\|J_x^{\frac{5}{2}+\delta}u\|_{L^{\infty}_TL^{2}_{xy}}\lesssim \|u\|_{L^{\infty}_TH^{s,s}_{xy}}$ and $\|J_x^{-\frac{1}{2}+\delta}J_y^{1+\delta}\|_{L^{\infty}_TL^{2}_{xy}}\lesssim  \|u\|_{L^{\infty}_TH^{s,s}_{xy}}.$ For the third term of the linear estimate we have, 
  $\|J_x^{\frac{3}{2}+\delta}J_y^{\delta}(u^3)\|_{L^{2}_{xy}}\leq \|J_x^{\frac{3}{2}+2\delta}(u^3)\|_{L^{2}_{xy}}+\|J_y^{\frac{3}{2}+2\delta}(u^3)\|_{L^{2}_{xy}}$
 by the Kato-Ponce commutator estimates,
 $$\|J_x^{\frac{3}{2}+\delta}(u^3)\|_{L^{2}_{xy}}\lesssim \|J_x^{\frac{3}{2}+\delta}u\|_{L^{2}_{xy}}\|u\|^2_{L^{\infty}_{xy}} \mbox{ and }\|J_y^{\frac{3}{2}+\delta}(u^3)\|_{L^{2}_{xy}}\lesssim \|J_y^{\frac{3}{2}+\delta}u\|_{L^{2}_{xy}}\|u\|^2_{L^{\infty}_{xy}}$$ 
 and so $$\|J_x^{\frac{3}{2}+\delta}J_y^{\delta}(u^3)\|_{L^1_TL^{2}_{xy}}\lesssim f_u(T)^2 \|u\|_{L^{\infty}_TH^{s,s}_{xy}}.$$
 Finally we get $$\|\partial_x u\|_{L^{2}_TL^{\infty}_{xy}}\lesssim (1+f_u(T)^2)\|u\|_{L^{\infty}_TH^{s,s}_{xy}}\lesssim \|\phi\|_{H^{s,s}_{xy}}\mbox{exp}(2f_u(T)^2)(1+f_u(T)^2).$$
 
 Now, for the final term of $f_u(T)$, by the linear estimate applied to $\partial_y u$, we have 
   \begin{equation*} 
\begin{split} 
 \|\partial_y u\|_{L^2_TL^{\infty}_{xy}}\lesssim \|J_x^{\frac{1}{2}+\delta}J^1_y u\|_{L^{\infty}_TL^{2}_{xy}}+\|J_x^{-\frac{3}{2}+\delta}J_y^{2+\delta}u\|_{L^{\infty}_TL^{2}_{xy}}+\|J_x^{\frac{1}{2}+\delta}\partial_y (u^3)\|_{L^{1}_TL^{2}_{xy}}
 \end{split} 
 \end{equation*} 
 with $$\|J_x^{\frac{1}{2}+\delta}J^1_y u\|_{L^{\infty}_TL^{2}_{xy}}\lesssim \|J_x^{\frac{5}{2}+\delta} u\|_{L^{\infty}_TL^{2}_{xy}}+\|J_y^{\frac{5}{4}+\frac{\delta}{2}} u\|_{L^{\infty}_TL^{2}_{xy}}\lesssim \|u\|_{L^{\infty}_TH^{s,s}_{xy}}$$ and $\|J_x^{-\frac{3}{2}+\delta}J_y^{2+\delta}u\|_{L^{\infty}_TL^{2}_{xy}}\lesssim \|u\|_{L^{\infty}_TH^{s,s}_{xy}}.$ For the third term in the linear estimate, we have 
 $$\|J_x^{\frac{1}{2}+\delta}J_y^{\delta}\partial_y (u^3)\|_{L^{2}_{xy}}\leq \|J_x^{\frac{1}{2}+\delta}J_y^{1+\delta} (u^3)\|_{L^{2}_{xy}}\leq \|J_x^{\frac{3}{2}+2\delta}(u^3)\|_{L^{2}_{xy}}+\|J_y^{\frac{3}{2}+2\delta}(u^3)\|_{L^{2}_{xy}}$$ 
so by the same reasoning, after applying the Kato-Ponce commutator estimates we get $$ \|J_x^{1+\delta}J_y^{\delta}(u^2\partial_y u)\|_{L^{\infty}_TL^2_{xy}}\lesssim \|u\|_{L^{\infty}_TH^{s,s}_{xy}}f_u(T)^2.$$
 Finally, we get $$\|\partial_y u\|_{L^{2}_TL^{\infty}_{xy}}\lesssim (1+f_u(T)^2)\|u\|_{L^{\infty}_TH^{s,s}_{xy}}\lesssim \|\phi\|_{H^{s,s}_{xy}}\mbox{exp}(2f_u(T)^2)(1+f_u(T)^2).$$
 
 All in all, we have that $f_u(T)\lesssim \|\phi\|_{H^{s,s}_{xy}}\mbox{exp}(2f_u(T)^2)(1+f_u(T)^2)$, and if $\|\phi\|_{H^{s,s}_{xy}}$ is small, then by a continuity argument we get that $f_u(T)\leq C$ if $T$ is sufficiently small. 
 \end{proof} 

 \section{Existence and Uniqueness} \label{Local Well-Posedness} 
 We start by stating a well-known local-wellposedness result from Iorio and Nunes (see \cite{iorio}, Section 4):
 \begin{lemma}
 Assume $\phi \in H^{\infty}(M\times \mathbb{T})$, where $M$ is either $\mathbb{R}$ or $\mathbb{T}$. Then there is $T = T(\|\phi\|_{H^3}) > 0$ and a solution $u \in C([-T, T ] : H^{\infty}(\mathbb{M}\times \mathbb{T}) )$ of the initial value problem 
 \begin{equation*}
\begin{cases} \partial_t u+\partial^3_x u-\partial_x^{-1}\partial_y^2 u + u^2\partial_x u=0,\\
u(0,x,y)=\phi(x,y).
\end{cases}
\end{equation*} 
 The proof for $\mathbb{R}\times \mathbb{T}$ and $\mathbb{T}\times\mathbb{T}$ is the same as the proof in \cite{iorio} for $\mathbb{R}\times \mathbb{R}.$ 
  \end{lemma}
 For the fifth order case we have the following result inspired also by Iorio and Nunes' \cite{iorio} 
  \begin{lemma}
 Assume $\phi \in H^{\infty}(\mathbb{R}\times \mathbb{T})$. Then there is $T = T(\|\phi\|_{H^5}) > 0$ and a solution $u \in C([-T, T ] : H^{\infty}(\mathbb{R}\times \mathbb{T}) )$ of the initial value problem 
 \begin{equation*}
\begin{cases} \partial_t u-\partial^5_x u-\partial_x^{-1}\partial_y^2 u + u^2\partial_x u=0,\\
u(0,x,y)=\phi(x,y).
\end{cases}
\end{equation*} 
  \end{lemma}
We proceed to prove the local well-posedness result. 
\begin{theorem} 
The initial value problem (\ref{eqmkp3RT}) is locally well-posed in $H^{s,s}(\mathbb{R}\times\mathbb{T}),s>2.$ More precisely, given $u_0 \in H^{s,s}(\mathbb{R}\times\mathbb{T}),s>2$, there exists $T=T(\|u_0\|_{H^{s,s}})$ and a unique solution $u$ to the IVP such that $u \in C([0,T]:H^{s,s}(\mathbb{R}\times \mathbb{T}))$, $u, \partial_x u,\partial_y u\in L^2_TL^{\infty}_{xy}.$ Moreover, the mapping $u_0 \rightarrow  u$ in $C([0,T]:H^{s,s}(\mathbb{R}\times \mathbb{T}))$ is continuous. 
\end{theorem} 

\begin{theorem} 
The initial value problem (\ref{eqmkp3TT}) is locally well-posedness in $H^{s,s}(\mathbb{T}\times\mathbb{T}),s>\frac{19}{8}.$ More precisely, given $u_0 \in H^{s,s}(\mathbb{T}\times\mathbb{T}),s>\frac{19}{8}$, there exists $T=T(\|u_0\|_{H^{s,s}})$ and a unique solution $u$ to the IVP such that $u \in C([0,T]:H^{s,s}(\mathbb{T}\times \mathbb{T}))$, $u, \partial_x u,\partial_y u\in L^2_TL^{\infty}_{xy}.$ Moreover, the mapping $u_0 \rightarrow u $ in $C([0,T]:H^{s,s}(\mathbb{T}\times \mathbb{T}))$ is continuous. 
\end{theorem} 

\begin{theorem} 
 The initial value problem (\ref{eqmkp5RT}) is locally well-posed in $H^{s,s}(\mathbb{R}\times\mathbb{T}),s>\frac{5}{2}.$ More precisely, given $u_0 \in H^{s,s}(\mathbb{R}\times\mathbb{T}),s>\frac{5}{2}$, there exists $T=T(\|u_0\|_{H^{s,s}})$ and a unique solution $u$ to the IVP such that $u \in C([0,T]:H^{s,s}(\mathbb{R}\times \mathbb{T}))$, $u, \partial_x u,\partial_y u\in L^2_TL^{\infty}_{xy}.$ Moreover, the mapping $u_0 \rightarrow \in C([0,T]:H^{s,s}(\mathbb{R}\times \mathbb{T}))$ is continuous. 
 \end{theorem} 
 
 We present the proof for existence and uniqueness in the case of the third order mKP-I on $\mathbb{R}\times\mathbb{T}$, since the other two cases are similar.  
\begin{proof}
Let $u_0 \in H^{s,s}(\mathbb{R}\times\mathbb{T})$ and fixed $u_{0,\epsilon}\in H^{s,s}(\mathbb{R}\times\mathbb{T})\cap H^{\infty}_{-1}(\mathbb{R}\times\mathbb{T})$ such that $\|u_0-u_{0,\epsilon}\|_{H^{s,s}}\rightarrow 0$ and $\|u_{0,\epsilon}\|_{H^{s,s}}\leq 2\|u_0\|_{H^{s,s}}.$ 

We know by the Iorio-Nunes local well-posedness result that $u_{0,\epsilon}$ gives a unique solution $u_{\epsilon}.$ We have by the a priori bound that $\|u_{\epsilon}\|_{L^2_TL^{\infty}_{xy}}+\|\partial_x u_{\epsilon}\|_{L^2_TL^{\infty}_{xy}}+\|\partial_y u_{\epsilon}\|_{L^2_TL^{\infty}_{xy}} \leq C_T$ and by the previous result, $\mbox{sup}_{0<t<T}\|u_{\epsilon}\|_{H^{s,s}}\leq C_T.$ 

 Henceforth, 
  \begin{equation*} 
  \begin{split} 
  \partial_t \|u_\varepsilon -u_{\varepsilon'}\|^2_{L^2}&=\int(u_\varepsilon-u_{\varepsilon'}) \partial_x(\frac{u_\varepsilon^3}{3}-\frac{u_{\varepsilon'}^3}{3})\\&=\int\partial_x(u_\varepsilon-u_{\varepsilon'})\cdot(u_\varepsilon-u_{\varepsilon'})\frac{u_{\varepsilon}^2+u_{\varepsilon}u_{\varepsilon'}+u_{\varepsilon'}^2}{3}=\\ 
  \\&=\int(u_\varepsilon-u_{\varepsilon'})^2\partial_x[\frac{u_{\varepsilon}^2+u_{\varepsilon}u_{\varepsilon'}+u_{\varepsilon'}^2}{3}]\\&\leq \| u_\varepsilon -u_{\varepsilon'}\|^2_{L^2}\big(\|u_\varepsilon\|^2_{L^{\infty}_{xy}}+\|\partial_x u_\varepsilon\|^2_{L^{\infty}_{xy}}+\|u_{\varepsilon'}\|^2_{L^{\infty}_{xy}}+\|\partial_x u_{\varepsilon'}\|^2_{L^{\infty}_{xy}})\\& \leq (\beta_{u_{\varepsilon}}(t)+\beta_{u_{\varepsilon'}}(t))\|u_\varepsilon-u_{\varepsilon'}\|^2_{L^2}.
  \end{split}
  \end{equation*} 
   and by Gr{\"o}nwall's inequality $$\|u_{\epsilon}-u_{\epsilon'}\|^2_{L^{\infty}_TL^2_{xy}}\lesssim_T\|u_{0,\epsilon}-u_{0,\epsilon'}\|^2_{L^2_{xy}},$$ hence $\mbox{sup}_{0<t<T}\|u_{\epsilon}-u_{\epsilon'}\|_{L^2_{xy}}\rightarrow 0,$ hence we can find $u \in C([0,T]:H^{s',s'}(\mathbb{R}\times\mathbb{T}))\cap L^{\infty}([0,T]:H^{s,s}(\mathbb{R}\times\mathbb{T}))$ with $s'<s.$ The fact that $u$ is a solution of the IVP is clear now. Uniqueness also comes from the previous Gr{\"o}nwall's inequality. 
\end{proof} 

\section{Continuity with respect to time} 
 We proceed by a standard Bona-Smith argument (\cite{bona}). 
\begin{definition} 
For $\phi \in H^{s,s}(\mathbb{R}\times\mathbb{T})$ with $s>2$, let $\phi_k=P_{(3)}^k\phi$ where $\widehat{P^k_{(3)}g}(\xi,n)=\widehat{g}(\xi,n)\cdot 1_{[0,k]}(|\xi|)\cdot1_{[0,k]}(|n|).$ Let $$h^{(3)}_{\phi}(k)=\Big[\sum_{n \in \mathbb{Z}}\int _{|\xi|+|n|\geq k}|\widehat{\phi}(\xi,n)|^2[(1+\xi^2)^s+(1+n^2)^s]d\xi\Big]^{\frac{1}{2}}. $$ 
\end{definition}
Clearly, $h^{(3)}_{\phi}$ is nondecreasing in $k$ and $\lim_{k \rightarrow \infty}h^{(3)}_{\phi}(k)=0.$
By Plancherel, 
\begin{equation*}
\begin{split}
\|\phi-\phi_{k}\|_{L^2_{xy}}=\|\widehat{\phi}-\widehat{\phi}_k\|_{L^2_{xy}}&=\Big[\sum_{|n|\geq k}\int_{|\xi|\geq k} |\widehat{\phi}(\xi,n)|^2d\xi\Big]^2\\&\leq \Big[\sum_{n \in \mathbb{Z}}\int_{|\xi|+|n| \geq k}|\widehat{\phi}(\xi,n)|\frac{[(1+\xi^2)^s+(1+n^2)^s]^{\frac{1}{2}}}{k^{2s}}\Big] \lesssim k^{-s}h^{(3)}_{\phi}(k)
\end{split}
\end{equation*}

\begin{definition} 
For $\phi \in H^{s,s}(\mathbb{T}\times\mathbb{T})$ with $s>\frac{19}{8}$, let $\phi_k=\widetilde{P}_{(3)}^k\phi$ where $\widehat{\widetilde{P}^k_{(3)}g}(m,n)=\widehat{g}(m,n)\cdot 1_{[0,k]}(|m|)\cdot1_{[0,k]}(|n|).$ Let $\widetilde{h}^{(3)}_{\phi}(k)=\Big[\sum\limits_{m \in \mathbb{Z}}\sum\limits_{\substack{n \in \mathbb{Z}\\|m|+|n|\geq k}}|\widehat{\phi}(m,n)|^2[(1+m^2)^s+(1+n^2)^s]\Big]^{\frac{1}{2}}. $
\end{definition} 
 Clearly, $\widetilde{h}^{(3)}_{\phi}$ is nondecreasing in $k$ and $\lim_{k \rightarrow \infty}\widetilde{h}^{(3)}_{\phi}(k)=0.$
By Plancherel, 
\begin{equation*}
\begin{split}
\|\phi-\phi_{k}\|_{L^2_{xy}}=\|\widehat{\phi}-\widehat{\phi}_k\|_{L^2_{xy}}&=\Big[\sum_{|n|\geq k}\sum_{|m|\geq k} |\widehat{\phi}(m,n)|^2\Big]^{\frac{1}{2}}\\&\leq \Big[\sum_{m,n \in \mathbb{Z},}\sum_{|m|+|n| \geq k}|\widehat{\phi}(m,n)|\frac{[(1+m^2)^s+(1+n^2)^s]^{\frac{1}{2}}}{k^{2s}}\Big] \lesssim k^{-s}\widetilde{h}^{(3)}_{\phi}(k)
\end{split}
\end{equation*}

\begin{definition}
For $\phi \in H^{s,s}(\mathbb{R}\times\mathbb{T})$ with $s>\frac{5}{2}$, let $\phi_k=P_{(5)}^k\phi$ where $\widehat{P^k_{(5)}g}(\xi,n)=\widehat{g}(\xi,n)\cdot 1_{[0,k]}(|\xi|)\cdot1_{[0,k]}(|n|).$ Let $h^{(5)}_{\phi}(k)=[\sum_{n \in \mathbb{Z}}\int _{|\xi|+|n|\geq k}|\hat{\phi}(\xi,n)|^2[(1+\xi^2)^s+(1+n^2)^s]d\xi]^{\frac{1}{2}}. $ 
\end{definition} 
Clearly, $h^{(5)}_{\phi}$ is nondecreasing in $k$ and $\lim_{k \rightarrow \infty}h^{(5)}_\phi(k)=0.$
By Plancherel, 
\begin{equation*}
\begin{split}
\|\phi-\phi_{k}\|_{L^2_{xy}}=\|\widehat{\phi}-\widehat{\phi}_k\|_{L^2_{xy}}&=[\sum_{|n|\geq k} \int_{|\xi|\geq k}|\hat{\phi}(\xi,n)|^2d\xi]^2\\&\leq [\sum_{n \in \mathbb{Z}}\int_{|\xi|+|n| \geq k}|\hat{\phi}(\xi,n)|\frac{[(1+\xi^2)^s+(1+n^2)^s]^{\frac{1}{2}}}{k^{2s}}] \lesssim k^{-s}h^{(5)}_{\phi}(k)
\end{split}
\end{equation*}

In all the cases, from their respective definitions, we have that, if $p\geq s$, then  
$$\|J^p_x\phi_k\|_{L^2_{xy}}\lesssim C(T,\|\phi\|_{H^{s,s}})k^{p-s}\mbox{ and } \|J^p_y\phi_k\|_{L^2_{xy}}\lesssim C(T,\|\phi\|_{H^{s,s}})k^{p-s}.$$
Since $\phi_k\in H^{\infty}$, by local well-posedness result of Iorio and Nunes, they give rise to unique solutions $u_k$ in $H^{\infty}.$ The above estimates together with (\ref{Jx}) and (\ref{Jy}), if $p \geq s$, we also have 
\begin{equation}\label{eqx}
\|J^p_x u_k\|_{L^{\infty}_TL^{2}_{xy}}\leq C(T,\|\phi\|_{H^{s,s}})k^{p-s}
\end{equation}
 and 
 \begin{equation}\label{eqy}
 \|J^p_y u_k\|_{L^{\infty}_TL^{2}_{xy}}\leq C(T,\|\phi\|_{H^{s,s}})k^{p-s}.
 \end{equation} 
Denote $\omega=u_k-u_{k'}$ with $k<k'.$ 
Now choose $0\leq q\leq s.$ By using that $\|\phi-\phi_{k}\|_{L^2_{xy}}\lesssim k^{-s}h^{(3)}_{\phi}(k)$ for the third order $\mathbb{R}\times\mathbb{T}$ case, $\|\phi-\phi_{k}\|_{L^2_{xy}}\lesssim k^{-s}\widetilde{h}^{(3)}_{\phi}(k)$ for the third order $\mathbb{T}\times\mathbb{T}$ case and $\|\phi-\phi_{k}\|_{L^2_{xy}}\lesssim k^{-s}h^{(5)}_{\phi}(k)$ for the fifth order $\mathbb{R}\times\mathbb{T}$ case, together with the interpolation inequality, 
$$\|J_x^q\omega\|_{L^{\infty}_TL^2_{xy}}\leq \|J_x^s\omega\|^{\frac{q}{s}}_{L^{\infty}_TL^2_{xy}}\|\omega\|^{1-\frac{q}{s}}_{L^{\infty}_TL^2_{xy}}\lesssim \|\omega\|^{1-\frac{q}{s}}_{L^{\infty}_TL^2_{xy}}$$
it yields 
\begin{equation}\label{Jxomega3RT}
\|J_x^q\omega\|_{L^{\infty}_TL^2_{xy}}\lesssim k^{q-s}h^{(3)}_{\phi}(k)^{1-\frac{q}{s}} 
\end{equation} 
respectively, 
\begin{equation}\label{Jxomega3TT}
\|J_x^q\omega\|_{L^{\infty}_TL^2_{xy}}\lesssim k^{q-s}\widetilde{h}^{(3)}_{\phi}(k)^{1-\frac{q}{s}}
\end{equation} 
and 
\begin{equation}\label{Jxomega5RT}
\|J_x^q\omega\|_{L^{\infty}_TL^2_{xy}}\lesssim k^{q-s}h^{(5)}_{\phi}(k)^{1-\frac{q}{s}}.
\end{equation} 
Similarly, we get results for $J_y$, more precisely,
\begin{equation}\label{Jyomega3RT}
\|J_y^q\omega\|_{L^{\infty}_TL^2_{xy}}\lesssim k^{q-s}h^{(3)}_{\phi}(k)^{1-\frac{q}{s}}
\end{equation} 
respectively, 
\begin{equation}\label{Jyomega3TT}
\|J_y^q\omega\|_{L^{\infty}_TL^2_{xy}}\lesssim k^{q-s}\widetilde{h}^{(3)}_{\phi}(k)^{1-\frac{q}{s}}
\end{equation}
and
\begin{equation}\label{Jyomega5RT}
\|J_y^q\omega\|_{L^{\infty}_TL^2_{xy}}\lesssim k^{q-s}h^{(5)}_{\phi}(k)^{1-\frac{q}{s}}. 
\end{equation}

 \begin{lemma} \label{JxJy}
 We have the following estimates: 
 \begin{itemize} 
 \item [a)] 
  \begin{equation*}
 \begin{split}
 \|J^s_x \omega\|_{L^{\infty}_TL^2_{xy}}& \lesssim \mbox{exp}(\frac{1}{2}f_{\omega}(T)^2+\frac{1}{2}f_{u_{k}}(T)^2)\Big[\|J^s_x \omega(0)\|_{L^{2}_{xy}}
 \\&+\|J^s_x\omega\|_{L^{\infty}_TL^{2}_{xy}}\|J^{s+1}_x u_{k}\|_{L^{\infty}_TL^{2}_{xy}}(\|\omega\|_{L^2_TL^{\infty}_{xy}}^2+\|\omega\|_{L^2_TL^{\infty}_{xy}}\|u_{k}\|_{L^2_TL^{\infty}_{xy}})
 \\&+\|J^s_x\omega\|_{L^{\infty}_TL^{2}_{xy}}\|J^{s}_x u_{k}\|_{L^{\infty}_TL^{2}_{xy}}(\|\omega\|_{L^2_TL^{\infty}_{xy}}+\|\partial_x\omega\|_{L^2_TL^{\infty}_{xy}})\|u_k\|_{L^2_TL^{\infty}_{xy}}
 \\&+\|J^s_x\omega\|_{L^{\infty}_TL^{2}_{xy}}\|J^{s}_x u_{k}\|_{L^{\infty}_TL^{2}_{xy}}\|\omega\|_{L^2_TL^{\infty}_{xy}}(\|\partial_x u_k\|_{L^2_TL^{\infty}_{xy}}+\|\partial_x \omega\|_{L^2_TL^{\infty}_{xy}})
 \\&+\|J^s_x\omega\|_{L^{\infty}_TL^{2}_{xy}}\|J^{s}_x u_{k}\|_{L^{\infty}_TL^{2}_{xy}} \|\omega\|_{L^2_TL^{\infty}_{xy}}^2\Big].
 \end{split} 
 \end{equation*}
 \item[b)]   
   \begin{equation*}
 \begin{split}
 \|J^s_y \omega\|_{L^{\infty}_TL^2_{xy}}& \lesssim \mbox{exp}(\frac{1}{2}f_{\omega}(T)^2+\frac{1}{2}f_{u_{k}}(T)^2)\Big[\|J^s_y \omega(0)\|_{L^{2}_{xy}}
 \\&+\|J^s_y\omega\|_{L^{\infty}_TL^{2}_{xy}}\|J^s_y u_{k}\|_{L^{\infty}_{T}L^{2}_{xy}}\|\omega\|_{L^2_TL^{\infty}_{xy}}(\|\partial_x u_{k}\|_{L^2_TL^{\infty}_{xy}}+\|\partial_y u_{k}\|_{L^2_TL^{\infty}_{xy}})
 \\&+\|J^s_y\omega\|_{L^{\infty}_TL^{2}_{xy}}\|J^s_y u_{k}\|_{L^{\infty}_{T}L^{2}_{xy}}\Big(\|u_{k}\|_{L^2_TL^{\infty}_{xy}}\|\partial_y\omega\|_{L^2_TL^{\infty}_{xy}}+f_{\omega}(T)^2\Big)
 \\&+\|J^s_y\omega\|_{L^{\infty}_TL^{2}_{xy}}\|J^s_x u_{k}\|_{L^{\infty}_{T}L^{2}_{xy}}\|\omega\|_{L^2_TL^{\infty}_{xy}}(\| u_{k}\|_{L^2_TL^{\infty}_{xy}}+\|\partial_y u_{k}\|_{L^2_TL^{\infty}_{xy}})
\\&+\|J^s_y\omega\|_{L^{\infty}_TL^{2}_{xy}}\|J^s_x u_{k}\|_{L^{\infty}_{T}L^{2}_{xy}}\Big(\|u_{k}\|_{L^2_TL^{\infty}_{xy}}\|\partial_y\omega\|_{L^2_TL^{\infty}_{xy}}+f_{\omega}(T)^2\Big)
 \\&+\|J^s_y\omega\|_{L^{\infty}_TL^{2}_{xy}}\|J^s_x \omega\|_{L^{\infty}_{T}L^{2}_{xy}}\|\omega\|_{L^2_TL^{\infty}_{xy}}(\| u_{k}\|_{L^2_TL^{\infty}_{xy}}+\|\partial_y u_{k}\|_{L^2_TL^{\infty}_{xy}})
  \\&+\|J^s_y\omega\|_{L^{\infty}_TL^{2}_{xy}}\|J^s_x \omega\|_{L^{\infty}_{T}L^{2}_{xy}}\Big(\|u_{k}\|_{L^2_TL^{\infty}_{xy}}\|\partial_y\omega\|_{L^2_TL^{\infty}_{xy}}+f_{\omega}(T)^2\Big)
  \\&+\|J^s_y\omega\|_{L^{\infty}_TL^{2}_{xy}}\|J^s_x \omega\|_{L^{\infty}_{T}L^{2}_{xy}}\|u_{k}\|_{L^2_TL^{\infty}_{xy}}\|\partial_y u_k\|_{L^2_TL^{\infty}_{xy}}
 \\&+\|J^s_y\omega\|_{L^{\infty}_TL^{2}_{xy}}\|J^{s+1}_y u_{k}\|_{L^{\infty}_{T}L^{2}_{xy}}(\|\omega\|^2_{L^{\infty}_{T}L^{2}_{xy}}+\|\omega\|_{L^{\infty}_{T}L^{2}_{xy}}\|u_{k}\|_{L^{\infty}_{T}L^{2}_{xy}})\Big].
  \end{split} 
 \end{equation*}
 \end{itemize} 
 \end{lemma} 
 
\begin{proof} 
\begin{equation} \label{eqomega3T}
\partial_t \omega+\partial_x^3\omega-\partial_x^{-1}\partial_y^2\omega+\omega^2\partial_x \omega+3u_{k}^2\partial_x \omega +3u_{k}\omega \partial_x u_{k}-3u_{k}\omega \partial_x \omega -3\omega^2\partial_x u_{k}=0.
\end{equation} 
\begin{equation} \label{eqomega5T}
\partial_t \omega-\partial_x^5\omega-\partial_x^{-1}\partial_y^2\omega+\omega^2\partial_x \omega+3u_{k}^2\partial_x \omega +3u_{k}\omega \partial_x u_{k}-3u_{k}\omega \partial_x \omega -3\omega^2\partial_x u_{k}=0.
\end{equation} 
\begin{itemize} 
\item[a)] 
We apply $J_x^s$ to (\ref{eqomega3T}) and then we multiply by $J_x^s\omega$, in order to get 
  \begin{equation*}
  \begin{split}
  \frac{d}{dt}\|J_x^s\omega\|^2_{L^2}=&\int J_x^s(\omega^2\partial_x \omega)J_x^s\omega +3\int J_x^s(u_{k}^2\partial_x \omega)J^s_x\omega+3\int J_x^s(u_{k}\omega\partial_xu_{k})J_x^s\omega\\&-3\int J_x^s(u_{k}\omega\partial_x \omega)J_x^s\omega -3\int J_x^s(\omega^2\partial_x u_{k})J_x^s\omega
  \end{split} 
  \end{equation*} 
  and we will analyze each term in the sum. 
  
  We have $(I)=\int J_x^s(\omega^2\partial_x \omega)J_x^s\omega$, $(II)=\int J_x^s(u_{k}^2\partial_x \omega)J^s_x\omega$, $(III)=\int J_x^s(u_{k}\omega\partial_x u_{k})J_x^s\omega$, $(IV)=\int J_x^s(u_{k}\omega\partial_x \omega)J_x^s\omega$ and $(V)=\int J_x^s(\omega^2\partial_x u_{k})J_x^s\omega$. 
  
  For $(I)=\int J_x^s(\omega^2\partial_x \omega)J_x^s\omega=\int[J_x^s(\omega^2\partial_x \omega)-\omega^2J_x^s\partial_x \omega]J_x^s\omega+\int\omega^2J_x^s\partial_x \omega J_x^s\omega$, and we will denote $(I)_1=\int[J_x^s(\omega^2\partial_x \omega)-\omega^2J_x^s\partial_x \omega]J_x^s\omega$ and $(I)_2=\int\omega^2J_x^s\partial_x \omega J_x^s\omega$. For the first one, we have by the Kato-Ponce commutator estimate
  \begin{equation*}
  \begin{split} 
  (I)_1 &\leq \|J_x^s \omega\|_{L^2_{xy}}\|J^s_x(\omega^2\partial_x \omega)-\omega^2 J_x^s\partial_x \omega\|_{L^2_{xy}}\\
  &\leq \|J^s_x\omega\|_{L^2_{xy}} \|\omega\|_{L^{\infty}_{xy}}\Big[\|\partial_x \omega\|_{L^{\infty}_{xy}}\cdot \|J^s_x\omega\|_{L^2_{xy}}+(\|\omega\|_{L^{\infty}_{xy}}+\|\partial_x \omega\|_{L^{\infty}_{xy}})\|J_x^{s-1}\partial_x\omega\|_{L^2_{xy}}\Big]\\&\leq \|J^s_x\omega\|_{L^2_{xy}}^2\cdot \beta_{\omega}(t)
 \end{split}
 \end{equation*} 
 and 
$$(I)_2 \leq \|J_x^s\omega\|_{L^2_{xy}}^2\|\omega\|_{L^{\infty}_{xy}}\|\partial_x \omega\|_{L^{\infty}_{xy}} \leq \|J_x^s\omega\|_{L^2_{xy}}^2\beta_{\omega}(t)$$ 
so $(I)\lesssim \|J_x^s\omega\|_{L^2_{xy}}^2 \beta_\omega(t)$.  

Now, $(II)=\int J^s_x(u_{k}^2\partial_x\omega)J^s_x\omega=\int[J^s_x(u_{k}^2\partial_x\omega)-u_{k}^2J^s_x\partial_x \omega]J^s_x\omega+\int u_{k}^2J^s_x\partial_x\omega J^s_x\omega$ and we denote $(II)_1=\int[J^s_x(u_{k}^2\partial_x\omega)-u_{k}^2J^s_x\partial_x \omega]J^s_x\omega$ and $(II)_2=\int u_{k}^2J^s_x\partial_x\omega J^s_x\omega$. For the first term we have by the Kato-Ponce commutator estimate
\begin{equation*}
\begin{split} 
(II)_1&\lesssim \|J^s_x\omega\|_{L^2_{xy}}\cdot \|J^s_x(u_{k}^2\partial_x \omega)-u_{k}^2J^s_x\partial_x\omega \|_{L^2_{xy}}\\&\lesssim \|J^s_x\omega \|_{L^2_{xy}}\|u_{k}\|_{L^{\infty}_{xy}}\Big[\|\partial_x \omega \|_{L^{\infty}_{xy}}\|J^s_x u_{k} \|_{L^2_{xy}}+(\|u_{k}\|_{L^{\infty}_{xy}}+\|\partial_x u_{k}\|_{L^{\infty}_{xy}})\|J^{s-1}_x\partial_x \omega\|_{L^2_{xy}}]\\
&\lesssim \|J^s_x\omega\|_{L^2_{xy}}^2\beta_{u_{k}}(t)+\|J^s_x\omega\|_{L^2_{xy}}\|J^s_xu_{k}\|_{L^2_{xy}}\|\partial_x\omega\|_{L^{\infty}_{xy}}\|u_{k}\|_{L^{\infty}_{xy}}
\end{split} 
\end{equation*} 
Also, we have 
\begin{equation*}
\begin{split} 
(II)_2 \lesssim \|J^s_x\omega \|_{L^2_{xy}}^2\|u_{k}\|_{L^{\infty}}\|\partial_x u_{k}\|_{L^{\infty}}\lesssim \|J^s_x\omega \|^2_{L^2_{xy}}\beta_{u_{k}}(t)
\end{split}
\end{equation*} 
Therefore, $$(II) \lesssim \|J^s_x\omega\|^2_{L^2_{xy}}\beta_{u_{k}}(t)+\|J^s_x\omega\|_{L^2_{xy}}\|J^s_xu_{k}\|_{L^2_{xy}}\|\partial_x\omega\|_{L^{\infty}_{xy}}\|u_{k}\|_{L^{\infty}_{xy}}.$$ 

Let $(III)=\int J^s_x(u_{k}\omega \partial_x u_{k})J^s_x\omega=\int [J^s_x(u_{k}\omega\partial_x u_{k})-u_{k}\omega J^s_x\partial_x u_{k}]J^s_x\omega+\int u_{k}\omega J^s_x\partial_x u_{k}J^s_x\omega$ and denote $(III)_1=\int [J^s_x(u_{k}\omega \partial_x u_{k})-u_{k}\omega J^s_x\partial_x u_{k}]J^s_x\omega$ and $(III)_2=\int u_{k}\omega J^s_x\partial_x u_{k}J^s_x\omega$. We have by the Kato-Ponce commutator estimate
\begin{equation*}
\begin{split} 
(III)_1&\lesssim \|J^s_x \omega \|_{L^2_{xy}}\|J^s_x(u_{k}\omega \partial_x u_{k})-u_{k}\omega J_x^s\partial_x u_{k}\|_{L^2_{xy}}\\ &\lesssim \|J^s_x\omega\|_{L^2_{xy}}\Big[\|J_x^su_{k}\|_{L^2_{xy}}\|\omega\|_{L^{\infty}_{xy}}\|\partial_x u_{k}\|_{L^{\infty}_{xy}}\\&+\|J_x^{s-1}\omega\|_{L^{2}_{xy}}(\|u_{k}\|_{L^{\infty}_{xy}}\|\partial_x u_{k}\|_{L^{\infty}_{xy}}+\|\partial_x u_{k}\|^2_{L^{\infty}_{xy}})\\&+\|J_x^{s-1}\partial_x u_{k}\|_{L^2_{xy}}(\|\omega\|_{L^{\infty}_{xy}}\|\partial_x u_{k}\|_{L^{\infty}_{xy}}+\|\omega\|_{L^{\infty}_{xy}}\| u_{k}\|_{L^{\infty}_{xy}}+\|\partial_x \omega\|_{L^{\infty}_{xy}}\|u_{k}\|_{L^{\infty}_{xy}})\Big]\\
&\lesssim \|J^s_x\omega\|^2_{L^2_{xy}}\beta_{u_{k}}(t)+\|J^s_x\omega\|_{L^2_{xy}}\|J^s_xu_{k}\|_{L^2_{xy}}\|\omega\|_{L^{\infty}_{xy}}\|\partial_x u_{k}\|_{L^{\infty}_{xy}}\\&+\|J^s_x\omega\|_{L^2_{xy}}\|J^s_xu_{k}\|_{L^2_{xy}}\|\omega\|_{L^{\infty}_{xy}}\| u_{k}\|_{L^{\infty}_{xy}}+\|J^s_x\omega\|_{L^2_{xy}}\|J^s_xu_{k}\|_{L^2_{xy}}\|\partial_x \omega\|_{L^{\infty}_{xy}}\|u_{k}\|_{L^{\infty}_{xy}}
\end{split}
\end{equation*} 
Also, $(III)_2\lesssim \|J^s_x\omega\|_{L^2_{xy}}\|J^{s+1}_xu_{k}\|_{L^2_{xy}}\|\omega\|_{L^{\infty}_{xy}}\|u_{k}\|_{L^{\infty}_{xy}}$ and so therefore 
\begin{equation*}
\begin{split} 
(III)&\lesssim \|J^s_x\omega\|^2_{L^2_{xy}}\beta_{u_{k}}(t)+\|J^s_x\omega\|_{L^2_{xy}}\|J^s_xu_{k}\|_{L^2_{xy}}\|\omega\|_{L^{\infty}_{xy}}\|\partial_x u_{k}\|_{L^{\infty}_{xy}}\\&+\|J^s_x\omega\|_{L^2_{xy}}\|J^s_xu_{k}\|_{L^2_{xy}}\|\omega\|_{L^{\infty}_{xy}}\| u_{k}\|_{L^{\infty}_{xy}}+\|J^s_x\omega\|_{L^2_{xy}}\|J^s_xu_{k}\|_{L^2_{xy}}\|\partial_x \omega\|_{L^{\infty}_{xy}}\|u_{k}\|_{L^{\infty}_{xy}}
\\&+ \|J^s_x\omega\|_{L^2_{xy}}\|J^{s+1}_xu_{k}\|_{L^2_{xy}}\|\omega\|_{L^{\infty}_{xy}}\|u_{k}\|_{L^{\infty}_{xy}}.
\end{split} 
\end{equation*} 
 Again, $(IV)=\int J^s_x(u_{k}\omega \partial_x \omega)J^s_x\omega=\int [J^s_x(u_{k}\omega \partial_x \omega) - u_{k}\omega J^s_x\partial_x \omega]J^s_x \omega +\int u_{k}\omega J^s_x\partial_x \omega J^s_x \omega$ and we denote $(IV)_1=\int [J^s_x(u_{k}\omega \partial_x \omega) - u_{k}\omega J^s_x\partial_x \omega]J^s_x \omega$ and $(IV)_2=\int u_{k}\omega J^s_x\partial_x \omega J^s_x \omega$.  We have by the Kato-Ponce commutator estimate
 \begin{equation*}
 \begin{split} 
 (IV)_1 &\lesssim \|J^s_x\omega\|_{L^2_{xy}}\|J^s_x(u_{k}\omega \partial_x \omega)-u_{k}\omega J^s_x\partial_x \omega\|_{L^2_{xy}}\\& \lesssim \|J^s_x\omega\|_{L^2_{xy}}\|\partial_x \omega\|_{L^{\infty}_{xy}}\Big[\|J^s_x\omega\|_{L^2_{xy}}(\|u_{k}\|_{L^{\infty}_{xy}}+\|\partial_x u_{k}\|_{L^{\infty}_{xy}})+\|J^s_x u_{k}\|_{L^2_{xy}}\|\omega\|_{L^{\infty}_{xy}}\Big]\\&+\|J^s_x\omega\|_{L^2_{xy}}\Big[\|u_{k}\|_{L^{\infty}_{xy}}(\|\omega\|_{L^{\infty}_{xy}}+\| \partial_x\omega\|_{L^{\infty}_{xy}})+\|\partial_x u_{k}\|_{L^{\infty}_{xy}}\|\omega\|_{L^{\infty}_{xy}})\|J^{s-1}_x\partial_x \omega\|_{L^{2}_{xy}}\\&\lesssim \|J^s_x \omega \|^2_{L^{2}_{xy}}(\|u_{k}\|_{L^{\infty}_{xy}}+\|\partial_x u_{k}\|_{L^{\infty}_{xy}})(\| \partial_x\omega\|_{L^{\infty}_{xy}}+\|\omega\|_{L^{\infty}_{xy}})\\&+\|J^s_x \omega\|_{L^{2}_{xy}}\|J^s_xu_{k}\|_{L^{2}_{xy}}\|\omega\|_{L^{\infty}_{xy}}\|\partial_x \omega\|_{L^{\infty}_{xy}}.
 \end{split}
 \end{equation*}

 Also, $(IV)_2\lesssim \|J_x^s\omega\|^2_{L^2_{xy}}(\|\partial_xu_{k}\|_{L^{\infty}_{xy}}\|\omega\|_{L^{\infty}_{xy}}+\|\partial_x\omega\|_{L^{\infty}_{xy}}\|u_{k}\|_{L^{\infty}_{xy}})$ and so therefore 
 \begin{equation*}
  \begin{split} 
 (IV)& \lesssim \|J^s_x \omega \|^2_{L^{2}_{xy}}(\|u_{k}\|_{L^{\infty}_{xy}}+\|\partial_x u_{k}\|_{L^{\infty}_{xy}})(\| \partial_x\omega\|_{L^{\infty}_{xy}}+\|\omega\|_{L^{\infty}_{xy}})\\&+\|J^s_x \omega\|_{L^{2}_{xy}}\|J^s_xu_{k}\|_{L^{2}_{xy}}\|\omega\|_{L^{\infty}_{xy}}\|\partial_x \omega\|_{L^{\infty}_{xy}} .
 \end{split} 
 \end{equation*}
 
 Again, $(V)=\int J^s_x(\omega^2\partial_x u_{k})J^s_x\omega = \int [J^s_x(\omega^2\partial_x u_{k})-\omega^2 J^s_x\partial_x u_{k}]J^s_x\omega+\int \omega^2 J^s_x\partial_x u_{k} J^s_x\omega$ and we denote $(V)_1= \int [J^s_x(\omega^2\partial_x u_{k})-\omega^2 J^s_x\partial_x u_{k}]J^s_x\omega$ and $(V)_2=\int \omega^2 J^s_x\partial_x u_{k} J^s_x\omega$. We have by the Kato-Ponce commutator estimate
 \begin{equation*} 
 \begin{split} 
 (V)_1 &\lesssim \|J^s_x \omega\|_{L^2_{xy}}\|J^s_x (\omega^2\partial_x u_{k})-\omega^2J^s_x\partial_x u_{k}\|_{L^2_{xy}} \\& \lesssim \|J^s_x \omega\|_{L^2_{xy}}\|\omega\|_{L^{\infty}_{xy}}\Big[\|\partial_x u_{k}\|_{L^{\infty}_{xy}}\|J^s_x\omega\|_{L^2_{xy}}+(\|\omega\|_{L^{\infty}_{xy}}+\|\partial_x \omega\|_{L^{\infty}_{xy}})\|J^{s-1}_x\partial_x u_{k}\|_{L^{2}_{xy}}\Big] \\& \lesssim \|J^s_x \omega \|^2_{L^{2}_{xy}}\|\omega\|_{L^{\infty}_{xy}}\|\partial_x u_{k}\|_{L^{\infty}_{xy}}+\|J^s_x \omega\|_{L^{2}_{xy}} \|J^s_x u_{k}\|_{L^{2}_{xy}}\|\omega\|_{L^{\infty}_{xy}} (\|\omega\|_{L^{\infty}_{xy}}+\|\partial_x \omega\|_{L^{\infty}_{xy}}).
 \end{split}
 \end{equation*}
 Also, $(V)_2\lesssim \|J^s_x \omega\|_{L^{2}_{xy}}\|J^{s+1}_xu_{k}\|_{L^{2}_{xy}}\|\omega\|_{L^{\infty}_{xy}}^2$ and so therefore 
  \begin{equation*} 
 \begin{split} 
 (V)& \lesssim \|J^s_x\omega\|^2_{L^{2}_{xy}}(\beta_{u_{k}}(t)+\beta_{\omega}(t))\\&+\|J^s_x\omega\|_{L^{2}_{xy}}\Big(\|J^s_x u_{k}\|_{L^{2}_{xy}}\|\omega\|_{L^{\infty}_{xy}}(\|\omega\|_{L^{\infty}_{xy}}+\|\partial_x \omega\|_{L^{\infty}_{xy}})+\|J^{s+1}_x u_{k}\|_{L^{2}_{xy}}\|\omega\|_{L^{\infty}_{xy}}^2\Big).  \end{split}
 \end{equation*}
 Now, putting together all the terms we get that  
 \begin{equation}\label{eqdif3x}
 \begin{split}
  \frac{d}{dt}\|J^s_x\omega\|^2_{L^{2}_{xy}}&\lesssim \big(\|J^s_x \omega\|^2_{L^2_{xy}}\big)\big(\beta_{\omega}(t)+\beta_{u_{k}}(t)\big)\\&+\|J^s_x\omega\|_{L^{2}_{xy}}\|J^{s+1}_x u_{k}\|_{L^{2}_{xy}}(\|\omega\|_{L^{\infty}_{xy}}^2+\|\omega\|_{L^{\infty}_{xy}}\|u_{k}\|_{L^{\infty}_{xy}})\\&+\|J^s_x\omega\|_{L^{2}_{xy}}\|J^{s}_x u_{k}\|_{L^{2}_{xy}}(\|\omega\|_{L^{\infty}_{xy}}\|u_k\|_{L^{\infty}_{xy}}+\|\partial_x\omega\|_{L^{\infty}_{xy}}\|u_k\|_{L^{\infty}_{xy}})\\&+\|J^s_x\omega\|_{L^{2}_{xy}}\|J^{s}_x u_{k}\|_{L^{2}_{xy}}(\|\omega\|_{L^{\infty}_{xy}}\|\partial_x u_k\|_{L^{\infty}_{xy}}+\|\omega\|_{L^{\infty}_{xy}}\|\partial_x \omega\|_{L^{\infty}_{xy}}+ \|\omega\|_{L^{\infty}_{xy}}^2)
  \end{split}
  \end{equation}
  We are using the following variant of Gr{\"o}nwall's inequality: 
  \begin{lemma}
  If $\alpha(t),\beta(t)$ are two non-negative functions, and $\frac{d}{dt}u(t)\leq u(t)\beta(t)+\alpha(t)$ for all $t\in [0,T]$ then $$u(t)\leq e^{\int_0^t\beta(s)ds}\Big(u(0)+\int_0^t\alpha(s)ds\Big).$$
  \end{lemma}
  By putting $u(t)=\|J_x^s\omega\|_{L^2_{xy}}$, $\beta(t)=\beta_{\omega}(t)+\beta_{u_{k}}(t)\geq0$ and 
  \begin{equation*} 
  \begin{split}
  \alpha(t)&=\|J^{s+1}_x u_{k}\|_{L^{2}_{xy}}(\|\omega\|_{L^{\infty}_{xy}}^2+\|\omega\|_{L^{\infty}_{xy}}\|u_{k}\|_{L^{\infty}_{xy}})\\&+\|J^{s}_x u_{k}\|_{L^{2}_{xy}}(\|\omega\|_{L^{\infty}_{xy}}\|u_k\|_{L^{\infty}_{xy}}+\|\partial_x\omega\|_{L^{\infty}_{xy}}\|u_k\|_{L^{\infty}_{xy}})\\&+\|J^{s}_x u_{k}\|_{L^{2}_{xy}}(\|\omega\|_{L^{\infty}_{xy}}\|\partial_x u_k\|_{L^{\infty}_{xy}}+\|\omega\|_{L^{\infty}_{xy}}\|\partial_x \omega\|_{L^{\infty}_{xy}}+ \|\omega\|_{L^{\infty}_{xy}}^2)\geq 0
  \end{split}
  \end{equation*} 
by applying the lemma to  (\ref{eqdif3x}) together with Cauchy-Schwarz we get 
  \begin{equation*}
 \begin{split}
 \|J^s_x \omega\|_{L^{\infty}_TL^2_{xy}}& \lesssim \mbox{exp}(\frac{1}{2}f_{\omega}(T)^2+\frac{1}{2}f_{u_{k}}(T)^2)\Big[\|J^s_x \omega(0)\|_{L^{2}_{xy}}
 \\&+\|J^s_x\omega\|_{L^{\infty}_TL^{2}_{xy}}\|J^{s+1}_x u_{k}\|_{L^{\infty}_TL^{2}_{xy}}(\|\omega\|_{L^2_TL^{\infty}_{xy}}^2+\|\omega\|_{L^2_TL^{\infty}_{xy}}\|u_{k}\|_{L^2_TL^{\infty}_{xy}})
 \\&+\|J^s_x\omega\|_{L^{\infty}_TL^{2}_{xy}}\|J^{s}_x u_{k}\|_{L^{\infty}_TL^{2}_{xy}}(\|\omega\|_{L^2_TL^{\infty}_{xy}}+\|\partial_x\omega\|_{L^2_TL^{\infty}_{xy}})\|u_k\|_{L^2_TL^{\infty}_{xy}}
 \\&+\|J^s_x\omega\|_{L^{\infty}_TL^{2}_{xy}}\|J^{s}_x u_{k}\|_{L^{\infty}_TL^{2}_{xy}}\|\omega\|_{L^2_TL^{\infty}_{xy}}(\|\partial_x u_k\|_{L^2_TL^{\infty}_{xy}}+\|\partial_x \omega\|_{L^2_TL^{\infty}_{xy}})
 \\&+\|J^s_x\omega\|_{L^{\infty}_TL^{2}_{xy}}\|J^{s}_x u_{k}\|_{L^{\infty}_TL^{2}_{xy}} \|\omega\|_{L^2_TL^{\infty}_{xy}}^2\Big].
 \end{split} 
 \end{equation*}

\item[b)]
We apply $J_y^s$ to (\ref{eqomega3T}) and then we multiply by $J_y^s\omega$, in order to get 
  \begin{equation*}
  \begin{split}
  \frac{d}{dt}\|J_y^s\omega\|^2_{L^2}=&\int J_y^s(\omega^2\partial_x \omega)J_y^s\omega +3\int J_y^s(u_{k}^2\partial_x \omega)J^s_y\omega+3\int J_y^s(u_{k}\omega\partial_xu_{k})J_y^s\omega\\&-3\int J_y^s(u_{k}\omega\partial_x \omega)J_y^s\omega -3\int J_y^s(\omega^2\partial_x u_{k})J_y^s\omega
  \end{split} 
  \end{equation*} 
  and we will analyze each term in the sum.  
  
  We have $(I)=\int J_y^s(\omega^2\partial_x \omega)J_y^s\omega$, $(II)=\int J_y^s(u_{k}^2\partial_x \omega)J^s_y\omega$, $(III)=\int J_y^s(u_{k}\omega\partial_x u_{k})J_y^s\omega$, $(IV)=\int J_y^s(u_{k}\omega\partial_x \omega)J_y^s\omega$ and $(V)=\int J_y^s(\omega^2\partial_x u_{k})J_y^s\omega$. 
  
  For $(I)=\int J_y^s(\omega^2\partial_x \omega)J_y^s\omega=\int[J_y^s(\omega^2\partial_x \omega)-\omega^2J_y^s\partial_x \omega]J_y^s\omega+\int\omega^2J_y^s\partial_x \omega J_y^s\omega$, and we will denote $(I)_1=\int[J_y^s(\omega^2\partial_x \omega)-\omega^2J_y^s\partial_x \omega]J_y^s\omega$ and $(I)_2=\int\omega^2J_y^s\partial_x \omega J_y^s\omega$. For the first one, we have by the Kato-Ponce commutator estimate
  \begin{equation*}
  \begin{split} 
  (I)_1 &\leq \|J_y^s \omega\|_{L^2_{xy}}\|J^s_y(\omega^2\partial_x \omega)-\omega^2 J_y^s\partial_x \omega\|_{L^2_{xy}}\\
  &\leq \|J^s_y\omega\|_{L^2_{xy}} \|\omega\|_{L^{\infty}_{xy}}\Big[\|\partial_x \omega\|_{L^{\infty}_{xy}}\cdot \|J^s_y\omega\|_{L^2_{xy}}+(\|\omega\|_{L^{\infty}_{xy}}+\|\partial_y \omega\|_{L^{\infty}_{xy}})\|J_y^{s-1}\partial_x\omega\|_{L^2_{xy}}\Big]\\&\leq \|J^s_y\omega\|_{L^2_{xy}}\Big[\|\partial_x \omega\|_{L^{\infty}_{xy}}\cdot \|J^s_y\omega\|_{L^2_{xy}} \|\omega\|_{L^{\infty}_{xy}}\\&+(\|\omega\|^2_{L^{\infty}_{xy}}+\|\omega\|_{L^{\infty}}\|\partial_y \omega\|_{L^{\infty}_{xy}})(\|J_y^{s}\omega\|_{L^2_{xy}}+\|J_x^{s}\omega\|_{L^2_{xy}})\Big]\\&\leq (\|J^s_y\omega\|_{L^2_{xy}}^2+\|J^s_y\omega\|_{L^2_{xy}}\|J^s_x\omega\|_{L^2_{xy}})\cdot \beta_{\omega}(t)
 \end{split}
 \end{equation*} 
 and 
$$(I)_2 \leq \|J_y^s\omega\|_{L^2_{xy}}^2\|\omega\|_{L^{\infty}_{xy}}\|\partial_x \omega\|_{L^{\infty}_{xy}} \leq \|J_y^s\omega\|_{L^2_{xy}}^2\beta_{\omega}(t)$$ 
so $(I)\lesssim (\|J_y^s\omega\|_{L^2_{xy}}^2+\|J_y^s\omega\|_{L^2_{xy}}\|J_x^s\omega\|_{L^2_{xy}}) \beta_\omega(t)$.  

Now, $(II)=\int J^s_y(u_{k}^2\partial_x\omega)J^s_y\omega=\int[J^s_y(u_{k}^2\partial_x\omega)-u_{k}^2J^s_y\partial_x \omega]J^s_y\omega+\int u_{k}^2J^s_y\partial_x\omega J^s_y\omega$ and we denote $(II)_1=\int[J^s_y(u_{k}^2\partial_x\omega)-u_{k}^2J^s_y\partial_x \omega]J^s_y\omega$ and $(II)_2=\int u_{k}^2J^s_y\partial_x\omega J^s_y\omega$. For the first term we have by the Kato-Ponce commutator estimate
\begin{equation*}
\begin{split} 
(II)_1&\lesssim \|J^s_y\omega\|_{L^2_{xy}}\cdot \|J^s_y(u_{k}^2\partial_x \omega)-u_{k}^2J^s_y\partial_x\omega \|_{L^2_{xy}}\\&\lesssim \|J^s_y\omega \|_{L^2_{xy}}\|u_{k}\|_{L^{\infty}_{xy}}\Big[\|\partial_x \omega \|_{L^{\infty}_{xy}}\|J^s_y u_{k} \|_{L^2_{xy}}+(\|u_{k}\|_{L^{\infty}_{xy}}+\|\partial_y u_{k}\|_{L^{\infty}})\|J^{s-1}_x\partial_x \omega\|_{L^2_{xy}}\Big]\\
&\lesssim \|J^s_y\omega\|_{L^2_{xy}}^2\beta_{u_{k}}(t)+\|J^s_y\omega\|_{L^2_{xy}}\|J^s_yu_{k}\|_{L^2_{xy}}\|\partial_y\omega\|_{L^{\infty}_{xy}}\|u_{k}\|_{L^{\infty}_{xy}}\\&+\|J^s_x\omega\|_{L^2_{xy}}\|J^s_y\omega\|_{L^2_{xy}}\|\partial_y u_k\|_{L^{\infty}_{xy}}\|u_{k}\|_{L^{\infty}_{xy}}
\end{split} 
\end{equation*} 
where here we used that $\|J^{s-1}_x\partial_x \omega\|_{L^2_{xy}}\leq |J^s_y\omega \|_{L^2_{xy}}+|J^s_x\omega \|_{L^2_{xy}}$. Also, we have 
\begin{equation*}
\begin{split} 
(II)_2 \lesssim \|J^s_y\omega \|_{L^2_{xy}}^2\|u_{k}\|_{L^{\infty}}\|\partial_x u_{k}\|_{L^{\infty}}\lesssim \|J^s_y\omega \|^2_{L^2_{xy}}\beta_{u_{k}}(t)
\end{split}
\end{equation*} 
Therefore, 
\begin{equation*}
\begin{split} 
(II)&\lesssim \|J^s_y\omega\|_{L^2_{xy}}^2\beta_{u_{k}}(t)+\|J^s_y\omega\|_{L^2_{xy}}\|J^s_yu_{k}\|_{L^2_{xy}}\|\partial_y\omega\|_{L^{\infty}_{xy}}\|u_{k}\|_{L^{\infty}_{xy}}\\&+\|J^s_x\omega\|_{L^2_{xy}}\|J^s_y\omega\|_{L^2_{xy}}\|\partial_y u_k\|_{L^{\infty}_{xy}}\|u_{k}\|_{L^{\infty}_{xy}}
\end{split} 
\end{equation*} 
Let $(III)=\int J^s_y(u_{k}\omega \partial_x u_{k})J^s_y\omega=\int [J^s_y(u_{k}\omega\partial_x u_{k})-u_{k}\omega J^s_y\partial_x u_{k}]J^s_y\omega+\int u_{k}\omega J^s_y\partial_x u_{k}J^s_y\omega$ and denote $(III)_1=\int [J^s_y(u_{k}\omega \partial_x u_{k})-u_{k}\omega J^s_y\partial_x u_{k}]J^s_y\omega$ and $(III)_2=\int u_{k}\omega J^s_y\partial_x u_{k}J^s_y\omega$. We have by the Kato-Ponce commutator estimate
\begin{equation*}
\begin{split} 
(III)_1&\lesssim \|J^s_y \omega \|_{L^2_{xy}}\|J^s_y(u_{k}\omega \partial_x u_{k})-u_{k}\omega J_y^s\partial_x u_{k}\|_{L^2_{xy}}
\\ &\lesssim \|J^s_y\omega\|_{L^2_{xy}}\Big[\|J_y^su_{k}\|_{L^2_{xy}}\|\omega\|_{L^{\infty}_{xy}}\|\partial_x u_{k}\|_{L^{\infty}_{xy}}
\\&+\|J_y^{s-1}\omega\|_{L^{2}_{xy}}(\|\partial_y u_{k}\|_{L^{\infty}_{xy}}\|\partial_x u_{k}\|_{L^{\infty}_{xy}}+\|u_{k}\|_{L^{\infty}_{xy}}\|\partial_x u_k\|_{L^{\infty}_{xy}})
\\&+\|J_y^{s-1}\partial_x u_{k}\|_{L^2_{xy}}(\|\omega\|_{L^{\infty}_{xy}}\|\partial_y u_{k}\|_{L^{\infty}_{xy}}+\|u_{k}\|_{L^{\infty}_{xy}}\|\partial_y\omega\|_{L^{\infty}_{xy}})
\\&+\|J_y^s\omega\|_{L^2_{xy}}\|u_{k}\|_{L^{\infty}_{xy}}\|\partial_x u_{k}\|_{L^{\infty}_{xy}}\Big]\\
&\lesssim \|J^s_y\omega\|^2_{L^2_{xy}}\beta_{u_{k}}(t)+\|J^s_y\omega\|_{L^2_{xy}}\|J^s_x u_{k}\|_{L^2_{xy}}(\|\omega\|_{L^{\infty}_{xy}}\|\partial_y u_{k}\|_{L^{\infty}_{xy}}+\|u_{k}\|_{L^{\infty}_{xy}}\|\partial_y\omega\|_{L^{\infty}_{xy}})
\\&+\|J^s_y\omega\|_{L^2_{xy}}\|J^s_yu_{k}\|_{L^2_{xy}}\Big[\|\omega\|_{L^{\infty}_{xy}}(\|\partial_x u_{k}\|_{L^{\infty}_{xy}}+\|\|\partial_y u_{k}\|_{L^{\infty}_{xy}})+\|u_{k}\|_{L^{\infty}_{xy}}\|\partial_y\omega\|_{L^{\infty}_{xy}}\Big].
\end{split}
\end{equation*} 
Also, $$(III)_2\lesssim \|J^s_y\omega\|_{L^2_{xy}}\|J^{s+1}_y u_{k}\|_{L^2_{xy}}\|\omega\|_{L^{\infty}_{xy}}\|u_{k}\|_{L^{\infty}_{xy}}+\|J^s_y\omega\|_{L^2_{xy}}\|J^{s}_x u_{k}\|_{L^2_{xy}}\|\omega\|_{L^{\infty}_{xy}}\|u_{k}\|_{L^{\infty}_{xy}}$$ and so therefore 
\begin{equation*}
\begin{split} 
(III)&\lesssim \|J^s_y\omega\|^2_{L^2_{xy}}\beta_{u_{k}}(t)+\|J^s_y\omega\|_{L^2_{xy}}\|J^{s+1}_yu_{k}\|_{L^2_{xy}}\|\omega\|_{L^{\infty}_{xy}}\|u_{k}\|_{L^{\infty}_{xy}}
\\&+\|J^s_y\omega\|_{L^2_{xy}}\|J^s_yu_{k}\|_{L^2_{xy}}(\|\omega\|_{L^{\infty}_{xy}}\|\partial_x u_{k}\|_{L^{\infty}_{xy}}+\|\omega\|_{L^{\infty}_{xy}}\|\partial_y u_{k}\|_{L^{\infty}_{xy}}+\|u_{k}\|_{L^{\infty}_{xy}}\|\partial_y\omega\|_{L^{\infty}_{xy}}) 
\\&+\|J^s_y\omega\|_{L^2_{xy}}\|J^s_x u_{k}\|_{L^2_{xy}}(\|\omega\|_{L^{\infty}_{xy}}\|\partial_y u_{k}\|_{L^{\infty}_{xy}}+\|u_{k}\|_{L^{\infty}_{xy}}\|\partial_y\omega\|_{L^{\infty}_{xy}}+\|\omega\|_{L^{\infty}_{xy}}\|u_{k}\|_{L^{\infty}_{xy}}).
\end{split}
\end{equation*}
 Again, $(IV)=\int J^s_y(u_{k}\omega \partial_x \omega)J^s_y\omega=\int [J^s_y(u_{k}\omega \partial_x \omega) - u_{k}\omega J^s_y\partial_x \omega]J^s_y \omega +\int u_{k}\omega J^s_y\partial_x \omega J^s_y \omega$ and we denote $(IV)_1=\int [J^s_y(u_{k}\omega \partial_x \omega) - u_{k}\omega J^s_y\partial_x \omega]J^s_y \omega$ and $(IV)_2=\int u_{k}\omega J^s_y\partial_x \omega J^s_y \omega$.  We have by the Kato-Ponce commutator estimate
 \begin{equation*}
 \begin{split} 
 (IV)_1 &\lesssim \|J^s_y\omega\|_{L^2_{xy}}\|J^s_y(u_{k}\omega \partial_x \omega)-u_{k}\omega J^s_y\partial_x \omega\|_{L^2_{xy}}
 \\& \lesssim \|J^s_y\omega\|_{L^2_{xy}}\Big[\|\partial_x \omega\|_{L^{\infty}_{xy}}\|J^{s-1}_y\omega\|_{L^2_{xy}}(\|u_{k}\|_{L^{\infty}_{xy}}+\|\partial_y u_{k}\|_{L^{\infty}_{xy}})
 \\&+(\|u_{k}\|_{L^{\infty}_{xy}}\|\omega\|_{L^{\infty}_{xy}}+\|u_{k}\|_{L^{\infty}_{xy}}\| \partial_y\omega\|_{L^{\infty}_{xy}}+\|\partial_y u_{k}\|_{L^{\infty}_{xy}}\|\omega\|_{L^{\infty}_{xy}})\|J^{s-1}_y\partial_x \omega\|_{L^{2}_{xy}}
 \\&+\|\partial_x\omega\|_{L^{\infty}_{xy}}\|J^s_y u_{k}\|_{L^2_{xy}}\|\omega\|_{L^{\infty}_{xy}}\Big]
 \\&\lesssim \|J^s_y\omega \|^2_{L^{2}_{xy}}(\beta_{u_{k}}(t)+\beta_{\omega}(t))+\|J^s_y \omega\|_{L^{2}_{xy}}\|J^s_yu_{k}\|_{L^{2}_{xy}}\|\omega\|_{L^{\infty}_{xy}}\|\partial_x\omega\|_{L^{\infty}_{xy}}
 \\&+\|J^s_y \omega\|_{L^{2}_{xy}}\|J^s_x \omega\|_{L^{2}_{xy}}(\|u_{k}\|_{L^{\infty}_{xy}}\|\omega\|_{L^{\infty}_{xy}}+\|u_{k}\|_{L^{\infty}_{xy}}\| \partial_y\omega\|_{L^{\infty}_{xy}}+\|\partial_y u_{k}\|_{L^{\infty}_{xy}}\|\omega\|_{L^{\infty}_{xy}}).
 \end{split}
 \end{equation*}
 Also, $(IV)_2\lesssim \|J_y^s\omega\|^2_{L^2_{xy}}(\|\partial_xu_{k}\|_{L^{\infty}_{xy}}\|\omega\|_{L^{\infty}_{xy}}+\|\partial_x\omega\|_{L^{\infty}_{xy}}\|u_{k}\|_{L^{\infty}_{xy}})\lesssim \|J^s_y\omega \|^2_{L^2_{xy}}(\beta_{u_{k}}(t)+\beta_{\omega}(t))$ and so therefore 
 \begin{equation*}
 \begin{split} 
 (IV)& \lesssim  \|J^s_y\omega \|^2_{L^{2}_{xy}}(\beta_{u_{k}}(t)+\beta_{\omega}(t))+\|J^s_y \omega\|_{L^{2}_{xy}}\|J^s_yu_{k}\|_{L^{2}_{xy}}\|\omega\|_{L^{\infty}_{xy}}\|\partial_x\omega\|_{L^{\infty}_{xy}}
 \\&+\|J^s_y \omega\|_{L^{2}_{xy}}\|J^s_x \omega\|_{L^{2}_{xy}}(\|u_{k}\|_{L^{\infty}_{xy}}\|\omega\|_{L^{\infty}_{xy}}+\|u_{k}\|_{L^{\infty}_{xy}}\| \partial_y\omega\|_{L^{\infty}_{xy}}+\|\partial_y u_{k}\|_{L^{\infty}_{xy}}\|\omega\|_{L^{\infty}_{xy}}).
 \end{split} 
 \end{equation*}
 
 Again, $(V)=\int J^s_y(\omega^2\partial_x u_{k})J^s_y\omega = \int [J^s_y(\omega^2\partial_x u_{k})-\omega^2 J^s_y\partial_x u_{k}]J^s_y\omega+\int \omega^2 J^s_y\partial_x u_{k} J^s_y\omega$ and we denote $(V)_1= \int [J^s_y(\omega^2\partial_x u_{k})-\omega^2 J^s_y\partial_x u_{k}]J^s_y\omega$ and $(V)_2=\int \omega^2 J^s_y\partial_x u_{k} J^s_y\omega$. We have by the Kato-Ponce commutator estimate
 \begin{equation*} 
 \begin{split} 
 (V)_1 &\lesssim \|J^s_y \omega\|_{L^2_{xy}}\|J^s_y (\omega^2\partial_x u_{k})-\omega^2J^s_y\partial_x u_{k}\|_{L^2_{xy}} \\& \lesssim \|J^s_y\omega\|_{L^2_{xy}}\|\omega\|_{L^{\infty}_{xy}}\Big[\|\partial_x u_{k}\|_{L^{\infty}_{xy}}\|J^s_y\omega\|_{L^2_{xy}}+(\|\omega\|_{L^{\infty}_{xy}}+\|\partial_y \omega\|_{L^{\infty}_{xy}})\|J^{s-1}_y\partial_x u_{k}\|_{L^{2}_{xy}}\Big] \\& \lesssim \|J^s_y \omega \|^2_{L^{2}_{xy}}(\beta_{u_{k}}(t)+\beta_{\omega}(t))+\|J^s_y \omega\|_{L^{2}_{xy}}( \|J^s_y u_{k}\|_{L^{2}_{xy}}+ \|J^s_x u_{k}\|_{L^{2}_{xy}}) \beta_{\omega}(t).
 \end{split}
 \end{equation*}
 Also, $$(V)_2\lesssim \|J^s_y \omega\|_{L^{2}_{xy}}\|J^{s+1}_y u_{k}\|_{L^{2}_{xy}}\|\omega\|_{L^{\infty}_{xy}}^2+\|J^s_y \omega\|_{L^{2}_{xy}}\|J^{s}_x u_{k}\|_{L^{2}_{xy}}\|\omega\|_{L^{\infty}_{xy}}^2$$ and so therefore 
 \begin{equation*}
 \begin{split}
 (V)& \lesssim \|J^s_y\omega\|^2_{L^{2}_{xy}}(\beta_{u_{k}}(t)+\beta_{\omega}(t))+\|J^s_y\omega\|_{L^{2}_{xy}}\|J^s_y u_{k}\|_{L^{2}_{xy}}\beta_{\omega}(t)\\&+\|J^s_y\omega\|_{L^{2}_{xy}}\|J^s_x u_{k}\|_{L^{2}_{xy}}\beta_{\omega}(t)+\|J^s_y\omega\|_{L^{2}_{xy}}\|J^{s+1}_x u_{k}\|_{L^{2}_{xy}}\|\omega\|_{L^{\infty}_{xy}}^2.
 \end{split}
 \end{equation*} 
 We make the following notation: $$a(\omega,u_k)=\|\omega\|_{L^{\infty}_{xy}}(\|\partial_x u_{k}\|_{L^{\infty}_{xy}}+\|\partial_y u_{k}\|_{L^{\infty}_{xy}})+\|u_{k}\|_{L^{\infty}_{xy}}\|\partial_y\omega\|_{L^{\infty}_{xy}}+\beta_{\omega}(t),$$
 $$b(\omega,u_k)=\|\omega\|_{L^{\infty}_{xy}}(\|\partial_y u_{k}\|_{L^{\infty}_{xy}}+\| u_{k}\|_{L^{\infty}_{xy}})+\|u_{k}\|_{L^{\infty}_{xy}}\|\partial_y\omega\|_{L^{\infty}_{xy}}+\beta_{\omega}(t),$$
 $$c(\omega,u_k)=\|\omega\|_{L^{\infty}_{xy}}(\|\partial_y u_{k}\|_{L^{\infty}_{xy}}+\| u_{k}\|_{L^{\infty}_{xy}})+\|u_{k}\|_{L^{\infty}_{xy}}\|\partial_y\omega\|_{L^{\infty}_{xy}}+\|u_{k}\|_{L^{\infty}_{xy}}\|\partial_y u_{k}\|_{L^{\infty}_{xy}}+\beta_{\omega}(t).$$
 Now, putting together all the terms we get that  
 \begin{equation}\label{eqdif3y}
 \begin{split}
  \frac{d}{dt}\|J^s_y\omega\|^2_{L^{2}_{xy}}&\lesssim \|J^s_y \omega\|^2_{L^2_{xy}}\big(\beta_{\omega}(t)+\beta_{u_{k}}(t)\big)
  \\&+\|J^s_y\omega\|_{L^{2}_{xy}}\|J^s_y u_{k}\|_{L^{2}_{xy}}a(\omega,u_k)
  \\& +\|J^s_y\omega\|_{L^{2}_{xy}}\|J^s_x u_{k}\|_{L^{2}_{xy}}b(\omega,u_k)
  \\&+\|J^s_y\omega\|_{L^{2}_{xy}}\|J^s_x \omega\|_{L^{2}_{xy}}c(\omega,u_k)
  \\&+\|J^s_y\omega\|_{L^{2}_{xy}}\|J^{s+1}_x u_{k}\|_{L^{2}_{xy}}(\|\omega\|_{L^{\infty}_{xy}}^2+\|\omega\|_{L^{\infty}_{xy}}\|u_{k}\|_{L^{\infty}_{xy}}).
  \end{split}
  \end{equation}
 Using the variant of Gr{\"o}nwall's inequality from part a) and applying it to (\ref{eqdif3y}) with $u(t)=\|J^s_y\omega\|_{L^{2}_{xy}}$, $\beta(t)=\beta_{\omega}(t)+\beta_{u_{k}}(t)\geq0$ and 
 \begin{equation*}
 \begin{split}
 \alpha(t)&=\|J^s_y\omega\|_{L^{2}_{xy}}\|J^s_y u_{k}\|_{L^{2}_{xy}}a(\omega,u_k)
  \\& +\|J^s_y\omega\|_{L^{2}_{xy}}\|J^s_x u_{k}\|_{L^{2}_{xy}}b(\omega,u_k)
  \\&+\|J^s_y\omega\|_{L^{2}_{xy}}\|J^s_x \omega\|_{L^{2}_{xy}}c(\omega,u_k)
  \\&+\|J^s_y\omega\|_{L^{2}_{xy}}\|J^{s+1}_x u_{k}\|_{L^{2}_{xy}}(\|\omega\|_{L^{\infty}_{xy}}^2+\|\omega\|_{L^{\infty}_{xy}}\|u_{k}\|_{L^{\infty}_{xy}}).
  \end{split} 
 \end{equation*}
we obtain
  \begin{equation*}
 \begin{split}
 \|J^s_y \omega\|_{L^{\infty}_TL^2_{xy}}& \lesssim \mbox{exp}(\frac{1}{2}f_{\omega}(T)^2+\frac{1}{2}f_{u_{k}}(T)^2)\Big[\|J^s_y \omega(0)\|_{L^{2}_{xy}}
 \\&+\|J^s_y\omega\|_{L^{\infty}_TL^{2}_{xy}}\|J^s_y u_{k}\|_{L^{2}_{xy}}\|\omega\|_{L^2_TL^{\infty}_{xy}}(\|\partial_x u_{k}\|_{L^2_TL^{\infty}_{xy}}+\|\partial_y u_{k}\|_{L^2_TL^{\infty}_{xy}})
 \\&+\|J^s_y\omega\|_{L^{\infty}_TL^{2}_{xy}}\|J^s_y u_{k}\|_{L^{2}_{xy}}\Big(\|u_{k}\|_{L^2_TL^{\infty}_{xy}}\|\partial_y\omega\|_{L^2_TL^{\infty}_{xy}}+f_{\omega}(T)^2\Big)
 \\&+\|J^s_y\omega\|_{L^{\infty}_TL^{2}_{xy}}\|J^s_x u_{k}\|_{L^{2}_{xy}}\|\omega\|_{L^2_TL^{\infty}_{xy}}(\| u_{k}\|_{L^2_TL^{\infty}_{xy}}+\|\partial_y u_{k}\|_{L^2_TL^{\infty}_{xy}})
\\&+\|J^s_y\omega\|_{L^{\infty}_TL^{2}_{xy}}\|J^s_x u_{k}\|_{L^{2}_{xy}}\Big(\|u_{k}\|_{L^2_TL^{\infty}_{xy}}\|\partial_y\omega\|_{L^2_TL^{\infty}_{xy}}+f_{\omega}(T)^2\Big)
 \\&+\|J^s_y\omega\|_{L^{\infty}_TL^{2}_{xy}}\|J^s_x \omega\|_{L^{2}_{xy}}\|\omega\|_{L^2_TL^{\infty}_{xy}}(\| u_{k}\|_{L^2_TL^{\infty}_{xy}}+\|\partial_y u_{k}\|_{L^2_TL^{\infty}_{xy}})
  \\&+\|J^s_y\omega\|_{L^{\infty}_TL^{2}_{xy}}\|J^s_x \omega\|_{L^{2}_{xy}}\Big(\|u_{k}\|_{L^2_TL^{\infty}_{xy}}\|\partial_y\omega\|_{L^2_TL^{\infty}_{xy}}+f_{\omega}(T)^2\Big)
  \\&+\|J^s_y\omega\|_{L^{\infty}_TL^{2}_{xy}}\|J^s_x \omega\|_{L^{2}_{xy}}\|u_{k}\|_{L^2_TL^{\infty}_{xy}}\|\partial_y u_k\|_{L^2_TL^{\infty}_{xy}}
 \\&+\|J^s_y\omega\|_{L^{\infty}_TL^{2}_{xy}}\|J^{s+1}_y u_{k}\|_{L^{\infty}_{T}L^{2}_{xy}}(\|\omega\|^2_{L^{\infty}_{T}L^{2}_{xy}}+\|\omega\|_{L^{\infty}_{T}L^{2}_{xy}}\|u_{k}\|_{L^{\infty}_{T}L^{2}_{xy}})\Big].
  \end{split} 
 \end{equation*}
\end{itemize} 
\end{proof} 

\begin{lemma} \label{lasttermestimate}
For $p\leq s$, we have the following estimates:  
\begin{itemize} 
\item[(a)]
\begin{equation*}
\begin{split}
\|J_x^p[\omega(u_k^2+u_ku_{k'}+u_{k'}^2)]\|_{L^{1}_TL^2_{xy}}&\lesssim \|J_x^p\omega\|_{L^{\infty}_TL^2_{xy}}(\|u_k\|_{L^{2}_TL^{\infty}_{xy}}+\|u_{k'}\|_{L^{2}_TL^{\infty}_{xy}})\\&+\|\omega\|_{L^{2}_TL^{\infty}_{xy}}\|\phi\|_{H^{s,s}}(\|u_k\|_{L^{2}_TL^{\infty}_{xy}}+\|u_{k'}\|_{L^{2}_TL^{\infty}_{xy}}). 
\end{split} 
\end{equation*}
\item[(b)]
\begin{equation*}
\begin{split}
\|J_y^p[\omega(u_k^2+u_ku_{k'}+u_{k'}^2)]\|_{L^{1}_TL^2_{xy}}&\lesssim \|J_y^p\omega\|_{L^{\infty}_TL^2_{xy}}(\|u_k\|_{L^{2}_TL^{\infty}_{xy}}+\|u_{k'}\|_{L^{2}_TL^{\infty}_{xy}})\\&+\|\omega\|_{L^{2}_TL^{\infty}_{xy}}\|\phi\|_{H^{s,s}}(\|u_k\|_{L^{2}_TL^{\infty}_{xy}}+\|u_{k'}\|_{L^{2}_TL^{\infty}_{xy}}). 
\end{split} 
\end{equation*}
\end{itemize}
\end{lemma} 
\begin{proof} 
By using \ref{katoponce} part (c), we get that 
\begin{equation} \label{Jxomega} 
\begin{split}
\|J_x^p[\omega(u_k^2+u_ku_{k'}+u_{k'}^2)]\|_{L^2_{xy}}&\lesssim \|J_x^p\omega\|_{L^{\infty}_TL^2_{xy}}(\|u_k^2+u_ku_{k'}+u_{k'}^2\|_{L^{\infty}_{xy}}\\&\|\omega\|_{L^{\infty}_{xy}}\|J_x^p(u_k^2+u_ku_{k'}+u_{k'}^2)\|_{L^2_{xy}}.
\end{split} 
\end{equation}
Observe that  $\|u_k^2+u_ku_{k'}+u_{k'}^2\|_{L^{\infty}_{xy}} \lesssim \|u_k\|^2_{L^{\infty}_{xy}}+\|u_{k'}\|^2_{L^{\infty}_{xy}}$. Also, by \ref{katoponce} part (c) again, we have $$\|J_x^p(u_k^2+u_ku_{k'}+u_{k'}^2)\|_{L^2_{xy}}\lesssim (\|J_x^p u_k\|_{L^2_{xy}}+\|J_x^p u_{k'}\|_{L^2_{xy}})(\|u_k\|_{L^{\infty}}+\|u_{k'}^2\|_{L^{\infty}_{xy}}). $$

By \ref{Jx}, we get $\|J_x^p u_k\|_{L^2_{xy}}+\|J_x^p u_{k'}\|_{L^2_{xy}}\lesssim \|J_x^p \phi_k\|_{L^2_{xy}}+\|J_x^p \phi_{k'}\|_{L^2_{xy}} \lesssim \|\phi\|_{H^{s,s}}.$ Combining all the above observation together with \ref{Jxomega}, we get 
\begin{equation*}
\begin{split}
\|J_x^p[\omega(u_k^2+u_ku_{k'}+u_{k'}^2)]\|_{L^2_{xy}}&\lesssim \|J_x^p\omega\|_{L^2_{xy}}(\|u_k\|_{L^{\infty}_{xy}}+\|u_{k'}\|_{L^{\infty}_{xy}})\\&+\|\omega\|_{L^{2}_TL^{\infty}_{xy}}\|\phi\|_{H^{s,s}}(\|u_k\|_{L^{\infty}_{xy}}+\|u_{k'}\|_{L^{\infty}_{xy}}). 
\end{split} 
\end{equation*}
Integrating both sides from $0$ to $T$ and applying Cauchy-Schwarz inequality, we obtain the conclusion of the lemma for $J_x$. The proof for $J_y$ goes the same way. 
\end{proof} 

\begin{lemma} \label{fu3RT}
Suppose $u_k$ satisfies the IVP (\ref{eqmkp3RT}) with initial data $\phi_k=P^k_{(3)}\phi.$ We have $\|\omega\|_{L^{2}_TL^{\infty}_{xy}}\lesssim k^{(-1)-}$, $\|\partial_x \omega\|_{L^{2}_TL^{\infty}_{xy}}\lesssim k^{0-}$ and $\|\partial_y \omega\|_{L^{2}_TL^{\infty}_{xy}}\lesssim k^{0-}$ as $k \rightarrow \infty$. In particular, $f_{\omega}(T)\lesssim k^{0-}$ as $k \rightarrow \infty.$
\end{lemma} 
\begin{proof} 
Take $\delta <\frac{s-2}{2}.$ By the linear estimate in Proposition \ref{linearestimate3RT} applied to \ref{eqomega3T}, 
$$ \|\omega\|_{L^2_TL^{\infty}_{xy}}\lesssim \|J_x^{1+\delta}\omega\|_{L^{\infty}_TL^{2}_{xy}}+\|J_x^{-1}J_y^{1+\delta}\omega\|_{L^{\infty}_TL^{2}_{xy}}+\|J_x^{1+\delta}J_y^{\delta}[\omega(u_k^2+u_ku_{k'}+u_{k'}^2)]\|_{L^1_TL^{2}_{xy}}.$$ 

From \ref{Jxomega3RT} and \ref{Jyomega3RT} we have $\|J_x^{1+\delta}\omega\|_{L^{\infty}_TL^{2}_{xy}} \lesssim k^{1+\delta-s}h_{\phi}^{(3)}(k)^{1-\frac{1+\delta}{s}}$, together with $$\|J_x^{-1}J^{1+\delta}_y \omega\|_{L^{\infty}_TL^{2}_{xy}}\lesssim \|J^{1+\delta}_y \omega\|_{L^{\infty}_TL^{2}_{xy}}\lesssim k^{1+\delta-s}h_{\phi}^{(3)}(k)^{1-\frac{1+\delta}{s}}.$$
For the last term, we observe 
\begin{equation*}
\begin{split} 
\|J_x^{1+\delta}J_y^{\delta}[\omega(u_k^2+u_ku_{k'}+u_{k'}^2)]\|_{L^{1}_TL^{2}_{xy}}&\lesssim \|J_x^{1+2\delta}[\omega(u_k^2+u_ku_{k'}+u_{k'}^2)]\|_{L^{1}_TL^{2}_{xy}}\\&+\|J_y^{1+2\delta}[\omega(u_k^2+u_ku_{k'}+u_{k'}^2)]\|_{L^{1}_TL^{2}_{xy}}
\end{split} 
\end{equation*} 
By Lemma \ref{lasttermestimate} we get that 
\begin{equation*}
\begin{split}
\|J_x^{1+2\delta}[\omega(u_k^2+u_ku_{k'}+u_{k'}^2)]\|_{L^{1}_TL^2_{xy}}&\lesssim \|J_x^{1+2\delta}\omega\|_{L^{\infty}_TL^2_{xy}}(\|u_k\|_{L^{2}_TL^{\infty}_{xy}}+\|u_{k'}\|_{L^{2}_TL^{\infty}_{xy}})\\&+\|\omega\|_{L^{2}_TL^{\infty}_{xy}}\|\phi\|_{H^{s,s}}(\|u_k\|_{L^{2}_TL^{\infty}_{xy}}+\|u_{k'}\|_{L^{2}_TL^{\infty}_{xy}})
\end{split} 
\end{equation*}
and 
\begin{equation*}
\begin{split}
\|J_y^{1+2\delta}[\omega(u_k^2+u_ku_{k'}+u_{k'}^2)]\|_{L^{1}_TL^2_{xy}}&\lesssim \|J_y^{1+2\delta}\omega\|_{L^{\infty}_TL^2_{xy}}(\|u_k\|_{L^{2}_TL^{\infty}_{xy}}+\|u_{k'}\|_{L^{2}_TL^{\infty}_{xy}})\\&+\|\omega\|_{L^{2}_TL^{\infty}_{xy}}\|\phi\|_{H^{s,s}}(\|u_k\|_{L^{2}_TL^{\infty}_{xy}}+\|u_{k'}\|_{L^{2}_TL^{\infty}_{xy}}). 
\end{split} 
\end{equation*}
By \ref{Jxomega3RT} and \ref{Jyomega3RT} we have $\|J_x^{1+2\delta}\omega\|_{L^{\infty}_TL^{2}_{xy}} \lesssim k^{1+2\delta-s}h_{\phi}^{(3)}(k)^{1-\frac{1+2\delta}{s}}$ and $$\|J_y^{1+2\delta}\omega\|_{L^{\infty}_TL^{2}_{xy}} \lesssim k^{1+2\delta-s}h_{\phi}^{(3)}(k)^{1-\frac{1+2\delta}{s}}.$$ 
By combining the previous observations, we obtain 
 \begin{equation*} 
 \begin{split} 
  \| \omega\|_{L^2_TL^{\infty}_{xy}}&\lesssim k^{1+2\delta-s}h_{\phi}^{(3)}(k)^{1-\frac{1+2\delta}{s}}\mbox{max}(1,h_{\phi}^{(3)}(k)^{\frac{\delta}{s}})\\&+\|\omega\|_{L^{2}_TL^{\infty}_{xy}}\|\phi\|_{H^{s,s}}(\|u_k\|_{L^{2}_TL^{\infty}_{xy}}+\|u_{k'}\|_{L^{2}_TL^{\infty}_{xy}})
  \end{split}
  \end{equation*} 
Since we consider that $\|\phi\|_{H^{s,s}}$ is small enough, such that $\|\phi\|_{H^{s,s}}(\|u_k\|_{L^{2}_TL^{\infty}_{xy}}+\|u_{k'}\|_{L^{2}_TL^{\infty}_{xy}})\leq \frac{1}{2}$, we get that 
$$\|\omega\|_{L^2_TL^{\infty}_{xy}} \lesssim k^{1+2\delta-s}h_{\phi}^{(3)}(k)^{1-\frac{1+2\delta}{s}}\mbox{max}(1,h_{\phi}^{(3)}(k)^{\frac{\delta}{s}})\rightarrow 0$$ as $k\rightarrow \infty$ since $1+2\delta<s.$

The linear estimate \ref{linearestimate3RT} applied to $\partial_x \omega$ results in 
$$ \|\partial_x \omega\|_{L^2_TL^{\infty}_{xy}}\lesssim \|J_x^{2+\delta}\omega\|_{L^{\infty}_TL^{2}_{xy}}+\|J_y^{1+\delta}\omega\|_{L^{\infty}_TL^{2}_{xy}}+\|J_x^{2+\delta}J_y^{\delta}[\omega(u_k^2+u_ku_{k'}+u_{k'}^2)]\|_{L^1_TL^{2}_{xy}}.$$ 
 and by the same reasoning as above 
 \begin{equation*} 
 \begin{split} 
  \|\partial_x \omega\|_{L^2_TL^{\infty}_{xy}}&\lesssim k^{2+2\delta-s}h_{\phi}^{(3)}(k)^{1-\frac{2+2\delta}{s}}\mbox{max}(1,h_{\phi}^{(3)}(k)^{\frac{\delta}{s}})\\&+\|\omega\|_{L^{2}_TL^{\infty}_{xy}}\|\phi\|_{H^{s,s}}(\|u_k\|_{L^{2}_TL^{\infty}_{xy}}+\|u_{k'}\|_{L^{2}_TL^{\infty}_{xy}})
  \end{split}
  \end{equation*} 
 which, combined with the above fact that $ \|\omega\|_{L^2_TL^{\infty}_{xy}} \lesssim k^{1+2\delta-s}h_{\phi}^{(3)}(k)^{1-\frac{1+2\delta}{s}}\mbox{max}(1,h_{\phi}^{(3)}(k)^{\frac{\delta}{s}})$, for $k$ large enough, it gives us $\|\partial_x \omega\|_{L^{2}_TL^{\infty}_{xy}}\lesssim k^{2+2\delta-s}\rightarrow 0$ as $k \rightarrow \infty$ since $2+2\delta<s.$

Lastly, the linear estimate \ref{linearestimate3RT} applied to $\partial_y \omega$ results in 
$$ \|\partial_y \omega\|_{L^2_TL^{\infty}_{xy}}\lesssim \|J_x^{1+\delta}J_y^1\omega\|_{L^{\infty}_TL^{2}_{xy}}+\|J_y^{2+\delta}\omega\|_{L^{\infty}_TL^{2}_{xy}}+\|J_x^{1+\delta}J_y^{1+\delta}[\omega(u_k^2+u_ku_{k'}+u_{k'}^2)]\|_{L^1_TL^{2}_{xy}}.$$ 
and by the same reasoning as above 
 \begin{equation*} 
 \begin{split} 
  \|\partial_y \omega\|_{L^2_TL^{\infty}_{xy}}&\lesssim k^{2+2\delta-s}h_{\phi}^{(3)}(k)^{1-\frac{2+2\delta}{s}}\mbox{max}(1,h_{\phi}^{(3)}(k)^{\frac{\delta}{s}})\\&+\|\omega\|_{L^{2}_TL^{\infty}_{xy}}\|\phi\|_{H^{s,s}}(\|u_k\|_{L^{2}_TL^{\infty}_{xy}}+\|u_{k'}\|_{L^{2}_TL^{\infty}_{xy}})
  \end{split}
  \end{equation*} 
 which, combined with the above fact that $ \|\omega\|_{L^2_TL^{\infty}_{xy}} \lesssim k^{1+2\delta-s}h_{\phi}^{(3)}(k)^{1-\frac{1+2\delta}{s}}\mbox{max}(1,h_{\phi}^{(3)}(k)^{\frac{\delta}{s}})$, for $k$ large enough, it gives us $\|\partial_x \omega\|_{L^{2}_TL^{\infty}_{xy}}\lesssim k^{2+2\delta-s}\rightarrow 0$ as $k \rightarrow \infty$ since $2+2\delta<s.$
 \end{proof}
 
\begin{lemma} \label{fu3TT}
Suppose $u_k$ satisfies the IVP (\ref{eqmkp3TT}) with initial data $\phi_k=\widetilde{P}^k_{(3)}\phi.$ We have $\|\omega\|_{L^{2}_TL^{\infty}_{xy}}\lesssim k^{(-1)-}$, $\|\partial_x \omega\|_{L^{2}_TL^{\infty}_{xy}}\lesssim k^{0-}$ and $\|\partial_y \omega\|_{L^{2}_TL^{\infty}_{xy}}\lesssim k^{0-}$ as $k \rightarrow \infty$. In particular, $f_{\omega}(T)\lesssim k^{0-}$ as $k \rightarrow \infty.$
\end{lemma} 
\begin{proof} 
Take $\delta <\frac{s-\frac{19}{8}}{2}.$ By the linear estimate in Proposition \ref{linearestimate3TT} applied to \ref{eqomega3T}, 
  $$\|u\|_{L^2_TL^{\infty}_{xy}}\lesssim [\|J^{\frac{11}{8}+\delta}_x u\|_{L^{\infty}_TL^2_{xy}}+\|J_x^{-\frac{5}{8}}J_y^{1+\delta}u\|_{L^{\infty}_TL^2_{xy}}+\|J_x^{\frac{11}{8}+\delta}J_y^{\delta}f\|_{L^{1}_TL^2_{xy}}].$$

From \ref{Jxomega3TT} and \ref{Jyomega3TT} we have $\|J_x^{\frac{11}{8}+\delta}\omega\|_{L^{\infty}_TL^{2}_{xy}} \lesssim k^{\frac{11}{8}+\delta-s}\widetilde{h}_{\phi}^{(3)}(k)^{1-\frac{\frac{11}{8}+\delta}{s}}$, together with $$\|J_x^{-\frac{5}{8}}J^{1+\delta}_y \omega\|_{L^{\infty}_TL^{2}_{xy}}\lesssim \|J^{1+\delta}_y \omega\|_{L^{\infty}_TL^{2}_{xy}}\lesssim k^{1+\delta-s}\widetilde{h}_{\phi}^{(3)}(k)^{1-\frac{1+\delta}{s}}.$$
For the last term, we observe 
\begin{equation*}
\begin{split} 
\|J_x^{\frac{11}{8}+\delta}J_y^{\delta}[\omega(u_k^2+u_ku_{k'}+u_{k'}^2)]\|_{L^{1}_TL^{2}_{xy}}&\lesssim \|J_x^{\frac{11}{8}+2\delta}[\omega(u_k^2+u_ku_{k'}+u_{k'}^2)]\|_{L^{1}_TL^{2}_{xy}}\\&+\|J_y^{\frac{11}{8}+2\delta}[\omega(u_k^2+u_ku_{k'}+u_{k'}^2)]\|_{L^{1}_TL^{2}_{xy}}
\end{split} 
\end{equation*} 
By Lemma \ref{lasttermestimate} we get that 
\begin{equation*}
\begin{split}
\|J_x^{\frac{11}{8}+2\delta}[\omega(u_k^2+u_ku_{k'}+u_{k'}^2)]\|_{L^{1}_TL^2_{xy}}&\lesssim \|J_x^{\frac{11}{8}+2\delta}\omega\|_{L^{\infty}_TL^2_{xy}}(\|u_k\|_{L^{2}_TL^{\infty}_{xy}}+\|u_{k'}\|_{L^{2}_TL^{\infty}_{xy}})\\&+\|\omega\|_{L^{2}_TL^{\infty}_{xy}}\|\phi\|_{H^{s,s}}(\|u_k\|_{L^{2}_TL^{\infty}_{xy}}+\|u_{k'}\|_{L^{2}_TL^{\infty}_{xy}})
\end{split} 
\end{equation*}
and 
\begin{equation*}
\begin{split}
\|J_y^{\frac{11}{8}+2\delta}[\omega(u_k^2+u_ku_{k'}+u_{k'}^2)]\|_{L^{1}_TL^2_{xy}}&\lesssim \|J_y^{\frac{11}{8}+2\delta}\omega\|_{L^{\infty}_TL^2_{xy}}(\|u_k\|_{L^{2}_TL^{\infty}_{xy}}+\|u_{k'}\|_{L^{2}_TL^{\infty}_{xy}})\\&+\|\omega\|_{L^{2}_TL^{\infty}_{xy}}\|\phi\|_{H^{s,s}}(\|u_k\|_{L^{2}_TL^{\infty}_{xy}}+\|u_{k'}\|_{L^{2}_TL^{\infty}_{xy}}). 
\end{split} 
\end{equation*}
By \ref{Jxomega3TT} and \ref{Jyomega3TT} we have $\|J_x^{\frac{11}{8}+2\delta}\omega\|_{L^{\infty}_TL^{2}_{xy}} \lesssim k^{\frac{11}{8}+2\delta-s}\widetilde{h}_{\phi}^{(3)}(k)^{1-\frac{\frac{11}{8}+2\delta}{s}}$ and $$\|J_y^{\frac{11}{8}+\delta}\omega\|_{L^{\infty}_TL^{2}_{xy}} \lesssim k^{\frac{11}{8}+2\delta-s}\widetilde{h}_{\phi}^{(3)}(k)^{1-\frac{\frac{11}{8}+2\delta}{s}}.$$ 
By combining the previous observations, we obtain 
 \begin{equation*} 
 \begin{split} 
  \| \omega\|_{L^2_TL^{\infty}_{xy}}&\lesssim k^{\frac{11}{8}+2\delta-s}\widetilde{h}_{\phi}^{(3)}(k)^{1-\frac{\frac{11}{8}+2\delta}{s}}\mbox{max}(1,\widetilde{h}_{\phi}^{(3)}(k)^{\frac{\delta}{s}})\\&+\|\omega\|_{L^{2}_TL^{\infty}_{xy}}\|\phi\|_{H^{s,s}}(\|u_k\|_{L^{2}_TL^{\infty}_{xy}}+\|u_{k'}\|_{L^{2}_TL^{\infty}_{xy}})
  \end{split}
  \end{equation*} 
Since we consider that $\|\phi\|_{H^{s,s}}$ is small enough, such that $\|\phi\|_{H^{s,s}}(\|u_k\|_{L^{2}_TL^{\infty}_{xy}}+\|u_{k'}\|_{L^{2}_TL^{\infty}_{xy}})\leq \frac{1}{2}$, we get that 
$$\|\omega\|_{L^2_TL^{\infty}_{xy}} \lesssim k^{\frac{11}{8}+2\delta-s}\widetilde{h}_{\phi}^{(3)}(k)^{1-\frac{\frac{11}{8}+2\delta}{s}}\mbox{max}(1,\widetilde{h}_{\phi}^{(3)}(k)^{\frac{\delta}{s}})\rightarrow 0$$ as $k\rightarrow \infty$ since $\frac{11}{8}+2\delta<s.$

The linear estimate \ref{linearestimate3TT} applied to $\partial_x \omega$ results in 
$$ \|\partial_x \omega\|_{L^2_TL^{\infty}_{xy}}\lesssim \|J_x^{\frac{19}{8}+\delta}\omega\|_{L^{\infty}_TL^{2}_{xy}}+\|J_x^{\frac{3}{8}}J_y^{1+\delta}\omega\|_{L^{\infty}_TL^{2}_{xy}}+\|J_x^{\frac{19}{8}+\delta}J_y^{\delta}[\omega(u_k^2+u_ku_{k'}+u_{k'}^2)]\|_{L^1_TL^{2}_{xy}}.$$ 
 and by the same reasoning as above 
 \begin{equation*} 
 \begin{split} 
  \|\partial_x \omega\|_{L^2_TL^{\infty}_{xy}}&\lesssim k^{\frac{19}{8}+2\delta-s}\widetilde{h}_{\phi}^{(3)}(k)^{1-\frac{\frac{19}{8}+2\delta}{s}}\mbox{max}(1,\widetilde{h}_{\phi}^{(3)}(k)^{\frac{\delta}{s}})\\&+\|\omega\|_{L^{2}_TL^{\infty}_{xy}}\|\phi\|_{H^{s,s}}(\|u_k\|_{L^{2}_TL^{\infty}_{xy}}+\|u_{k'}\|_{L^{2}_TL^{\infty}_{xy}})
  \end{split}
  \end{equation*} 
 which, combined with the above fact that $$\|\omega\|_{L^2_TL^{\infty}_{xy}} \lesssim k^{\frac{11}{8}+2\delta-s}\widetilde{h}_{\phi}^{(3)}(k)^{1-\frac{\frac{11}{8}+2\delta}{s}}\mbox{max}(1,\widetilde{h}_{\phi}^{(3)}(k)^{\frac{\delta}{s}}),$$ for $k$ large enough, it gives us $\|\partial_x \omega\|_{L^{2}_TL^{\infty}_{xy}}\lesssim k^{\frac{19}{8}+2\delta-s}\rightarrow 0$ as $k \rightarrow \infty$ since $\frac{19}{8}+2\delta<s.$

Lastly, the linear estimate \ref{linearestimate3TT} applied to $\partial_y \omega$ results in 
$$ \|\partial_y \omega\|_{L^2_TL^{\infty}_{xy}}\lesssim \|J_x^{\frac{11}{8}+\delta}J_y^1\omega\|_{L^{\infty}_TL^{2}_{xy}}+\|J_y^{2+\delta}\omega\|_{L^{\infty}_TL^{2}_{xy}}+\|J_x^{\frac{11}{8}+\delta}J_y^{1+\delta}[\omega(u_k^2+u_ku_{k'}+u_{k'}^2)]\|_{L^1_TL^{2}_{xy}}.$$ 
and by the same reasoning as above 
 \begin{equation*} 
 \begin{split} 
  \|\partial_y \omega\|_{L^2_TL^{\infty}_{xy}}&\lesssim k^{\frac{19}{8}+2\delta-s}\widetilde{h}_{\phi}^{(3)}(k)^{1-\frac{\frac{19}{8}+2\delta}{s}}\mbox{max}(1,\widetilde{h}_{\phi}^{(3)}(k)^{\frac{\delta}{s}})\\&+\|\omega\|_{L^{2}_TL^{\infty}_{xy}}\|\phi\|_{H^{s,s}}(\|u_k\|_{L^{2}_TL^{\infty}_{xy}}+\|u_{k'}\|_{L^{2}_TL^{\infty}_{xy}})
  \end{split}
  \end{equation*} 
 which, combined with the above fact that $$\|\omega\|_{L^2_TL^{\infty}_{xy}} \lesssim k^{\frac{19}{8}+2\delta-s}h_{\phi}^{(3)}(k)^{1-\frac{\frac{19}{8}+2\delta}{s}}\mbox{max}(1,h_{\phi}^{(3)}(k)^{\frac{\delta}{s}}),$$ for $k$ large enough, it gives us $\|\partial_y \omega\|_{L^{2}_TL^{\infty}_{xy}}\lesssim k^{\frac{19}{8}+2\delta-s}\rightarrow 0$ as $k \rightarrow \infty$ since $\frac{19}{8}+2\delta<s.$

\end{proof}

\begin{lemma} \label{fu5RT}
Suppose $u_k$ satisfies the IVP (\ref{eqmkp5RT}) with initial data $\phi_k=P^k_{(5)}\phi.$ We have $\|\omega\|_{L^{2}_TL^{\infty}_{xy}}\lesssim k^{(-1)-}$, $\|\partial_x \omega\|_{L^{2}_TL^{\infty}_{xy}}\lesssim k^{0-}$ and $\|\partial_y \omega\|_{L^{2}_TL^{\infty}_{xy}}\lesssim k^{0-}$ as $k \rightarrow \infty$. In particular, $f_{\omega}(T)\lesssim k^{0-}$ as $k \rightarrow \infty.$
\end{lemma} 
\begin{proof} 
Take $\delta <\frac{s-\frac{5}{2}}{2}.$ By the linear estimate in Proposition \ref{linearestimate5RT} applied to \ref{eqomega5T}, 
$$ \|\omega\|_{L^2_TL^{\infty}_{xy}}\lesssim \|J_x^{\frac{3}{2}+\delta}\omega\|_{L^{\infty}_TL^{2}_{xy}}+\|J_x^{-\frac{3}{2}+\delta}J_y^{1+\delta}\omega\|_{L^{\infty}_TL^{2}_{xy}}+\|J_x^{\frac{1}{2}+\delta}J_y^{\delta}[\omega(u_k^2+u_ku_{k'}+u_{k'}^2)]\|_{L^1_TL^{2}_{xy}}.$$ 

From \ref{Jxomega5RT} and \ref{Jyomega5RT} we have $\|J_x^{\frac{3}{2}+\delta}\omega\|_{L^{\infty}_TL^{2}_{xy}} \lesssim k^{\frac{3}{2}+\delta-s}h_{\phi}^{(5)}(k)^{1-\frac{\frac{3}{2}+\delta}{s}}$, together with $$\|J_x^{-\frac{3}{2}+\delta}J^{1+\delta}_y \omega\|_{L^{\infty}_TL^{2}_{xy}}\lesssim \|J^{1+\delta}_y \omega\|_{L^{\infty}_TL^{2}_{xy}}\lesssim k^{1+\delta-s}h_{\phi}^{(5)}(k)^{1-\frac{1+\delta}{s}}.$$
For the last term, we observe 
\begin{equation*}
\begin{split} 
\|J_x^{\frac{1}{2}+\delta}J_y^{\delta}[\omega(u_k^2+u_ku_{k'}+u_{k'}^2)]\|_{L^{1}_TL^{2}_{xy}}&\lesssim \|J_x^{\frac{1}{2}+2\delta}[\omega(u_k^2+u_ku_{k'}+u_{k'}^2)]\|_{L^{1}_TL^{2}_{xy}}\\&+\|J_y^{\frac{1}{2}+2\delta}[\omega(u_k^2+u_ku_{k'}+u_{k'}^2)]\|_{L^{1}_TL^{2}_{xy}}
\end{split} 
\end{equation*} 
By Lemma \ref{lasttermestimate} we get that 
\begin{equation*}
\begin{split}
\|J_x^{\frac{1}{2}+2\delta}[\omega(u_k^2+u_ku_{k'}+u_{k'}^2)]\|_{L^{1}_TL^2_{xy}}&\lesssim \|J_x^{\frac{1}{2}+2\delta}\omega\|_{L^{\infty}_TL^2_{xy}}(\|u_k\|_{L^{2}_TL^{\infty}_{xy}}+\|u_{k'}\|_{L^{2}_TL^{\infty}_{xy}})\\&+\|\omega\|_{L^{2}_TL^{\infty}_{xy}}\|\phi\|_{H^{s,s}}(\|u_k\|_{L^{2}_TL^{\infty}_{xy}}+\|u_{k'}\|_{L^{2}_TL^{\infty}_{xy}})
\end{split} 
\end{equation*}
and 
\begin{equation*}
\begin{split}
\|J_y^{\frac{1}{2}+2\delta}[\omega(u_k^2+u_ku_{k'}+u_{k'}^2)]\|_{L^{1}_TL^2_{xy}}&\lesssim \|J_y^{\frac{1}{2}+2\delta}\omega\|_{L^{\infty}_TL^2_{xy}}(\|u_k\|_{L^{2}_TL^{\infty}_{xy}}+\|u_{k'}\|_{L^{2}_TL^{\infty}_{xy}})\\&+\|\omega\|_{L^{2}_TL^{\infty}_{xy}}\|\phi\|_{H^{s,s}}(\|u_k\|_{L^{2}_TL^{\infty}_{xy}}+\|u_{k'}\|_{L^{2}_TL^{\infty}_{xy}}). 
\end{split} 
\end{equation*}
By \ref{Jxomega5RT} and \ref{Jyomega5RT} we have $\|J_x^{\frac{1}{2}+2\delta}\omega\|_{L^{\infty}_TL^{2}_{xy}} \lesssim k^{\frac{1}{2}+2\delta-s}h_{\phi}^{(5)}(k)^{1-\frac{\frac{1}{2}+2\delta}{s}}$ and $$\|J_y^{\frac{1}{2}+2\delta}\omega\|_{L^{\infty}_TL^{2}_{xy}} \lesssim k^{\frac{1}{2}+2\delta-s}h_{\phi}^{(5)}(k)^{1-\frac{\frac{1}{2}+2\delta}{s}}.$$
By combining the previous observations, we obtain 
 \begin{equation*} 
 \begin{split} 
  \| \omega\|_{L^2_TL^{\infty}_{xy}}&\lesssim k^{\frac{3}{2}+\delta-s}h_{\phi}^{(3)}(k)^{1-\frac{\frac{1}{2}+2\delta}{s}}\mbox{max}(1,h_{\phi}^{(3)}(k)^{\frac{\delta-1}{s}})\\&+\|\omega\|_{L^{2}_TL^{\infty}_{xy}}\|\phi\|_{H^{s,s}}(\|u_k\|_{L^{2}_TL^{\infty}_{xy}}+\|u_{k'}\|_{L^{2}_TL^{\infty}_{xy}})
  \end{split}
  \end{equation*} 
Since we consider that $\|\phi\|_{H^{s,s}}$ is small enough, such that $\|\phi\|_{H^{s,s}}(\|u_k\|_{L^{2}_TL^{\infty}_{xy}}+\|u_{k'}\|_{L^{2}_TL^{\infty}_{xy}})\leq \frac{1}{2}$, we get that 
$$\|\omega\|_{L^2_TL^{\infty}_{xy}} \lesssim k^{\frac{3}{2}+2\delta-s}h_{\phi}^{(5)}(k)^{1-\frac{\frac{1}{2}+2\delta}{s}}\mbox{max}(1,h_{\phi}^{(5)}(k)^{\frac{\delta-1}{s}})\rightarrow 0$$ as $k\rightarrow \infty$ since $\frac{3}{2}+2\delta<s.$

The linear estimate \ref{linearestimate5RT} applied to $\partial_x \omega$ results in 
$$ \|\partial_x \omega\|_{L^2_TL^{\infty}_{xy}}\lesssim \|J_x^{\frac{5}{2}+\delta}\omega\|_{L^{\infty}_TL^{2}_{xy}}+\|J_x^{-\frac{1}{2}+\delta}J_y^{1+\delta}\omega\|_{L^{\infty}_TL^{2}_{xy}}+\|J_x^{\frac{3}{2}+\delta}J_y^{\delta}[\omega(u_k^2+u_ku_{k'}+u_{k'}^2)]\|_{L^1_TL^{2}_{xy}}.$$ 
 and by the same reasoning as above 
 \begin{equation*} 
 \begin{split} 
  \|\partial_x \omega\|_{L^2_TL^{\infty}_{xy}}&\lesssim k^{\frac{5}{2}+2\delta-s}h_{\phi}^{(5)}(k)^{1-\frac{\frac{3}{2}+2\delta}{s}}\mbox{max}(1,h_{\phi}^{(3)}(k)^{\frac{\delta-1}{s}})\\&+\|\omega\|_{L^{2}_TL^{\infty}_{xy}}\|\phi\|_{H^{s,s}}(\|u_k\|_{L^{2}_TL^{\infty}_{xy}}+\|u_{k'}\|_{L^{2}_TL^{\infty}_{xy}})
  \end{split}
  \end{equation*} 
 which, combined with the above fact that $$\|\omega\|_{L^2_TL^{\infty}_{xy}} \lesssim k^{\frac{3}{2}+2\delta-s}h_{\phi}^{(5)}(k)^{1-\frac{\frac{3}{2}+2\delta}{s}}\mbox{max}(1,h_{\phi}^{(5)}(k)^{\frac{\delta-1}{s}}),$$ for $k$ large enough, it gives us $\|\partial_x \omega\|_{L^{2}_TL^{\infty}_{xy}}\lesssim k^{\frac{5}{2}+2\delta-s}\rightarrow 0$ as $k \rightarrow \infty$ since $\frac{5}{2}+2\delta<s.$

Lastly, the linear estimate \ref{linearestimate5RT} applied to $\partial_y \omega$ results in 
$$ \|\partial_y \omega\|_{L^2_TL^{\infty}_{xy}}\lesssim \|J_x^{\frac{3}{2}+\delta}J_y^1\omega\|_{L^{\infty}_TL^{2}_{xy}}+\|J_y^{2+\delta}\omega\|_{L^{\infty}_TL^{2}_{xy}}+\|J_x^{\frac{1}{2}+\delta}J_y^{\delta}[\omega(u_k^2+u_ku_{k'}+u_{k'}^2)]\|_{L^1_TL^{2}_{xy}}.$$ 
and by the same reasoning as above 
 \begin{equation*} 
 \begin{split} 
  \|\partial_y \omega\|_{L^2_TL^{\infty}_{xy}}&\lesssim k^{\frac{5}{2}+\delta-s}h_{\phi}^{(5)}(k)^{1-\frac{\frac{5}{2}+\delta}{s}}\mbox{max}(1,h_{\phi}^{(5)}(k)^{\frac{\delta-1}{s}})\\&+\|\omega\|_{L^{2}_TL^{\infty}_{xy}}\|\phi\|_{H^{s,s}}(\|u_k\|_{L^{2}_TL^{\infty}_{xy}}+\|u_{k'}\|_{L^{2}_TL^{\infty}_{xy}})
  \end{split}
  \end{equation*} 
 which, combined with the above fact that $$\|\omega\|_{L^2_TL^{\infty}_{xy}} \lesssim k^{\frac{5}{2}+\delta-s}h_{\phi}^{(5)}(k)^{1-\frac{1+2\delta}{s}}\mbox{max}(1,h_{\phi}^{(5)}(k)^{\frac{\delta-1}{s}}),$$ for $k$ large enough, it gives us $\|\partial_x \omega\|_{L^{2}_TL^{\infty}_{xy}}\lesssim k^{\frac{5}{2}+2\delta-s}\rightarrow 0$ as $k \rightarrow \infty$ since $\frac{5}{2}+2\delta<s.$
 
\end{proof}

\begin{corollary} 
We have $\|\omega\|_{H^{s,s}}\rightarrow 0$ as $k \rightarrow \infty$, where $s>2$ for the initial value problem (\ref{eqmkp3RT}), $s>\frac{19}{8}$ for the initial value problem (\ref{eqmkp3TT}) and $s>\frac{5}{2}$ for initial value problem (\ref{eqmkp5RT}).
\end{corollary} 
\begin{proof} 
From ($\ref{eqx}$) and Lemmas \ref{fu3RT}, \ref{fu3TT} and \ref{fu5RT} we get $\|J_x^{s+1}u_k\|_{L^{\infty}_TL^{2}_{xy}}\|\omega\|_{L^{2}_TL^{\infty}_{xy}}\lesssim k^1\cdot k^{(-1)-}=k^{0-}$ and $k^{0-}\rightarrow 0$ as $k\rightarrow \infty$. 
From the Lemmas \ref{fu3RT}, \ref{fu3TT} used in Lemma \ref{JxJy} we obtain $$\|J_x^s\omega\|_{L^{\infty}_TL^{2}_{xy}}\lesssim \mbox{exp}(\frac{1}{2}f_{u_k}(T)^2+\frac{1}{2}f_{\omega}(T)^2)(\|J^s_x\omega(0)\|_{L^{\infty}_TL^{2}_{xy}}+Ck^{0-})\rightarrow 0$$ as $k \rightarrow \infty$, where we used that  $\|J^s_x\omega(0)\|_{L^{\infty}_TL^{2}_{xy}}\rightarrow 0$ as $k\rightarrow \infty$ and the boundedness of $f_{u_k}(T)$ and $f_{\omega}(T)$ by \ref{bound3RT}, \ref{bound3TT} and \ref{bound5RT}.

From ($\ref{eqy}$) and Lemmas \ref{fu3RT}, \ref{fu3TT} and \ref{fu5RT} we get $\|J_y^{s+1}u_k\|_{L^{\infty}_TL^{2}_{xy}}\|\omega\|_{L^{2}_TL^{\infty}_{xy}}\lesssim k^1\cdot k^{(-1)-}=k^{0-}$ and $k^{0-}\rightarrow 0$ as $k\rightarrow \infty$. 
From the Lemmas \ref{fu3RT}, \ref{fu3TT} used in Lemma \ref{JxJy} together with the fact we just proved, $\|J_x^s\omega\|_{L^{\infty}_TL^{2}_{xy}}\rightarrow 0$, we obtain $$\|J_y^s\omega\|_{L^{\infty}_TL^{2}_{xy}}\lesssim \mbox{exp}(\frac{1}{2}f_{u_k}(T)^2+\frac{1}{2}f_{\omega}(T)^2)(\|J^s_y\omega(0)\|_{L^{\infty}_TL^{2}_{xy}}+Ck^{0-})\rightarrow 0$$ as $k \rightarrow \infty$ and the boundedness of $f_{u_k}(T)$ and $f_{\omega}(T)$ by \ref{bound3RT}, \ref{bound3TT} and \ref{bound5RT}. 

Therefore, as $\|J_x^s\omega\|_{L^{\infty}_TL^{2}_{xy}}+\|J_y^s\omega\|_{L^{\infty}_TL^{2}_{xy}}\rightarrow 0$ as $k\rightarrow \infty$, it means that $u\in C([0,T]:H^{s,s}).$
\end{proof}

\section{Continuity of the flow map} 

We assume that $T \in [0,\infty)$ and $\phi^l \rightarrow \phi$ in $H^{s,s}(M\times \mathbb{T})$ as $l\rightarrow \infty.$ We are going to prove that $u^l \rightarrow u$ in $C([-T,T]: H^{s,s}(M\times \mathbb{T}))$ as $l \rightarrow \infty$, where $u^l$ and $u$ are solutions of the the initial value problem $\partial_tu+(-1)^{\frac{d+1}{2}}\partial_x^du-\partial_x^{-1}\partial_y^2u+u^2\partial_x u=0$ corresponding to initial data $\phi^l$ and $\phi$, for $d=3$, $M=\mathbb{R}$, $s>2$, for $d=3$, $M=\mathbb{T}$, $s>\frac{19}{8}$. and for $d=5$, $M=\mathbb{R}$, $s>\frac{5}{2}$.

For $k\geq 1,$ let as before, $\phi^l_k=P^k\phi^l$ and $u^l_k \in C([-T,T]:H^{\infty})$ the corresponding solutions. Denote by $\omega_k=u_k-u$. By the same estimates from Lemma \ref{JxJy}, Lemma \ref{fu3RT}, Lemma \ref{fu3TT}, Lemma \ref{fu5RT} applied to $\omega_k$ we get 
$$\|u_k-u\|_{H^{s,s}}\lesssim \mbox{exp}(\frac{1}{2}f_{\omega_k}(T)^2+\frac{1}{2}f_{u_k}(T)^2)(\|\phi_k-\phi\|_{H^{s,s}}+C(T,\|\phi_k\|_{H^{s,s}},\|\phi\|_{H^{s,s}})k^{0-}).$$
By the same reasoning, we have that 
$$\|u_k^l-u^l\|_{H^{s,s}}\lesssim \mbox{exp}(\frac{1}{2}f_{\omega_k^l}(T)^2+\frac{1}{2}f_{u_k^l}(T)^2)(\|\phi_k^l-\phi^l\|_{H^{s,s}}+C(T,\|\phi_k^l\|_{H^{s,s}},\|\phi^l\|_{H^{s,s}})k^{0-}).$$
Now, denote $\omega_k^l=u_k^l-u_k.$ By the same estimates from Lemma \ref{JxJy}, Lemma \ref{fu3RT}, Lemma \ref{fu3TT}, Lemma \ref{fu5RT} applied to $\omega_k^l$ 
$$\|u_k^l-u_k\|_{H^{s,s}}\lesssim \mbox{exp}(\frac{1}{2}f_{\omega_k^l}(T)^2+\frac{1}{2}f_{u_k^l}(T)^2)(\|\phi_k^l-\phi_k\|_{H^{s,s}}+C(T,\|\phi_k^l\|_{H^{s,s}},\|\phi_k\|_{H^{s,s}})k^{0-}).$$

By the boundedness of $f_{u_k}(T), f_{u_k^l}(T), f_{\omega_k}(T)$ and $f_{\omega_k^l}(T)$ by \ref{bound3RT}, by \ref{bound3TT} and by \ref{bound5RT} and the triangle inequality, we get 
\begin{equation*}
\begin{split}
\|u^l-u\|_{H^{s,s}}&\leq \|u_k-u\|_{H^{s,s}}+\|u_k^l-u_k\|_{H^{s,s}}+\|u^l_k-u^l\|_{H^{s,s}}\\& \lesssim \|\phi_k-\phi\|_{H^{s,s}}+\|\phi_k^l-\phi_k\|_{H^{s,s}}+\|\phi_k^l-\phi^l\|_{H^{s,s}}\\&+C(T,\|\phi\|_{H^{s,s}},\|\phi_k\|_{H^{s,s}},\|\phi^l\|_{H^{s,s}},\|\phi_k^l\|_{H^{s,s}})k^{0-} 
\end{split}
\end{equation*}
which, by letting $k \rightarrow \infty$, we get $\|u^l-u\|_{H^{s,s}}\lesssim \|\phi^l-\phi\|_{H^{s,s}}$ and proves the continuity of the flow map.


\begin{thebibliography}{99}
 \bibitem{ablowitz} 
  M.J. Ablowitz, A. Fokas: On the inverse scattering and direct linearizing transforms for the Kadomtsev-Petviashvili equation. \textit{Phys. Lett. A} 94(2), 67-70, 1983. 
 
  \bibitem{abramyan} 
L. A. Abramyan, Y. A. Stepanyants, The structure of two-dimensional solitons in media with anomalously small dispersion, \textit{Sov. Phys. JETP.}, 61(5): 963-966, 1985. 
 
 \bibitem{bona}
 J. L. Bona, R. Smith, The initial-value problem for the Korteweg-de Vries equation, \textit{Philos. Trans. Roy. Soc. London Ser. A} 278 (1287): 555-601, 1975. 
 
\bibitem{bourgain1} 
J.Bourgain, On the Cauchy problem for the Kadomstev-Petviashvili Equation, \textit{Geom. Funct. Anal.} 3 (4): 315-341, 1993. 

\bibitem{chen1} 
W. G. Chen, J. F. Li, C. X. Miao, On the low regularity of the fifth order Kadomtsev-Petviashvili I equation, \textit{J. Diff. Eqns.} 245(11): 3433-3469, 2008.

\bibitem{esfahani} 
A. Esfahani, Instability of solitary waves of the generalized higher-order KP equation, \textit{Nonlinearity} 24(3): 833-846, 2011.

\bibitem{turitsyn}
G. E. Fal'kovitch, S. K. Turitsyn, Stability of magnetoelastic solitons and self-focusing of sound in antiferromagnet, \textit{Soviet Phys. JETP} 62: 146-152, 1985.

\bibitem {guo1}
B. L. Guo, Z. H. Huo, S. M. Fang, Low regularity for the fifth order Kadomtsev- Petviashvili-I type equation, \textit{J. Diff. Eqns.} 263(9): 5696-5726, 2017.

\bibitem{guopeng}
Z. Guo, L. Peng, B. Wang, On the local regularity of the KP-I equation in anisotropic Sobolev space. \textit{J. Math. Pures Appl. (9)} 94(4): 414-432, 2010. 

\bibitem{hardac1} 
M. Hardac, Well-posedness for the Kadomtsev-Petviashvili II equation and generalisations. \textit{Transactions of the American Mathematical Society}, 360(12):6555-6572, 2008.

\bibitem{hardac2} 
M. Hadac, S. Herr, and H. Koch. Well-posedness and scattering for the KP-II equation in a critical space. \textit{Annales de l'Institut Henri Poincar{\'e} (C) Non Linear Analysis}, 26(3):917-941, 2009.

\bibitem{ionescu1}
A. D. Ionescu, C. E. Kenig, Local and global well-posedness of periodic KP-I equations. In \textit{Mathematical aspects of nonlinear dispersive equations}, volume 163 of \textit{Ann. of Math. Stud.}, pages 181-211. Princeton Univ. Press, NJ, 2007.

\bibitem{ionescu2} 
A. D. Ionescu, C. E. Kenig, D. Tataru, Global well-posedness of the KP-I initial- value problem in the energy space, \textit{Invent. Math.} 173 (2): 265-304, 2008.

\bibitem{iorio}
R. J. Iorio, W. V. L. Nunes, On equations of KP-type, \textit{Proc. Roy. Soc. Edinburgh Sect. A} 4(128): 725-743, 1998.

\bibitem{isaza} 
Pedro Isaza, Jorge Mej{\'i}a. Local and global Cauchy problems for the Kadomtsev-Petviashvili (KP-II) equation in Sobolev spaces of negative indices. \textit{Communications in Partial Differential Equations}, 26(5-6):1027-1054, 2001.

\bibitem{kadomtsev} 
B. B. Kadomtsev, V. I. Petviashvili, On the stability of solitary waves in weakly dispersive media, \textit{Soviet. Phys. Dokl.} (15): 539-541, 1970.

\bibitem{karpman} 
V. I. Karpman, V. Yu. Belashov, Dynamics of two dimensional solitons in weakly dispersive media, \textit{Phys. Lett. A.} 154:131-139, 1991.

\bibitem{kato}
T. Kato, G. Ponce, Commutator estimates and the Euler and Navier-Stokes equations, \textit{Comm. Pure Appl. Math.} 41(7): 891-907, 1988.

\bibitem{kenig1} 
C. E. Kenig, On the local and global well-posedness theory for the KP-I equation, \textit{Ann. Inst. H. Poincar{\'e} Anal. Non Lin{\'e}aire}, 21(6): 827-838, 2004.

\bibitem{lili} 
J. LI, X. Li, Well-posedness for the fifth order KP-II initial data problem in $H^{s,0}(\mathbb{R}\times\mathbb{T})$,  \textit{J. Differential Equations} 262 (3): 2196-2230, 2017. 

\bibitem{li1} 
J. F. Li, J. Xiao, Well-posedness of the fifth order Kadomtsev-Petviashvili I equation in anisotropic Sobolev spaces with nonnegative indices, \textit{J. Math. Pures Appl. (9)} 90(4): 338-352, 2008.

\bibitem{li2} 
Y. Li, W. Yan, Y. Zhang, Global well-posedness of the Cauchy problem for a fifth-order KP-I equation in anisotropic Sobolev spaces. 2017. doi: https://arxiv.org:1712.09334. 

\bibitem{liu}
Y. Liu, Blow up and instability of solitary-wave solutions to a generalized Kadomtsev-Petviashvili equation. \textit{Trans. Amer. Math. Soc.} 353 (1): 191-208, 2001. 

\bibitem{molinet1} 
L. Molinet, J. C. Saut, N. Tzvetkov, Well-posedness and ill-posedness results for the Kadomtsev-Petviashvili-I equation, \textit{Duke Math. J.} 115(2): 353-384, 2002.

\bibitem{molinet2} 
L. Molinet, J. C. Saut, N. Tzvetkov, Global well-posedness for the KP-I equation, \textit{Math. Ann.} 324(2): 255-275, 2002.

\bibitem{molinet3} 
L. Molinet, J. C. Saut, N. Tzvetkov, Correction: Global well-posedness for the KP-I equation [\textit{Math. Ann.}(324), 255-275, 2002] \textit{Math. Ann.} 328(4): 707-710, 2004.

\bibitem{robert} 
T. Robert, Global well-posedness of partially periodic KP-I equation in the energy space and application. \textit{Ann. Inst. H. Poincar{\'e} Anal. Non Lin{\'e}aire} 35 (7): 1773-1826, 2019.

\bibitem{robert1}
T. Robert, On the Cauchy problem for the periodic fifth-order KP-I equation, \textit{Differential Integral Equations}, 32(11-12):679-704, 2019. 

\bibitem{saut} 
J.-C. Saut, Remarks on the Generalized Kadomtsev-Petviashvili Equations. \textit{Indiana University Mathematics Journal}, 42(3): 1011-1026, 1993. 

\bibitem{saut1} 
J.-C. Saut, N. Tzvetkov, The Cauchy problem for the fifth order KP equations, \textit{J. Math. Pures Appl. (9)} 79 (4): 307-338, 2000. 

\bibitem{saut2} 
J.-C. Saut, N. Tzvetkov, On periodic KP-I type equations, \textit{Comm. Math. Phys.} 221(3): 451-476, 2001.

\bibitem{takaoka}
H. Takaoka, N. Tzvetkov. On the local regularity of the Kadomtsev-Petviashvili-II equation. \textit{International Mathematics Research Notices}, 2001(2):77-114, 2001.  
  
\bibitem{yan} 
W. Yan, Y. S. Li, J. H. Huang, J. Q. Duan, The Cauchy problem for two dimensional generalized Kadomtsev-Petviashvili-I equation in anisotropic Sobolev spaces. \textit{Anal. Appl. (Singap.)}, 18(3):469-522, 2020. 
  
\bibitem{zhang}
Y. Zhang. Local well-posedness of KP-I initial value problem on torus in the Besov space. \textit{Communications in Partial Differential Equations}, 41(2):256-281, 2016. 
  
 
\end{thebibliography}
\end{document}